\documentclass[11pt]{amsart}
\usepackage{amsthm,amsmath}
\usepackage{amssymb}
\usepackage{url}
\usepackage{color}
\usepackage[margin=1in, footskip=0.25in]{geometry}
\newtheorem{theorem}{Theorem}[section]
\newtheorem{corollary}[theorem]{Corollary}
\newtheorem{lemma}[theorem]{Lemma}
\newtheorem{remark}[theorem]{Remark}
\newtheorem{definition}[theorem]{Definition}
\newtheorem{hypothesis}[theorem]{Hypothesis}

\newtheorem{proposition}[theorem]{Proposition}

\newcommand{\Soc}{\mathrm{Soc}}
\newcommand{\G}{\mathrm{G}}
\newcommand{\Sz}{\mathrm{Sz}}

\newcommand{\Aut}{\mathrm{Aut}}
\newcommand{\Sp}{\mathrm{Sp}}
\newcommand{\PSL}{\mathrm{PSL}}
\newcommand{\A}{\mathrm{A}}
\newcommand{\C}{\mathrm{C}}
\newcommand{\D}{\mathrm{D}}
\newcommand{\Cos}{\mathrm{Cos}}
\newcommand{\Stab}{\mathrm{Stab}}

\newcommand{\Ree}{\mathrm{Ree}}
\newcommand{\PSp}{\mathrm{PSp}}
\newcommand{\R}{\mathrm{R}}
\newcommand{\m}{\mathrm{mult}}
\newcommand{\ICF}{\mathrm{InsolCF}}
\newcommand{\ppd}{\mathrm{ppd}}

\newcommand{\PGO}{\mathrm{PGO}}
\newcommand{\POmega}{\mathrm{P\Omega}}
\newcommand{\Sy}{\mathrm{S}}
\newcommand{\PGL}{\mathrm{PGL}}
\newcommand{\PSU}{\mathrm{PSU}}
\newcommand{\Out}{\mathrm{Out}}
\newcommand{\GU}{\mathrm{GU}}
\newcommand{\F}{\mathrm{F}}
\newcommand{\GL}{\mathrm{GL}}
\newcommand{\SL}{\mathrm{SL}}
\newcommand{\PGammaSp}{\mathrm{P\Gamma Sp}}

\newcommand{\Rad}{\mathrm{Rad}}
\newcommand{\Syl}{\mathrm{Syl}}

\newcommand{\Span}{\mathrm{Span}}

\newcommand{\GO}{\mathrm{GO}}
\newcommand{\PGSp}{\mathrm{PGSp}}

\newcommand{\E}{\mathrm{E}}
\begin{document}
\title{Vertex-primitive \(s\)-arc-transitive digraphs of symplectic groups}

%\author{Lei Chen\(^*\) \quad \quad Michael Giudici \quad\quad Cheryl E. Praeger\\ \\
%Department of Mathematics and Statistics\\
%The University of Western Australia\\
%35 Stirling Highway, Perth WA 6009\\
%Australia\\ \\
%\small Lei.Chen@research.uwa.edu.au\\%\thanks{The paper is part of the first author's PhD thesis and he acknowledges the support of travel grant PO19400093.}\thanks{\(^*\) The corresponding author}\\ 
%\small Michael.Giudici@uwa.edu.au\\
%\small Cheryl.Praeger@uwa.edu.au}

\thanks{$^*$Corresponding author. Email address: Lei.Chen@research.uwa.edu.au}
\thanks{The paper is part of the first author's PhD thesis under the supervision of the second and third authores, and he acknowledges the support of travel grant PO19400093.}

\author{Lei Chen\(^*\)}
\address{Department of Mathematics and Statistics, The University of Western Australia, Perth WA 6009}
\email{Lei.Chen@research.uwa.edu.au}

\author{Michael Giudici}
\address[Michael Giudici]{Department of Mathematics and Statistics, The University of Western Australia, Perth WA 6009}
\email{Michael.Giudici@uwa.edu.au}

\author{Cheryl E. Praeger}
\address{Department of Mathematics and Statistics, The University of Western Australia, Perth WA 6009}
\email{Cheryl.Praeger@uwa.edu.au}
\maketitle
\begin{abstract}
 A digraph is $s$-arc-transitive if its automorphism group is transitive on directed paths with $s$ edges, that is, on $s$-arcs. Although infinite families of finite $s$-arc transitive digraphs of arbitrary valency were constructed by the third author in 1989, existence of a vertex-primitive $2$-arc-transitive digraph was not known until an infinite family was constructed by the second author with Li and Xia in 2017. This led to a conjecture by the second author and Xia in 2018 that, for a finite vertex-primitive $s$-arc-transitive digraph, $s$ is at most $2$, together with their proof that it is sufficient to prove the conjecture for digraphs with an almost simple group of automorphisms. This paper confirms the conjecture for finite symplectic groups.    

\end{abstract}

\section{Introduction}

 The exploration of finite \(s\)-arc-transitive digraphs dates back to at least 1989 when the third author \cite{Praeger} constructed, for each \(s\) and \(d\), an infinite family of finite vertex-imprimitive \(s\)-arc-transitive digraphs of valency \(d\) that are not \((s+1)\)-arc-transitive. In that paper she asked if there exists a finite vertex-primitive \(s\)-arc-transitive digraph with $s\geqslant2$, and whether there
is an upper bound on \(s\) for such digraphs.  The situation for digraphs is very different from that for graphs, where for a finite \(s\)-arc transitive graph of valency at least 3, the parameter $s$ is at most 8, see  \cite{W}, and vertex-primitive $8$-arc-transitive examples exist.

The first question in \cite{Praeger} was answered positively nearly 30 years later in \cite{infinite} by the second author with Li and Xia,  and  shortly afterwards the second question about determining an upper bound on \(s\) was reduced \cite{quasi} to the case where the graph automorphism group is almost simple, and in that paper it was conjectured that \(s\leqslant2\). Since then this conjecture has been settled for several infinite families of almost simple groups: by the second author with Li and Xia for groups with socle \(\PSL_{n}(q)\) in \cite{linear}, by the authors with Pan, Wu and Yin \cite{alternating, alternating2} for alternating and symmetric groups, and the authors \cite{reeandsuzuki} for the small Ree groups and the Suzuki groups. The aim of this paper is to prove the conjecture for the finite symplectic groups:

\begin{theorem}\label{maintheorem}
    Let \(\Gamma\) be a \(G\)-vertex-primitive \((G,s)\)-arc-transitive digraph such that \(\Soc(G)=\PSp_{2n}(q)'\) with \(n\geqslant 2\). Then \(s\leqslant2\).
\end{theorem}

The condition \(n\geqslant2\) is imposed since \(\PSp_{2}(q)\cong\PSL_{2}(q)\), and the conjecture has already been proved for this case in \cite{linear}. As we noted above,  the first infinite family of examples of \(G\)-vertex-primitive \((G,2)\)-arc-transitive digraphs was given in \cite{infinite}; these examples have \(G=\PSL_{3}(p^{2})\) and vertex stabiliser \(\A_{6} \), where \(p\geqslant7\) is a prime such that \(p\equiv \pm2\pmod{5}\). In seeking the smallest example of such a digraph, Yin, Feng and Xia \cite{small} identifed the first member of this family, namely with \(p=7\) and \(3,075,815,460 > 3\cdot 10^9\) vertices, as being the smallest vertex-primitive \(2\)-arc-transitive digraph.

\medskip
\noindent
\textit{Summary of concepts and notation above and our proof strategy:}\quad 
A transitive action of a group  \(G\) is primitive if and only if the stabiliser of a point is a maximal subgroup of  \(G\). Also \(G\) is said to be an \emph{almost simple group} if  \(G\) has a unique \emph{minimal normal subgroup} \(T\) and \(T\) is a nonabelian simple group. This implies (identifying $T$ with the group of inner automorphisms of $T$) that \(T\triangleleft G\leqslant \Aut(T)\). 

\medskip

Thus we will investigate all the maximal subgroups of all the almost simple groups $G$ with socle \(\PSp_{2n}(q)'\). Aschbacher \cite{asch} developed a useful tool to classify the maximal subgroups of classical groups: they can be divided  into nine classes \(\mathcal{C}_{1},\ldots,\mathcal{C}_{9}\). The maximal subgroups in classes \(\mathcal{C}_{1}-\mathcal{C}_{8}\) are the so called \emph{maximal subgroups of geometric type}. Each has some geometric interpretation, and the specific subgroup structure can be found in \cite{kleidman}. The maximal subgroups of \(G\) in class \(\mathcal{C}_{9}\) are almost simple with socle \(S\), say, and the preimage of \(S\) in \(\Sp_{2n}(q)\) is absolutely irreducible and not writable over any proper subfield of \(\mathbf{F}_{q}\).
For the special case where $2n=4$ the automorphism group \(\Aut(\PSp_{4}(2^{f}))\nleqslant\PGammaSp_{4}(2^{f})\). Here, in addition to the families \(\mathcal{C}_{1}-\mathcal{C}_{9}\) of maximal subgroups, if \(G\nleqslant \PGammaSp_{4}(2^{f})\) then by \cite[Section 14]{asch} there are three classes of so-called  \emph{novelty maximal subgroups} \(\mathcal{A}_{i}\) for \(i=1,2,3\) to be considered. Note that a maximal subgroup  \(H\) of  \(G\) is a novelty maximal subgroup if \(H\cap \PSp_{4}(2^{f})\) is not maximal in \(\PSp_{4}(2^{f})\). 

\medskip

After giving some necessary preliminary results and discussion in Sections~\ref{s:prelim} and \ref{s:factn}, we give the main analysis for Theorem~\ref{maintheorem} in Section~\ref{sect:mainthm}. We  treat in separate subsections the possibilities that a vertex stabiliser $G_v$ is a maximal subgroup lying in  \(\mathcal{C}_{1},\ldots,\mathcal{C}_{9}\) and in \(\mathcal{A}_{1},\ldots,\mathcal{A}_{3}\). We remark that our treatment in Subsection~\ref{sub:c2} of the case where $G_v$ is a  \(\mathcal{C}_{2}\)-subgroup requires significant new theoretical results about homogeneous factorisations of wreath products which we develop in Subsection~\ref{s:homfact}.   This analysis is very different from the geometrical approach used in    Proposition~\ref{c1nonde} to deal with the \(\mathcal{C}_{1}\)-subgroups.

\medskip
The methods employed in this paper are a significant development of those used to analyse the analogous problem for linear groups with socle \(\PSL_{n}(q)\) in \cite{linear}. The new results or ones similar to them, especially those in Section~\ref{s:homfact} on homogeneous factorisations of wreath products, should find application for other classical groups. In particular, the analysis for unitary and orthogonal groups should follow a similar pattern to that for the symplectic groups \(\PSp_{2n}(q)\) in this paper, although additional new ideas will be needed.  Indeed,  in his PhD thesis  \cite{LChenThesis}, the first author  has thoroughly analysed,  for unitary groups \(\PSU_{m}(q)\), the cases where a vertex stabiliser is a maximal subgroup of type \(\mathcal{C}_{1}\), \(\mathcal{C}_{3}\), or \(\mathcal{C}_{4}\). The authors together with Xia are in the process of completing this analysis for the other kinds of maximal subgroups in forthcoming work. 
The situation for orthogonal groups might be trickier, since there are three different types of orthogonal groups, and many more factorisations.

\medskip

Now we remind the reader about some of the terms used above.  A \emph{digraph} \(\Gamma\) is a pair \((V,\to)\) such that \(V\) is the set of vertices and \(\to\) is an anti-symmetric and irreflexive relation on \(V\); $\Gamma$  is said to be \emph{finite} if \(|V|\) is finite and all the digraphs we consider in this paper will be finite. For a non-negative integer \(s\), we call a sequence \(v_{0}, v_{1},\dots, v_{s}\) in \(V\) an \emph{\(s\)-arc} if \(v_{i}\to v_{i+1}\) for each \(i\in\{0,1,\dots,s-1\}\). Note that a 1-arc is simply called an \emph{arc}. For \(G\leqslant \Aut(\Gamma)\), we say that \(\Gamma\) is a \((G,s)\)-arc-transitive digraph if \(G\) acts transitively on the set of \(s\)-arcs of \(\Gamma\), and that \(\Gamma\) is   \emph{\(G\)-vertex-primitive} if \(G\) acts primitively on \(V\). We note that an \((s+1)\)-arc-transitive digraph is naturally \(s\)-arc-transitive if every \(s\)-arc extends to an \((s+1)\)-arc. In particular, \((s+1)\)-arc-transitivity implies \(s\)-arc-transitivity for a \(G\)-vertex-primitive digraph. 

For comparison, note that a graph \(\Sigma\) is a pair \((V,\sim)\) such that \(V\) is the vertex set and \(\sim\) is a symmetric and irreflexive relation on \(V\). As in the case of digraphs, a sequence \(v_{0}, v_{1},\ldots, v_{s}\) is an \emph{\(s\)-arc} if \(v_{i}\sim v_{i+1}\) for \(0\leqslant i\leqslant s-1\) and \(v_{i}\neq v_{i+2}\) for each \(i\). A graph \(\Sigma\) is said to \emph{\(s\)-arc-transitive} if the automorphism group of \(\Sigma\) acts transitively on the set of \(s\)-arcs.

\section{Preliminaries}\label{s:prelim}

\subsection{Notation}

Throughout the paper, we will use the following notation:

\medskip
\noindent
For a group \(G\), we denote by 
\begin{itemize}
    \item \(\Soc(G)\), \(\R(G)\) and \(G^{(\infty)}\) the product of the minimal normal subgroups of \(G\), the largest soluble normal subgroup of \(G\) and the smallest normal subgroup of \(G\) such that \(G/G^{(\infty)}\) is soluble, respectively;
    \item \(\ICF(G)\) the set of insoluble composition factors of \(G\);
    \item \(\m_{G}(T)\) the multiplicity of a simple group \(T\) as an insoluble composition factor of \(G\);
    \item and for a prime \(p\) and  \(x\in G\), we say that $x$ is \emph{\(p\)-regular} if the order of \(x\) is not divisible by \(p\).
\end{itemize}
For two groups \(A\) and \(B\), we denote by 
\begin{itemize}
    \item \(A\times B\), \(A:B\), \(A.B\)  the direct product of \(A\) and \(B\), the semi-direct product of \(A\) and \(B\), and some unspecified extension of \(A\) by \(B\), respectively; and
    \item if  \(B\) acts on a set \(\Omega\), we let \(\{A_\omega\mid \omega\in \Omega\}\) be a family of isomorphic copies of \(A\) indexed by \(\Omega\) and \(K:=\prod_{\omega\in\Omega}A_\omega\). Then the \emph{wreath product}  \(A\wr B\) of \(A\) by \(B\) is the semidirect product \(K:B\), and \(K\) is the \emph{base group} of \(A\wr B\), where \(B\) acts on \(K\) by \((a_\omega)^{b}=(a_{\omega^{b^{-1}}})\), for \((a_\omega)\in K\) and \(b\in B\). 
\end{itemize}
%For a group \(G\) and a non-abelian simple group \(T\), we denote by \(\ICF(G)\) the set of insoluble composition factors of \(G\) and the multiplicity of \(T\) as a composition factor of \(G\), respectively.
%
For a finite set \(X\) and a finite integer \(n\), we denote by 
\begin{itemize}
    \item \(\Pi(X)\) and \(\Pi(n)\) the set of prime divisors of \(|X|\) and \(n\), respectively; and
    \item \(n_{p}\) the \(p\)-part of \(n\), where \(p\) is a prime.
\end{itemize}

\subsection{Number theoretic results}
Throughout the paper, we will frequently investigate the prime divisors of the orders of groups, so we also give some number theoretic results that we will use.
\begin{definition}
    For integers \(n\geqslant2\) and \(m\geqslant2\), a prime number \(p\) is called a \emph{primitive prime divisor} of \(n^{m}-1\) if \(p\) divides \(n^{m}-1\) but \(p\nmid (n^{i}-1)\) for any positive integer \(i<m\). For \(n^{m}-1\), we denote by \(\ppd(n,m)\) the set of all primitive prime divisors of \(n^{m}-1\).
\end{definition}
The next corollary is a product of an elegant theorem of Zsigmondy \cite{Zsigmondy}.
\begin{corollary}\label{existppd}
   For \(n, m\geqslant2\), the set \(\ppd(n,m)\neq \varnothing\) except when \((n,m)=(2,6)\), or \(n+1\) is a power of \(2\) and \(m=2\). Moreover, if $r\in\ppd(n,m)$, then $r\geqslant m+1$.
\end{corollary}
%Note that for \(p\in\ppd(n,m)\) we have from Fermat's Little Theorem that \(p\equiv 1\pmod{m}\) and so \(p>m\).
\begin{lemma}{(Legendre's formula)}\label{sizeppd}\cite[Lemma 2.1]{linear}
 For any positive integer \(n\) and prime \(p\) we have \((n!)_{p}<p^{n/(p-1)}\).
\end{lemma}
\begin{corollary}\label{n-3}
For any positive integer \(n\geqslant5\), we have \(n_{2}< 2^{n-4}\).
\end{corollary}
\begin{proof}
The result is trivially true for \(n\leqslant 8\). Thus it suffices for us to consider the case where \(n\geqslant 9\). We note that \(((n-1)!)_{2}\geqslant (8!)_{2}=2^{7}\) when \(n\geqslant 9\). By Lemma \ref{sizeppd} we obtain that \((n!)_{2}< 2^{n}\) and therefore \((n!)_{2}\leqslant 2^{n-1}\). Thus
\[n_{2}=\frac{(n!)_{2}}{((n-1)!)_{2}}\leqslant \frac{2^{n-1}}{2^{7}}<2^{n-4}
\] and the result follows.
\end{proof}
\subsection{Graph theoretic results}\label{cosgh}
Here we will give some idea of how to construct a digraph. For a group \(G\) and subgroup \(H\) of \(G\), let \(V:=\{Hx\mid x\in G\}\) and \(g\) be an element in \(G\setminus H\) such that \(g^{-1}\notin HgH\). We define a binary relation \(\to\) on \(V\) such that \(Hx\to Hy\) if and only if \(yx^{-1}\in HgH\) for any \(x,y\in G\). Then \((V,\to)\) is a digraph and we denote it by \(\Cos( G, H, g)\). In particular, the action \(R_{H}(G)\) of \(G\) on \(V\) is the right multiplication which preserves the relation \(\to\). Thus \(R_{H}(G)\leqslant \Aut(\Cos(G, H, g))\). By \cite[Lemma 2.4]{quasi}, \(\Cos(G, H, g)\) is \(R_{H}(G)\)-arc-transitive.

Here are some results in \cite{quasi} and \cite{linear} that reveal some important properties for an \(s\)-arc-transitive digraph \(\Gamma\) where \(s\geqslant2\). 
The following results give the conditions for a \(G\)-vertex-primitive digraphs to be \(s\)-arc-transitive, which are cited directly from \cite{linear}.

\begin{lemma}\label{s-1}\cite[Lemma 2.12]{linear}
  Let \(\Gamma\) be a digraph and \(G\leqslant \Aut(\Gamma)\) such that \(G\) acts \(s\)-arc-transitively on \(\Gamma\) with \(s\geqslant2\). Let \(L\) be a non-trivial normal subgroup of \(G\). If \(L\) acts transitively on the vertex set of \(\Gamma\), then \(L\) acts \((s-1)\)-arc-transitively on \(\Gamma\).
\end{lemma}
\begin{lemma}{\rm \cite[Lemma 2.14]{linear}\label{normal}}
Let \(\Gamma\) be a connected \(G\)-arc-transitive digraph with arc \(v\rightarrow w\). Let \(g\in G\) such that \(v^{g}=w\). Then \(g\) normalises no proper nontrivial normal subgroup of \(G_{v}\).
\end{lemma}
\begin{lemma}\label{0}\cite[Lemma 2.11]{linear}
Let \(\Gamma\) be a digraph, and \(v_{0}\rightarrow v_{1}\rightarrow\ldots\rightarrow v_{s-1}\rightarrow v_{s}\) be a \(s\)-arc of \(\Gamma\) with \(s\geqslant2\). Suppose that \(G\) acts arc-transitively on \(\Gamma\). Then \(G\) acts \(s\)-arc-transitively on \(\Gamma\) if and only if \(G_{v_{1}\ldots v_{i}}=G_{v_{0}v_{1}\ldots v_{i}}G_{v_{1}\ldots v_{i}v_{i+1}}\) for each \(i\in\{1,\ldots,s-1\}\). 

\end{lemma}
Note that the factorisation of \(G_{v_{1}\ldots v_{i}}\) occurring in Lemma \ref{0} satisfies the property that the two factors are isomorphic to each other. 
\begin{definition}\label{hmfac}
    For a group \(G\) and proper subgroups \(H, K\leqslant G\), we say \(G=HK\) is a \emph{homogeneous factorisation} if \(H\cong K\).
\end{definition}
We now set out the following hypothesis that we will use throughout this section.

\begin{hypothesis}\label{ga}
\rm{Let \(\Gamma\) be a vertex-primitive \((G,s)\)-arc-transitive digraph for some \(s\geqslant 2\), and let \(u\to v\to w\) be a \(2\)-arc and \(g\in G\) such that \((u,v)^{g}=(v,w)\). Then by Lemma \(\ref{0}\) we have \((G_{uv})^{g}=G_{vw}\) and \(G_{v}=G_{uv}G_{vw}\) is a homogeneous factorisation.}

\end{hypothesis}

Note that a necessary condition for a digraph \(\Gamma\) to be \(G\)-vertex-primitive and \((G,2)\)-arc-transitive is that \(G_{v}\) is a maximal subgroup of $G$ and  \(G_{v}\) admits a homogeneous factorisation. Therefore, to disprove the \(2\)-arc-transitivity of \(G\) it suffices for us to show that \(G_{v}\) does not have such a homogeneous factorisation.

We have the following elementary results.

\begin{lemma}\label{coprime ab}
Suppose that \(\Gamma\) is a \((G,s)\)-arc-transitive digraph with arc \(u\to v\). Let \(g\in G\) and \((u,v)^{g}=(v,w)\). Suppose that \(G_{v}=T:K\), where \(T\cong C_{a}\) and \(|K|=b\) such that \(\gcd(a,b)=1\). Then \(s\leqslant 1\).
\end{lemma}
\begin{proof}
Suppose for a contradiction that \(s\geqslant2\). Then we deduce by Lemma \ref{0} that \(G_{v}=G_{uv}G_{vw}\) with \(G_{uv}^{g}=G_{vw}\). Let us denote by \(\pi\) the natural projection from \(G_{v}\) to \(G_{v}/T\). Then for any subgroup \(H\leqslant G_{v}\) we have that \(|H|=|H\cap T||\pi(H)|\). Let \(p\) be a prime divisor of \(a\) and \(P\) be a subgroup of order \(p\). Then we have that \(|P|=|P\cap T||\pi(P)|\). Since \(\pi(P)\leqslant K\) and \(|K|=b\) is not divisible by \(p\), we have that \(P\leqslant T\). Moreover, let \(Q\) be another subgroup of order \(p\) in \(G_{v}\), then we have that \(Q\leqslant T\). Since \(T\) is cyclic, we deduce that \(P=Q\) and \(P\) is the unique subgroup of order \(p\) in \(G_{v}\). Note that \(\Pi(G_{v})=\Pi(G_{uv})=\Pi(G_{vw})\). Thus \(p\in\Pi(G_{uv})\cap\Pi(G_{vw})\) and this gives that \(P\leqslant G_{uv}\cap G_{vw}\). Since \(P\) is the unique subgroup of order \(p\) in \(G_{v}\), we have that \(P^{g}=P\). However, this contradicts Lemma \ref{normal}.
\end{proof}
We will investigate the conditions of \(3\)-arc-transitivity for different groups under Hypothesis \ref{ga} in the following lemmas.
\begin{lemma}\label{valency}
  Suppose that Hypothesis \ref{ga}  holds and \(s\geqslant3\). Then, for each prime \(p\) dividing \(|G_{v}|\), we have \(|G_{v}|_{p}\geqslant \gamma_{p}^{3}\) where \(\gamma\) is the valency of the digraph.
\end{lemma}

\begin{proof}
 Let \(\Omega\) be the set of \(3\)-arcs starting with \(v\) in \(\Gamma\). Then \(|\Omega|=\gamma^{3}\) and so \(|\Omega|_{p}^{3}=\gamma_{p}^{3}\). Since \(G\) acts \(3\)-arc-transitively on \(\Gamma\) we have \(G_{v}\) acts transitively on \(\Omega\) and therefore \(|G_{v}|_{p}\geqslant \gamma_{p}^{3}\).
\end{proof}

\subsection{Computational methods}\label{r:magma}
We used the computer system MAGMA \cite{magma} to check small groups to investigate if they admit homogeneous factorisations with certain properties, and if so find all of them. For a group \(G\), a necessary condition for \(G\) to admit a homogeneous factorisation \(G=HK\) is that for any prime \(p\in \Pi(G)\), we have \(|H|_{p}=|K|_{p}\geqslant |G|_{p}^{1/2}\). We use the MAGMA command \textbf{LowIndexSubgroups} to generate the set `\textbf{lis}' of all pairs $(H,K)$ of subgroups $H,K$ of \(G\), up to conjugacy, with this property for all $p\in\Pi(G)$. (Note that \textbf{LowIndexSubgroups} generates the representatives of conjugacy classes.) %Hence \(H\) and \(K\) are representatives of conjugacy classes of subgroups of \(G\).} 
We examine the orders of these subgroups and create a sub-list `\textbf{lis2}' consisting of pairs \((H,K)\) from `\textbf{lis}' such that \(|H|=|K|\). To decide, for a pair $(H,K)$ in `\textbf{lis2}', whether \(G=HK\) it suffices to check if \(|G|\cdot|H\cap K|=|H|\cdot|K|\),  and we note that this equality is independent of the choice of $H,K$ in their conjugacy classes. For these groups $G$, it was possible to  use the MAGMA command \textbf{meet} to compute the intersection $H\cap K$ and to check whether this equality held by typing the command 
\begin{center}
\begin{verbatim}# G eq (#a[1])*(#a[2])/#(a[1] meet a[2]) \end{verbatim}    
\end{center}
where, for `\textbf{a}' in `\textbf{lis2}',  `\textbf{a[1]}' and `\textbf{a[2]}' stand for \(H\) and \(K\), respectively. Thus a further sub-list `\textbf{lis3}' is generated consisting of all pairs $(H,K)$ in `\textbf{lis2}' for which $G=HK$.  Finally, for pairs $(H,K)$ in `\textbf{lis3}', we check if \(H\) and \(K\) are isomorphic via the MAGMA command \textbf{IsIsomorphic(a[1],a[2])}, producing a final sub-list `\textbf{lis4}' of pairs $(H,K)$ corresponding to homogeneous factorisations of $G$.
Here is an example of how we check the homogeneous factorisations of \(G_{v}=\Sp_{2}(3)\wr\Sy_{4}\) in Lemma \ref{sp23c2}.
\begin{verbatim}
> G:=WreathProduct(Sp(2,3),Sym(4));
> lis:=LowIndexSubgroups(G,1152);
> lis2:=[<x,y>:x,y in lis|x ne y and #x eq #y];    
> lis3:=[a:a in lis2| #H eq (#a[1])*(#a[2])/#(a[1] meet a[2])];
> #lis4:=[a: a in lis3|IsIsomorphic(a[1],a[2])];
\end{verbatim}    
With similar code, we  check for homogeneous factorisations of \(G_{v}=\Sp_{2}(3)\wr\Sy_{4}\) in Lemma \ref{sp23c2}, and for \(G_{v}=(\PSp_{2}(3)\times\PGO_{4}^{-}(3)).2\) and \((\PSp_{2}(3)\times\PGO_{4}^{-}(3)).[4]\) in Lemma \ref{l:c4-1}. In each of these cases the resulting `\textbf{lis4}' was the empty list, showing that none of these groups have homogeneous factorisations.

% {\color{red} Lei: This is good now. Thanks. I edited quite a bit so the text matches your code. Please check that you are happy with it.}

\section{Group factorisations}\label{s:factn}

In this chapter, we will develop some results on various group factorisations, which will play an important role in Sections 4 and 5.

\subsection{Basic facts about group factorisations}

The following lemma investigates necessary conditions for a group to admit a factorisation.

\begin{lemma}\label{HK}
Let \(G\) be a finite group and \(T\) be a proper non-trivial subgroup of \(G\) such that $T$ is not normal in $G$. Suppose that \(g\in G\), and that \(H,K\leqslant G\) are such that \(H\leqslant N_{G}(T)\) and \(K\leqslant N_{G}(T^{g})\). Then \(HK\neq G\).
\end{lemma}
\begin{proof}
Suppose for a contradiction that \(G=HK\). Let us consider the transitive action of \(G\)  by conjugation on \( \Sigma:= \{ T^x\mid x\in G\}\). The stabiliser in $G$ of an element $T^x\in\Sigma$ is $N_G(T^x)$, and since $T$ is not normal in $G$, the size $|\Sigma|>1$. Since $H\leqslant N_G(T)$, the factorisation $G=HK$ implies that $G=N_G(T)K$,  and as $N_G(T)$ is the stabiliser of the `point' $T$ of $\Sigma$ we conclude that $K$ is transitive on $\Sigma$. However the condition $K\leqslant N_G(T^g)$ implies that $K$ stabilises $T^g\in\Sigma$, and we have a contradiction.
\end{proof}

Recall that we have defined the homogeneous factorisations in Definition \ref{hmfac}. The following results discuss homogeneous factorisations of certain groups.

\begin{lemma}\label{abp}\cite[Lemma 3.2]{linear}
For any homogeneous factorisation \(G=AB\) we have that \(\Pi(A)=\Pi(B)\) and for any prime \(p\), we have \(|A|_{p}^{2}=|B|_{p}^{2}\geqslant|G|_{p}\).
\end{lemma}

\begin{lemma}\label{wreath}\cite[Lemma 3.5]{linear}
Suppose that \(T\wr S_{k}\leqslant G\leqslant R\wr S_{k}\) where \(T\leqslant R\). Let \(M\) be the base subgroup of  $R\wr S_k$ and \(G=AB\) be a homogeneous factorisation for \(G\) such that \(A\) is transitive on the set of $k$ components of $M$. Then with \(\varphi_{i}(A\cap M)\) being the projection of \(A\cap M\) to the \(i\)-th component of \(M\), we have \(\varphi_{1}(A\cap M)\cong\cdots\cong\varphi_{k}(A\cap M)\) and \(\Pi(\varphi_{1}(A\cap M))\supseteq \Pi(T)\).
\end{lemma}
\begin{lemma}\label{GT}
Suppose that \(G=T\wr S_{k}\), and \(G=AB\) is a homogeneous factorisation. Let \(M=T_{1}\times\cdots\times T_{k}\) be the base subgroup of \(G\), where \(T_{1}\cong\cdots\cong T_{k}\cong T\). Let us denote by \(\pi\) the projection from \(G\) to \(G/M\) and, for each $i=1,\dots, k$, let \(\varphi_{i}\) be the projection from \(M\) to \(T_{i}\). Suppose that \(A\) is transitive on \(\{T_{1},\ldots, T_{k}\}\). Then for any odd prime \(p\in \Pi(T)\), we have \(|\varphi_{1}(A\cap M)|_{p}^{2}\geqslant |T|_{p}\).
\end{lemma}

\begin{proof}
Suppose for a contradiction that there exists a prime \(p\geqslant3\) such that \(|\varphi_{1}(A\cap M)|_{p}^{2}<|T|\). Let \(p^{a}\) and \(p^{b}\) denote the \(p\)-parts of \(|\varphi_{1}(A\cap M)|\) and \(|T|\), respectively. Then \(p^{2a}\leqslant p^{b-1}\) and we have 
\begin{equation}\nonumber
\begin{split}
    |A|_{p}^{2}&\leqslant|\varphi_{1}(A\cap M)|^{2k}_{p}|\pi(A)|_{p}^{2}\\
               &=p^{2ak}|\pi(A)|_{p}^{2}\leqslant p^{2ak}(k!)^{2}_{p}<p^{2ak}p^{\frac{2k}{p-1}}\\
               &\leqslant p^{(b-1)k}p^{\frac{2k}{p-1}}\leqslant p^{bk}=|T|_{p},
               \end{split}
\end{equation}
which is a contradiction, by Lemma \ref{abp}.
\end{proof}
The following results investigate the factorisations of a group involving symmetric groups and alternating groups.

\begin{lemma}\label{k=5}
Let \(N\triangleleft G\) such that \(G/N\cong\Sy_{k}\) where \(k=5\) or \(6\) and \(|N|_{5}=1\). Suppose that \(G=HK\) is a homogeneous factorisation. Then either \(\A_{k}\triangleleft HN/N \cap KN/N \leqslant \Sy_{k}\), or, interchanging \(H\) and \(K\) if necessary, \(k=6\) and \(\PSL_{2}(5)\triangleleft HN/N\leqslant \PGL_{2}(5)\) with \(\A_{5}\triangleleft KN/N\leqslant\Sy_{5}\).
\end{lemma}
\begin{proof}
Since \(G=HK\), we have  \((HN/N)(KN/N)=G/N\cong\Sy_{k}\). On the other hand, since \(|N|_{5}=1\) and \(k=5\) or \(6\), we have \(|G|_{5}=5\) and hence also \(|H|_{5}=|K|_{5}>1\) by Lemma \ref{abp}. Thus \(|HN/N|_{5}>1\) and \( |KN/N|_{5}>1\). On the other hand, note that \(|G/N|_{3}\geqslant(5!)_{3}=3\). This implies that at least one of \(|HN/N|\) and \(|KN/N|\) is divisible by \(3\). Without loss of generality, assume that \(|HN/N|_{3}\geqslant3\) and so \(|HN/N|\) is divisible by \(15\). Then since \(HN/N\) is a subgroup of \(\Sy_{5}\) or \(\Sy_{6}\), we check by MAGMA \cite{magma} (see Subsection \ref{r:magma}) that either \(\A_{k}\leqslant HN/N\), or \(k=6\), \(\PSL_{2}(5)\triangleleft HN/N\leqslant \PGL_{2}(5)\) or \(\A_{5}\leqslant HN/N\leqslant\Sy_{5}\). In both cases \(HN/N\) has an insoluble composition factor \(T\) isomorphic to \(\A_{6}\) or \(\A_{5}\)(\(\cong\PSL_{2}(5)\)). Since \(|N|_{5}=1\), we see that \(T\) is not a section of \(N\) and since \(K\cong H\), it follows that \(T\) is a section of \(KN/N\). 

If \(T=\A_{k}\), then the first condition \(\A_{k}\triangleleft HN/N\cap KN/N\leqslant\Sy_{k}\) holds. Therefore we assume that \(k=6\) and \(T=\PSL_{2}(5)(\cong\A_{5})\), then since \((HN/N)(KN/N)=\Sy_{6}\) and \(\PSL_{2}(5)\triangleleft HN/N\leqslant \PGL_{2}(5)\), we deduce by \cite[Theorem D, Remark 2]{LPS} that \(\A_{5}\triangleleft KN/N=\Sy_{5}\), proving the result. 
\end{proof}

\begin{lemma}\label{ab}\cite[Theorem D and Remark 2]{LPS}
Let \(\A_{n}\leqslant G\leqslant \Sy_{n}\) and \(G\) acts on a set \(\Omega\) of size \(n\geqslant 3\). Suppose that \(G=AB\) with subgroups \(A\) and \(B\) of \(G\). Then at least one of \(A\) or \(B\) is transitive on \(\Omega\).
\end{lemma}

\subsection{Homogeneous factorisations of \(\Sp_{2m}(q)\wr\Sy_{k}\)}\label{s:homfact}

By \cite[Table 3.5 C]{kleidman}, a maximal \(\mathcal{C}_{2}\)-subgroup of \(\Sp_{2n}(q)\) has the form \(G=\Sp_{2m}(q)\wr\Sy_{k}\) where \(n=mk\), for some $m\geqslant 1$ and $k\geqslant 2$. In this section we investigate the homogeneous factorisations of such groups.

\begin{hypothesis}\label{wrspsk}
\rm{Let \(G=\Sp_{2m}(q)\wr\Sy_{k}\) such that \(m\geqslant1, k\geqslant2\), and \(q=p^{f}\) for some prime \(p\), with \((m,p)\neq(2,2)\), and \((m,q)\neq (1,2), (1,9)\). Let us denote by
\(M\) the base subgroup of \(G\), \(M_{i}\) the \(i\)-th component of \(M\) for \(i=1,\ldots, k\) and \(Z\) the centre of \(M_{1}\). We also define \(\pi\) and \(\varphi_{i}\) to be the natural projections from \(G\) to \(G/M\) and from \(M\) to \(M_{i}\) for \(i=1,\ldots, k\), respectively. Let \(G=HK\) be a homogeneous factorisation. Then \(\Sy_{k}=\pi(G)=\pi(H)\pi(K)\) and by Lemma \ref{ab} we deduce that at least one of \(\pi(H)\) or \(\pi(K)\) is transitive. Without loss of generality, assume that \(\pi(H)\) is transitive.}
\end{hypothesis}

\begin{lemma}\label{prerequisites}
Suppose that Hypothesis \ref{wrspsk} holds and \((m,q)\neq (1,3)\). Then the following hold:
\begin{description}
\item[(a)] Up to isomorphism, \(H\cap M\) has a unique insoluble composition factor \(T\), and \(T=\PSp_{2m}(q), \Omega_{2m}^{-}(q)\) (\(q\) even), \(\PSp_{2}(q^{2})\) \((m=2)\), or \(\A_{c}\) (with \(c=7, 8\) and \((m,q)=(3,2)\)).
\item[(b)] \(Z.T\leqslant \varphi_{1}(H\cap M)\leqslant (Z.T).2\). Moreover, either \(Z=1\), or \(Z.T\) is quasisimple. In particular, for any odd prime \(s\) we have \(|\varphi_{1}(H\cap M)|_{s}=|T|_{s}\).
\item[(c)] Let \(\ell\) denote the multiplicity of \(T\) as a composition factor in \(H\cap M\). Then \(\ell\) divides \(k\), and \(\pi(H)\) preserves a partition of \(\{1,\ldots,k\}\) with \(k/\ell\) parts of size \(\ell\).
\end{description}
\end{lemma}

\begin{proof} 
Suppose that Hypothesis \ref{wrspsk} holds. Then since \(\pi(H)\) is transitive, it follows from Lemma \ref{wreath} that \(\varphi_{1}(H\cap M)\cong\cdots\cong\varphi_{k}(H\cap M)\) and \(\Pi(\varphi_{1}(H\cap M))=\Pi(\varphi(\Sp_{2m}(q)))\). Note that \(Z=\C_{(2,q-1)}\) and \(M_{1}/Z=\PSp_{2m}(q)\).

\medskip
\noindent
\textit{Claim~1:} \(\Pi(\varphi_{1}(H\cap M)Z/Z)\cup\{2\}=\Pi(\PSp_{2m}(q))\). Moreover, \(\Pi(\varphi_{1}(H\cap M)Z/Z)=\Pi(\PSp_{2m}(q))\) unless \(q=p\) is a Mersenne prime and \(m=1\).

\medskip
\noindent

For \(q\) even we have \(Z\) is trivial and so \(\varphi_{1}(H\cap M)Z/Z=\varphi_{1}(H\cap M)\) and \(\PSp_{2m}(q)=\Sp_{2m}(q)\). Thus the claim directly follows. 

For \(q\) odd we have \(Z\cong \C_{2}\) and so
\[
    \Pi(\varphi_{1}(H\cap M)Z/Z)\cup\{2\}=\Pi(\varphi_{1}(H\cap M))=\Pi(\Sp_{2m}(q))=\Pi(\PSp_{2m}(q)).
\]
By \cite[Corollary 6 (c)]{transitive} we deduce that for a subgroup \(S\) of \(\PSp_{2m}(q)\) such that \(q\) is odd and \(\Pi(S)\cup\{2\}=\Pi(\PSp_{2m}(q))\), either \(\Pi(S)=\Pi(\PSp_{2m}(q))\) or \(q=p\) is a Mersenne prime and \(m=1\). Therefore Claim~1 follows.

As a result of Claim~1 and \cite[Corollary 5, Table 10.3 and  10.7]{transitive}, one of the following holds. 
\begin{description}
\item[(i)] \(m, q\) are even, \(Z=1\) and \(\Omega^{-}_{2m}(q)\leqslant\varphi_{1}(H\cap M)\leqslant
\Omega^{-}_{2m}(q).2\).
\item[(ii)] \(m=3\), \(q=2\), \(Z=1\) and \(\A_{c}\leqslant\varphi_{1}(H\cap M)\leqslant\Sy_{c}\) where \(c=7,8\).
\item[(iii)] \(m=1\), \(q=p\) and \(\varphi_{1}(H\cap M)Z/Z\leqslant \C_{q}:\C_{\frac{q-1}{2}}\), where \(q=2^{a}-1\) is a Mersenne prime.
\item[(iv)] \(m=2\), \(q=7\) and \(\varphi_{1}(H\cap M)Z/Z=\A_{7}\).
\item[(v)] \(m=2\), \(q=3\) and \(\varphi_{1}(H\cap M)Z/Z\leqslant 2^{4}.\A_{5}\) or \(\Sy_{5}\).
\item[(vi)] \(m=2\), and \(\PSp_{2}(q^{2})\leqslant\varphi_{1}(H\cap M)Z/Z\leqslant \PSp_{2}(q^{2}).2\).
\item[(vii)] \(\varphi_{1}(H\cap M)Z/Z=\PSp_{2m}(q)\).
\end{description}
Note that \(\PSU_{4}(2)\cong\PSp_{4}(3)\) for case (iv), and the possibilities of \(\A_{6}\) and \(\Sy_{6}\) are included in case (vi) since \(\PSL_{2}(9)\cong\A_{6}\). Also \(\PSL_{2}(q)\cong\PSp_{2}(q)\) gives the examples in (iii).

\medskip
\noindent
\textit{Claim~2:}
Up to isomorphism, \(H\cap M\) has a unique insoluble composition factor \(T\), and \(T=\PSp_{2m}(q)\), \(\Omega_{2m}^{-}(q)\) (\(q\) even), \(\PSp_{2}(q^{2})\) (\(m=2\)), or \(\A_{c}\) (with \(c=7,8\) and \((m,q)=(3,2)\)).

\medskip
\noindent
To prove Claim~2 it is sufficient to show that cases (iii)-(v) do not arise. Let us define \(r\) by
\[
    r:=\begin{cases}
    2&\text{if \((m,q)=(1,2^{a}-1)\)}\\
    3&\text{if \((m,q)=(2,3)\)}\\
    7&\text{if \((m,q)=(2,7)\)}.
    \end{cases}
\]
Then  \[|\varphi_{1}(H\cap M)Z/Z|_{r}=\begin{cases}
|q(q-1)/2|_{r}=1&\text{if \( r=2\)}\\|2^{4}.\A_{5}|_{r}=|\Sy_{5}|_{r}=r&\text{if  \(r=3\)}\\|\A_{7}|_{r}=r&\text{if \(r=7\).}
\end{cases}\]
If \(r=2\) (in case (iii)), then 
\[
|\varphi_{1}(H\cap M)|_{2}\leqslant|\varphi_{1}(H\cap M)Z/Z|_{2}|Z|_{2}\leqslant 2.
\] 
and
\begin{equation}\label{3}|H|_{2}^{2}\leqslant|\varphi_{1}(H\cap M)|_{2}^{2k}(k!)_{2}^{2}\leqslant 2^{2k}(k!)_{2}^{2}<2^{3k}(k!)_{2}.
\end{equation} 
On the other hand, noting that \(q+1 = 2^a\geqslant 8\) (since \((m,q)\neq (1,3)\)), so  \(|\Sp_{2}(q)|_{2}=|\SL_{2}(q)|_{2}=(q+1)_{2}(q-1)_{2}=2^{a+1}\geqslant 2^{4}\) and therefore
\[
    |G|_{2}=|\Sp_{2}(q)|_{2}^{k}(k!)_{2}=2^{(a+1)k}(k!)_{2}\geqslant2^{4k}(k!)_{2}>|H|_{2}^{2}.
\] which is a contradiction. 

If \(r\) is odd (in case (iv) or (v)), then \(r=q\) and so \(|Z|_{r}=(2)_{r}=1\) and therefore  \(|\varphi_{1}(H\cap M)|_{r}=r\). However,
it follows from Lemma \ref{GT} that \(|\varphi_{1}(H\cap M)|_{q}\geqslant |\Sp_{4}(q)|_{q}^{1/2} = q^2\), which is a contradiction. Thus none of cases (iii), (iv), (v) hold for $G$, and Claim~2, hence also part (a) of Lemma \ref{prerequisites}, is proved. 

\medskip
We now prove part (b). Note that \(|Z|=1\) when \(q\) is even. Thus part (b) holds for \(T=\PSp_{2m}(2^{f})\), \(\Omega_{2m}^{-}(q)\) and \(\A_{c}\) for \(c=7,8\). In the remaining cases \(T=\PSp_{2m}(q)\), or \(m=2\) and \(T=\PSp_{2}(q^{2})\) with \(q\) odd in both cases. By \cite[Theorem 8.7]{Taylor} the symplectic groups are perfect over a field of odd order \(q\), and this implies that \(\Sp_{2m}(q)\) has no subgroup isomorphic to \(\PSp_{2m}(q)\) or \(\PSp_{2}(q^{2})\). Thus \(Z.T\) is non-split. Since \(|Z|=2\) for \(q\) odd, we deduce that \(Z.T\) is also quasisimple and the assertion about odd \(q\) follows immediately. Thus part (b) of Lemma \ref{prerequisites} holds.

Let \(\ell:=\m_{H\cap M}(T)\) and \(Z(M)\) be the centre of \(M\). Then part (c) follows from Scott's Lemma, see \cite[Theorem 4.16]{csaba}.
\end{proof}

\begin{lemma}\label{Sp}
Suppose that Hypothesis \ref{wrspsk} holds and \((m,q)\neq (1,3)\). Then \(k=2\).
\end{lemma}
\begin{proof}
Suppose for a contradiction that Hypothesis \ref{wrspsk} holds with \(k\geqslant3\). Then by part (a) and (b) of Lemma \ref{prerequisites},  \(\varphi_{1}(H\cap M)\) has a unique insoluble composition factor \(T\) and \(Z.T\leqslant \varphi_{1}(H\cap M)\leqslant (Z.T).2\). Moreover by part (c) of Lemma \ref{prerequisites} we have \(\ell\) divides \(k\) (recall that \(\ell:=\m_{H\cap M}(T)\)).

\medskip
\noindent
\textit{Claim~1:}
 \(\ell=k\). 

\medskip
\noindent
Suppose to the contrary that \(\ell<k\). Suppose first that $\ell>1$. Then the partition in part (c) of Lemma \ref{prerequisites} is nontrivial and  \(\pi(H)\) is imprimitive.
Then  \(\pi(H)\leqslant \Sy_{k/\ell}\wr\Sy_{\ell}\), and $k\geqslant 4$ and is not prime.  If \(k\geqslant8\), then by Bertrand's Postulate \cite[Theorem 8.7]{ivan}, there exists a prime \(\phi(k)\) such that \(k/2<\phi(k)\leqslant k-2\). Let us define \(d\) by 
\[
    d=\begin{cases}
    3&\text{if \(k=4\)}\\
    5&\text{if \( k=6\)}\\
    \phi(k)&\text{if \(k\geqslant8\)}.
    \end{cases}
\]
Note that \(d>k/2\) and so \(d>k/\ell\) and \(d>\ell\). Consequently, \(d\) does not divide \((\frac{k}{\ell})!\) or \(\ell!\). Hence \(|\pi(H)|_{d}=1\) while \((k!)_{d}=d\). Let us denote by \(d^{c}:=|\Sp_{2m}(q)|_{d}\). By part (b) of Lemma \ref{prerequisites},  \(|\varphi_{1}(H\cap M)|_{d}=|T|_{d}\) since \(d\) is odd. Thus we have that:
\[%\label{h2}
    |H|_{d}^{2}\leqslant(|T|^{k/2}_{d}|\pi(H)|_{d})^{2}=|T|_{d}^{k}\leqslant|\Sp_{2m}(q)|^{k}_{d}=d^{kc}
\]
and 
\[
    |H|_{d}^{2}\geqslant|G|_{d}=|\Sp_{2m}(q)|_{d}^{k}(k!)_{d}=d^{kc+1},
\]
which is a contradiction. 

Hence $\ell=1$. Since \(\Sp_{2m}(q)\) is insoluble, we deduce by Burnside's \(p^{a}q^{b}\) Theorem \cite[ Theorem 3.3, page 131]{Gorenstein} that \(|\Sp_{2m}(q)|\) has at least three prime divisors. In particular, there exists \(r\in\Pi(\Sp_{2m}(q))\) such that \(r\geqslant5\). Let us denote by \(r^{a}\) the \(r\)-part of \(|\Sp_{2m}(q)|\). Then \(a\geqslant1\), and by Lemma \ref{sizeppd}, \((k!)_{r}\leqslant r^{k/4}< r^{ak/3}\)  as \(r\geqslant 5\). Thus, since \(2+k/3<k\),
\[
    |H|_{r}^{2}\leqslant|\Sp_{2m}(q)|_{r}^{2}(k!)_{r}^{2}=r^{2a}(k!)_{r}^{2}\leqslant r^{2a}r^{ak/3}(k!)_{r}< r^{ak}(k!)_{r}, 
\]
and 
\[
  |H|_{r}^{2}\geqslant  |G|_{r}=|\Sp_{2m}(q)|_{r}^{k}(k!)_{r}=r^{ak}(k!)_{r},
\]
which is a contradiction. This proves Claim~1.

\medskip
\noindent
\textit{Claim~2:} Let \(X=Z.T\) as in part (b) of Lemma \ref{prerequisites}. Then $(H\cap M)'\cong X^k$, and in particular $Z(M)\lhd (H\cap M)'$ and $(H\cap M)'/Z(M)\cong T^k$ is a minimal normal subgroup of $H/Z(M)$. Moreover, if $N\lhd H$ and $N < (H\cap M)'$, then $N\leqslant Z(M)=Z((H\cap M)')$. 

\medskip
\noindent
Note that \(Z\) is trivial when \(q\) is even, and in this case $Z(M)=1$ and the equality $(H\cap M)'\cong T^k=X^k$ follows from part (b)-(c) of Lemma \ref{prerequisites} and Claim~1. The other assertions follow from the fact that $\pi(H)$ is transitive and \(\PSp_{2m}(q)\) is non-abelian simple for the values of \(m\) and \(q\) considered. 
Thus we may assume that $q$ is odd. Then, by part (b) of Lemma \ref{prerequisites}, \(Z.T\) is a non-split extension with \(|Z|=2\), and  \(\varphi_{i}((H\cap M)')=X\) for \(1\leqslant i\leqslant k\).

For \(1\leqslant i\leqslant k\), define the projection \(\widehat{\pi_{i}}: M\to\prod_{j\neq i}M_{j}\). Then \((H\cap M)'\cap M_{i}\) is the kernel of \(\widehat{\pi_{i}}\) restricted to \((H\cap M)'\). Thus
\begin{equation}
(H\cap M)'/((H\cap M)'\cap M_{i})\cong\widehat{\pi_{i}}((H\cap M)')\leqslant\prod_{j\neq i}\varphi_{j}((H\cap M)')\cong X^{k-1}.
\end{equation}
Since \(X^{k-1}\) has only \(k-1\) composition factors isomorphic to \(T\), we conclude that \((H\cap M)'\cap M_{i}\) also has a composition factor \(T\). Since \(Z.T\) is non-split, it follows that \(Z.T\leqslant (H\cap M)'\cap M_{i}\), and so \((H\cap M)'\cap M_{i}=Z.T\) by part (b) of Lemma \ref{prerequisites}. This holds for all \(i\), and so \((H\cap M)'\cong X^{k}\). In particular $(H\cap M)'$ contains $Z^k=Z(M)=Z((H\cap M)')$ and $(H\cap M)'/Z(M)\cong T^k$. Since $\pi(H)$ is a transitive subgroup of $S_k$ it follows that $(H\cap M)'/Z(M)$ is a minimal normal subgroup of $H/Z(M)$. Finally suppose that $N\lhd H$ and $N<(H\cap M)'$. If $N\not\leqslant Z(M)$, then $NZ(M)/Z(M)$ is a nontrivial normal subgroup of $H/Z(M)$ contained in $(H\cap M)'/Z(M)$ and it follows from the minimality of $(H\cap M)'/Z(M)$ that $NZ(M) = (H\cap M)'$.
Since $X=Z.T$ is non-split, it follows that $\varphi_i(N)=X$ for all $i$, and the argument at the beginning of this paragraph (applied to $N$ instead of $(H\cap M)'$) shows that $N\cong X^k$, and hence $N=(H\cap M)'$, which is a contradiction. Hence $N\leqslant Z(M)$, and Claim~2 is proved.

\medskip

We now consider the structure of $K$ in the homogeneous factorisation $G=HK$. Recall that  $H$ and $K$ are isomorphic.

\medskip
\noindent
\textit{Claim~3:}\label{LM}
Let $\psi:H\to K$ be an isomorphism, and let $L:=\psi((H\cap M)')$ and $Y:=\psi(Z((H\cap M)'))$. 
 Then $Y<L\leqslant K\cap M$, $Y\lhd K$, $L\cong X^k$, $Y=Z(L)$, and each $N\lhd K$ with $N<L$ satisfies $N\leqslant Y$. 

 \medskip
 
It follows from Claim~2 that $Y\lhd K$, $L\cong X^k$, $Y=Z(L)$, and each $N\lhd K$ with $N<L$ satisfies $N\leqslant Y$. It remains to prove that $L\leqslant K\cap M$. 
Suppose to the contrary that  $L\not \leqslant K\cap M$. Then $N:=L\cap M$ is equal to $L\cap (K\cap M)$ (since $L<K$), $N$ is a proper subgroup of $L$ (as otherwise $L=L\cap (K\cap M)\leqslant K\cap M$), and $N\lhd K$ (since both $L$ and $K\cap M$ are normal subgroups of $K$). Therefore $N\leqslant Y$, by the last part of the claim. This implies that \(N\) is soluble and therefore the set of insoluble composition factors of \(L\) and \(L/N\) coincide. Note that 
\[
    L/N=L/(L\cap M)=\pi(L)\leqslant \Sy_{k}.
\] This together with the fact that \(L\cong X^{k}=(Z.T)^{k}\) has \(k\) copies of insoluble composition factors isomorphic to \(T\) implies that \(L/N\) has \(k\) copies of insoluble composition factors isomorphic to \(T\) and so \((k!)\) is divisible by \(|T|^{k}\), which is impossible. Thus Claim~3 is proved.   

\medskip
\noindent
\textit{Claim~4:}
\(T=\PSp_{2m}(q)\).

\medskip
\noindent
Assume that \(T\neq\PSp_{2m}(q)\). Then by Lemma \ref{prerequisites} we have \(T=\Omega_{2m}^{-}(q)\), \(\PSp_{2}(q^{2})\) or \(\A_{c}\) for \(c=7,8\).  By \cite[Theorem \(B\Gamma\)]{asch} and \cite[Table 8.12]{holt} we deduce that all the subgroups of \(M_{1}\) isomorphic to \(X=Z.T\) are conjugate. Therefore all the subgroups of \(M\) isomorphic to \(X^{k}\) are conjugate. Thus $N:=(H\cap M)'$ is conjugate to the subgroup $L$ of  Claim~3, so \( N^{g}=L\) for some \(g\in M\). Note that $N$ is not normal in $G$ for any of these groups $T$ (as otherwise $N\lhd M$ and each composition factor of $N$ would be $\PSp_{2m}(q)$).  Consequently, we have $g\in G$, \(H\leqslant N_{G}(N)\), \(K\leqslant N_{G}(N^{g})\), and  $G=HK$, which violates Lemma \ref{HK}, and proves the claim.

We conclude from Claim~2 that the subgroups $(H\cap M)'$ and $L$ of $M$ from Claims~3 and~4 satisfy \((H\cap M)'\cong L\cong \Sp_{2m}(q)^k\), and hence  \((H\cap M)'=L=M\). Thus \(H=M.\pi(H)\) and \(K=M.\pi(K)\), and since \(H\cong K\) we deduce that \(\Sy_{k}=\pi(H)\pi(K)\) is a homogeneous factorisation. Moreover each of $\pi(H), \pi(K)$ is a proper subgroup of $\Sy_k$ since $H, K$ are proper subgroups of $G$. Since $\Sy_k$ has no homogeneous factorisation if \(k\leqslant4\), we have  \(k\geqslant5\). It follows from \cite[Theorem 1.1]{BP} that \(k=6\), \(\pi(H)=\PGL_{2}(5)\) and \(\pi(K)=\Sy_{5}\). Then \(\pi(K)\) has two orbits on \(\{M_{1},\ldots, M_{6}\}\), of lengths \(5\) and \(1\), respectively. Without loss of generality, we assume that \(\{M_{6}\}\) is the orbit of length \(1\). Then \(M_{6}\) is normalised by \(\pi(K)\) and therefore \(M_{6}\triangleleft K\) and $M_6 < L$, which contradicts the last assertion of Claim~3. This contradiction completes the proof of Lemma~\ref{Sp}.
\end{proof}

\begin{lemma}\label{k=2}
Suppose that Hypothesis \ref{wrspsk} holds, \((m,q)\neq (1,3)\) and \(k=2\). Then, interchanging \(M_{1}\) and \(M_{2}\), and \(H\) and \(K\) if necessary, we have \(K^{(\infty)}=M_{1}\) with \(K\leqslant M\).
\end{lemma}

\begin{proof}
Suppose that Hypothesis \ref{wrspsk} holds. Let \(\psi: H\to K\) be an isomorphism. By Lemma \ref{prerequisites} we have \(Z.T\leqslant\varphi_{1}(H\cap M)\leqslant (Z.T).2\) for \(T=\PSp_{2m}(q)\), \(\Omega_{2m}^{-}(q)\) (with \(q\) even), \(\PSp_{2}(q^{2})\) (with \(m=2\)), or \(\A_{c}\) (with \((m,q)=(3,2)\)) for \(c=7,8\). Moreover, \(\ell\)(\(:=\m_{H\cap M}(T)\)) divides \(k\). As \(k=2\), it follows that \(\ell=1\) or \(2\).

\medskip
\noindent
\textit{Claim~1:}
 \(\ell=1\).

\medskip
\noindent
Suppose for a contradiction that the assertion is false and \(\ell=2\). Then since \(H\cap M=(H\cap M_{2}).\varphi_{1}(H\cap M)\) and \(\m_{\varphi_{1}(H\cap M)}(T)=1\) , it follows that \(\m_{H\cap M_{2}}(T)=1\). Note that \(Z.T\leqslant \varphi_{2}(H\cap M)\leqslant (Z.T).2\) (as \(\varphi_{1}(H\cap M)\cong\varphi_{2}(H\cap M)\)) and \(H\cap M_{2}\triangleleft \varphi_{2}(H\cap M)\) (as \(H\cap M_{2}\triangleleft H\cap M\)). It follows that \((H\cap M_{2})'\triangleleft Z.T\). Since either \(Z.T\) is a non-split extension, or \(Z=1\) (by part (b) of Lemma \ref{prerequisites}), we deduce that \((H\cap M_{2})'\cong Z.T\). As \(\pi(H)=\Sy_{2}\) and \(H\cap M_{1}\cong H\cap M_{2}\), we conclude that \((H\cap M_{1})'\cong (H\cap M_{2})'\cong Z.T\). This implies that
\[
    (\prod_{i=1}^{2}(H\cap M_{i}))'=\prod_{i=1}^{2}(H\cap M_{i})'\cong (Z.T)^{2}
\]
and so 
\[
    |H\cap M:\prod_{i=1}^{2}(H\cap M_{i})'|\leqslant |\prod_{i=1}^{2}\varphi_{i}(H\cap M):\prod_{i=1}^{2}(H\cap M_{i})'|\leqslant 4.
\]
Hence \(H\cap M/\prod_{i=1}^{2}(H\cap M_{i})'\) is abelian and \(\prod_{i=1}^{2}(H\cap M_{i})'\geqslant (H\cap M)'\). On the other hand \(H\cap M\geqslant \prod_{i=1}^{2}(H\cap M_{i})\). Thus we deduce that \((H\cap M)'\cong \prod_{i=1}^{2}(H\cap M_{i})'\cong (Z.T)^{2}\). In particular \(Z(M)\triangleleft (H\cap M)'\) and therefore \((H\cap M)'/Z(M)\) is a minimal normal subgroup of \(H/Z(M)\) by the transitivity of \(\pi(H)\). 

Recall the isomorphism \(\psi: H\to K\). In particular, we have \(\psi((H\cap M)')/\psi(Z(M))\) is a minimal normal subgroup of \(K/\psi(Z(M))\). Note that \((K\cap M)\psi(Z(M))/\psi(Z(M))\triangleleft K/\psi(Z(M))\) as \(K\cap M\triangleleft K\). Thus one of the following occurs:
\begin{description}
\item[(i)]\(\psi((H\cap M)')/\psi(Z(M)) \cap (K\cap M)\psi(Z(M))/\psi(Z(M))=1\), or
\item[(ii)]\(\psi((H\cap M)')/\psi(Z(M))\leqslant (K\cap M)\psi(Z(M))/\psi(Z(M))\).
\end{description}
First assume that case (i) occurs. Then \(\psi((H\cap M)')\cap (K\cap M)=\psi(H\cap M)'\cap M\) is contained in \(Z(M)\), an abelian subgroup, and this implies that \(\ICF(\psi((H\cap M)'))=\ICF(\pi(\psi((H\cap M)')))\). However, this is impossible as \(\pi(\psi((H\cap M)'))\) is contained in \(\Sy_{2}\), which is cyclic.
Thus case (ii) occurs and so \(\psi((H\cap M)')\leqslant (K\cap M)Z(M)\leqslant M\). 

From the above argument we see that both \((H\cap M)'\) and \(\psi((H\cap M)')\) are contained in \(M\). For \(T=\Omega_{2m}^{-}(q)\) (with \(m,q\) even), \(\PSp_{2}(q^{2})\) (with \(q\) odd), or \(\A_{c}\) (with \((m,q)=(3,2)\) and \(c=7,8\)), we deduce by \cite[Theorem B\(\Gamma\)]{asch} and \cite[Table 8.12]{holt} that all of the subgroups of \(M_{1}\) isomorphic to \(Z.T\) are conjugate. Therefore all the subgroups of \(M\) isomorphic to \((Z.T)^{2}\) are conjugate. Thus there exists \(g\in G\) such that \(((H\cap M)')^{g}=\psi((H\cap M)')\). Note that \((Z.T)^{2}\cong (H\cap M)'\triangleleft G\). This together with the fact \((H\cap M)'\triangleleft H\) and \(\psi((H\cap M)')\triangleleft K\) implies that \(G=HK\leqslant N_{G}((H\cap M)') N_{G}(((H\cap M)')^{g}\), violating Lemma \ref{HK}. Thus \(T=\PSp_{2m}(q)\). However, then we would have \(\Sp_{2m}(q)^{2}\cong(H\cap M)'=M\) and therefore \(H=M.2=G\), which is a contradiction. Hence Claim~1 is proved.

\medskip
\noindent
\textit{Claim~2:}
 \(T=\PSp_{2m}(q)\).

\medskip
\noindent
Suppose for a contradiction that \(T\neq\PSp_{2m}(q)\). Then \(T=\Omega_{2m}^{-}(q), \PSp_{2}(q^{2})\), or \(\A_{c}\) for \(c=7,8\). A direct computation tells us that 
\begin{equation}\label{333}
    |\varphi_{1}(H\cap M)|\leqslant|(Z.T).2|<|\Sp_{2m}(q)|/3.
\end{equation}
On the other hand, since \(\ell=1\) by Claim~1 and \(H\cap M\cong (H\cap M_{2}).\varphi_{1}(H\cap M)\) and \(\m_{\varphi_{1}(H\cap M)}(T)=1\), it follows that \(H\cap M_{2}\) is soluble. As \(H\cap M_{2}\triangleleft \varphi_{2}(H\cap M)\) and \((Z.T)\leqslant \varphi_{2}(H\cap M)\leqslant (Z.T).2\), we conclude that \(H\cap M_{2}\lesssim Z\). In particular \(|H\cap M_{2}|\leqslant2\). This together with equation \eqref{333} implies that
\[
    |H|^{2}=4|H\cap M_{2}|^{2}|\varphi_{1}(H\cap M)|^{2}< 16(|\Sp_{2m}(q)|/3)^{2}\leqslant 2|\Sp_{2m}(q)|^{2}=|G|,
\]
which contradicts the assumption that \(G=HK\) is a homogeneous factorisation. Hence Claim~2 is proved.

\medskip
For any subgroup \(S\leqslant G\) we denote by \(\overline{S}:=SZ(M)/Z(M)\) and \(\phi_{i}\) the projection from \(\overline{M}\) to \(\overline{M_{i}}\). Then \(\PSp_{2m}(q)^{2}.2\cong\overline{G}=\overline{H}\,\overline{K}\). By Claim~1 and 2 we deduce that \(\overline{H\cap M}\) is a diagonal subgroup of \(\overline{M}\). In particular, there exists \(\sigma\in\Aut(\PSp_{2m}(q))\) such that \(\overline{H\cap M}=\{ (x, x^{\sigma})\mid x\in\overline{M_{1}}\}\). 

\medskip
\noindent
\textit{Claim~3:}
 \(\overline{H}\) is a maximal subgroup of \(\overline{G}\). Then \(\overline{G}\) acts primitively on \(\Omega:=[\overline{G}: \overline{H}]\).

\medskip
\noindent
Note that \(\overline{H}/(\overline{H}\cap\overline{M})\cong\pi(H)\cong\Sy_{2}\) and \(\phi_{1}(\overline{H}\cap\overline{M})\cong\phi_{2}(\overline{H}\cap\overline{M})\cong \PSp_{2m}(q)\). Let \(P\) be a subgroup satisfying \(\overline{H}\leqslant P<\overline{G}\). Then \(P/(P\cap\overline{M})\cong \overline{H}/(\overline{H}\cap\overline{M})\cong \Sy_{2}\) and \(\phi_{1}(P\cap\overline{M})\cong\phi_{2}(P\cap \overline{M})\cong \PSp_{2m}(q)\). Note that \(\PSp_{2m}(q)\) is non-abelian simple as \((m,q)\neq (1,2), (1,3)\). Then by Scott's Lemma \cite[Theorem 4.16(iii)]{csaba} we deduce that \(P\cap \overline{M}\cong \PSp_{2m}(q)^{j}\) for \(j=1,2\). Since \(P\neq \overline{G}\), it follows that \(j=1\) and \(P=\overline{H}\). Therefore \(\overline{H}\) is maximal in \(\overline{G}\) and \(\overline{G}\) acts primitively on \(\Omega\) with \(\overline{H}\) a point stabiliser, proving Claim~3.

\medskip
Since the action of \(\overline{G}\) on \(\Omega\) is primitive, it follows from the O'Nan-Scott Theorem that this action is of simple diagonal type, see \cite[Table 7.1]{csaba}. Moreover, since \(\overline{G}=\overline{H}\,\overline{K}\) the group \(\overline{K}\) acts transitively on \(\Omega\). Then, by \cite[Theorem 1, Example 1.2]{transitive}, either there exist proper subgroups \(B_{1}, B_{2}\) such that \(\PSp_{2m}(q)=B_{1}B_{2}\) and \(B_{1}\times B_{2}\triangleleft\overline{K}\), or \(\overline{K}\) contains a nontrivial subnormal subgroup of \(\overline{G}\).

First assume that the former case occurs. Note that \(\overline{K}=KZ(M)/Z(M)\cong K/(K\cap Z(M))\) and \(K\cap H\). Then 
the composition factor set for \(\overline{K}\) is contained in that for \(H\), which is \(\{\PSp_{2m}(q), 2\}\). Since \(B_{i}\neq \PSp_{2m}(q)\) for \(i=1,2\), we deduce that \(|B_{i}|\) are some powers of \(2\) for \(i=1,2\) and therefore \(|\PSp_{2m}(q)|\) is some power of \(2\), which is a contradiction.

Thus the latter case occurs and \(\overline{K}\) contains some nontrivial subnormal subgroup of \(\overline{G}\). In other words, \(\overline{M_{i}}\leqslant\overline{K}\) for \(i=1\) or \(2\). Without loss of generality, \(\overline{M_{1}}\leqslant \overline{K}\) and so \(M_{1}\leqslant K\). Moreover, since \(K\cap M/(K\cap M_{2})\cong\varphi_{1}(K\cap M)\), the group \(K\cap M=(K\cap M_{2})\times M_{1}\). If \(\pi(K)=2\), then \(M_{2}=K\cap M_{2}\cong K\cap M_{1}= M_{1} \). This would imply that \(\m_{K}(\PSp_{2m}(q))=2\), contradicting Claim~1. Thus \(\pi(K)=1\) and \(K\leqslant M\). Since \(K=(K\cap M_{2})\times M_{1}\), we have 
\[
\{\PSp_{2m}(q)\}=\ICF(K)=\ICF(K\cap M_{2})\cup\ICF(M_{1}).
\]
Moreover, since \[1=\m_{K}(\PSp_{2m}(q))=\m_{K\cap M_{2}}(\PSp_{2m}(q))+\m_{M_{1}}(\PSp_{2m}(q))\] and \(\m_{M_{1}}(\PSp_{2m}(q))=1\), it follows that \(\m_{K\cap M_{2}}(\PSp_{2m}(q))=0\). Hence \(\ICF(K\cap M_{2})=\varnothing\) and so \(K\cap M_{2}\) is soluble. Therefore \(K^{(\infty)}=M_{1}\). Since \(H\cong K\), we have \(H^{(\infty)}\cong M_{1}\cong\Sp_{2m}(q)\). Since \(H\cong(H\cap M)=\pi(H)\cong\C_{2}\), we have \(H^{(\infty)}\leqslant H\cap M\leqslant M\cong\Sp_{2m}(q)^{2}\). Note that for \(Q\leqslant M\cong\Sp_{2m}(q)^{2}\) such that \(Q\cong\Sp_{2m}(q)\), either \(Q=M_{i}\) for \(i=1\) or \(2\), or there exists \(\alpha\in\Aut(\Sp_{2m}(q))\) such that \(Q=\{(x,x^{\alpha})| x\in M_{1}\}\). If \(H^{(\infty)}=M_{i}\) for \(i=1\) or \(2\), then since \(\pi(H)\) acts transitively on \(\{M_{1}, M_{2}\}\), it follows that \(M_{3-i}\leqslant H\) and so \(M=M_{1}\times M_{2}\leqslant H\). This implies that 
\(\m_{H}(\PSp_{2m}(q))=2\), which is a contradiction to Claim~1. Hence there exists \(\alpha\in\Aut(\Sp_{2m}(q))\) such that \(H^{(\infty)}=\{(x,x^{\alpha}) \mid x\in M_{1}\}\), proving the lemma.
\end{proof}

\begin{lemma}\label{Sp23}
Suppose that Hypothesis \ref{wrspsk} holds, \((m,q)=(1,3)\) and \(k\geqslant 5\). Then \(k=6\) and \(\PSL_{2}(5)\triangleleft \pi(H)\leqslant \PGL_{2}(5)\) with \(\A_{5}\triangleleft\ \pi(K)\leqslant \Sy_{5}\).
\end{lemma}
\begin{proof}
Since \(H\cong K\), it follows that \(\ICF(H)=\ICF(K)\). Note that \(\ICF(H)=\ICF(H\cap M)\cup\ICF(\pi(H))\) and \(\ICF(K)=\ICF(K\cap M)\cup\ICF(\pi(K))\). This together with the fact that \(M\) is soluble implies \(\ICF(\pi(H))=\ICF(\pi(K))\). We therefore deduce from \cite[Lemma 2.5]{linear} and Lemma \ref{k=5} that either both of \(\pi(H)\) and \(\pi(K)\) contain \(\A_{k}\), or \(k=6\), \(\PSL_{2}(5)\triangleleft \pi(H)\leqslant \PGL_{2}(5)\) and \(\A_{5}\triangleleft \pi(K)\leqslant \Sy_{5}\)(note that Hypothesis \ref{wrspsk} has assumed that \(\pi(H)\) is transitive). Thus in order to prove the result it suffices to show that the former case cannot occur. 

Suppose for a contradiction that \(\A_{k}\) is contained in \(\pi(H)\) and \(\pi(K)\). Then since \(M\) is soluble and has no section isomorphic to \(\A_{k}\), we deduce by \cite[Lemma 3.6]{linear} that \(\pi(H)=\pi(K)=\Sy_{k}\). This together with \cite[Lemma 3.5]{linear} implies that \(\varphi_{1}(H\cap M)\cong\cdots\cong\varphi_{k}(H\cap M)\), and \(\varphi_{1}(K\cap M)\cong\cdots\cong\varphi_{k}(K\cap M)\), and \(\{2,3\}=\Pi(\PSp_{2}(3))\subseteq\Pi(\varphi_{1}(H\cap M))\cap\Pi(\varphi_{1}(K\cap M))\).

\medskip
\noindent
\textit{Claim~1:} \(\varphi_{1}(H\cap M)\cong\varphi_{1}(K\cap M)\cong \Sp_{2}(3)\).

\medskip
\noindent
Since \(\{2,3\}\subseteq \Pi(\varphi_{1}(H\cap M))\), we have that either \(\varphi_{1}(H\cap M)\cong 2.3\) or \(\Sp_{2}(3)\). Suppose that the former case occurs. Then \(|\varphi_{1}(H\cap M)|_{2}=2\) and so \(|H\cap M|_{2}\leqslant 2^{k}\). This implies that
\[
    |H|^{2}_{2}=|H\cap M|_{2}^{2}|\pi(H)|_{2}^{2}\leqslant 2^{2k}(k!)_{2}^{2}<2^{3k}(k!)_{2}.
\]
However, since \(|G|_{2}=|M|_{2}(k!)_{2}=2^{3k}(k!)_{2}\), this implies that \(|H|_{2}^{2}<|G|_{2}\), which is impossible. Hence \(\varphi_{1}(H\cap M)\cong \Sp_{2}(3)\). 

The same argument can be applied on \(\varphi_{1}(K\cap M)\). Thus the claim is proved.

Note that \(M\) has a unique Sylow \(2\)-subgroup and we denote it by \(Q\). Then \(Q\cong (\C_{2}.\C_{2}^{2})^{k} = Q_8^k\). For any subgroup \(N\leqslant G\), we denote by \(\overline{N}:=NQ/Q\). Then \(\overline{M}\cong\C_{3}^{k}\) and  \(\C_{3}\wr\Sy_{k}\cong\overline{G}=\overline{H}\,\overline{K}\).
Since \(\pi(H)=\pi(K)=\Sy_{k}\) and \(\varphi_{1}(H\cap M)\cong\varphi_{1}(K\cap M)\), we find that \(\Sy_{k}= \pi(\overline{H})=\pi(\overline{K})\) and \(\varphi_{1}(\overline{H}\cap\overline{M})\cong\varphi_{1}(\overline{K}\cap\overline{M})\cong\C_{3}\). 

 Since \(\pi(\overline{H})=\Sy_{k}\), we see that \(\overline{H}\cap\overline{M}\) is a submodule of the permutation module \(\overline{M}\) of \(\Sy_{k}\) over a field of characteristic $3$. Similarly, \(\overline{K}\cap\overline{M}\) is a submodule of the permutation module \(\overline{M}\) of \(\Sy_{k}\). By \cite[Lemma 2]{mortimer}, the only submodules of the permutation module \(\overline{M}\) of \(\Sy_{k}\) are \(0\), \(\overline{M}\), a submodule of dimension \(1\) and a submodule of dimension \(k-1\). 

\medskip
\noindent
\textit{Claim~2:} \(\overline{M}\leqslant\overline{H}\).

\medskip
\noindent
Suppose that Claim~2 is false. Then \(\overline{H}\cap\overline{M}\) is either \(0\), a submodule of dimension \(1\) or a submodule of dimension \(k-1\).

First assume that \(\overline{H}\cap\overline{M}\) is \(0\) or a submodule of dimension \(1\). Then \(|\overline{H}\cap \overline{M}|_{3}\leqslant 3\) and 
\[
|\overline{H}|_{3}=|\overline{H}\cap\overline{M}|_{3}|\pi(\overline{H})|_{3}\leqslant 3|\pi(\overline{H})|_{3}\leqslant3(k!)_{3}.
\]
This implies that \(|H|_{3}^{2}= |Q|_{3}^{2}|\overline{H}|_{3}^{2}\leqslant 3^{2}(k!)_{3}\).
On the other hand, \(|G|_{3}=|M|_{3}(k!)_{3}=3^{k}(k!)_{3}\). Since \(|H|_{3}^{2}\geqslant|G|_{3}\), we have 
\(
3^{k}(k!)_{3}\leqslant 3^{2}(k!)_{3}^{2}.
\)
This implies that  \((k!)_{3}\geqslant 3^{k-2}\), which is impossible as \(k\geqslant5\).

Hence \(\overline{H}\cap\overline {M}\) is a submodule of \(\overline{M}\) with dimension \(k-1\). Indeed there is a unique submodule \(L\) of \(\overline{M}\) of dimension \(k-1\), namely
\[
    L=\{(x_{1},\ldots, x_{k})| x^{3}=e, x_{i}\in \{x, x^{2}, e\}, \prod_{i=1}^{k}x_{i}=e\}.
\]
Thus \(L=\overline{H}\cap\overline{M}\). Since \(H\cong K\) and \(|Q|_{3}=1\), we have that \(|\overline{H}|_{3}=|\overline{K}|_{3}\). Recall that \(\Sy_{k}= \pi(\overline{H})= \pi(\overline{K})\). Hence \(|\pi(\overline{H})|_{3}=|\pi(\overline{K})|_{3}=(k!)_{3}\) and therefore \(|\overline{H}\cap\overline{M}|_{3}=|\overline{K}\cap\overline{M}|_{3}\). Note that \(\overline{K}\cap\overline{M}\) is also a submodule of \(\overline{M}\). Thus \(\overline{K}\cap\overline{M}\) is a submodule of dimension \(k-1\). In particular, \(L=\overline{K}\cap\overline{M}\), and we note that $L$ is \(\overline{G}\)-invariant. Thus we obtain that \(\overline{H}/L\cong\pi(\overline{H})\) and \(\overline{K}/L\cong\pi(\overline{K})\).  This further implies that \(\overline{H}\cap\overline{K}\cap \overline{M}=L\) and \(\overline{H}\cap\overline{K}=L.\pi(\overline{H}\cap\overline{K})\).

Since \(\overline{G}=\overline{H}\,\overline{K}\) and \(\pi(\overline{H})=\pi(\overline{K})=\Sy_{k}\), we find that for any prime \(p\)
\begin{equation}\label{19}
    |\overline{G}|_{p}=\frac{|\overline{H}|_{p}|\overline{K}|_{p}}{|\overline{H}\cap\overline{K}|_{p}}=\frac{|L|_{p}|\pi(\overline{H})|_{p}|L|_{p}|\overline{K}|_{p}}{|L|_{p}|\pi(\overline{H}\cap\overline{K})|_{p}}=\frac{|L|_{p}(k!)_{p}^{2}}{|\pi(\overline{H}\cap\overline{K})|_{p}}.
\end{equation}
On the other hand, \(|\overline{G}|_{p}=|\overline{M}|_{p}(k!)_{p}\) and \(|\overline{M}|_{p}=|L|_{p}(3)_{p}\). This together with equation \eqref{19} implies that
\[
    (k!)_{p}=(3)_{p}|\pi(\overline{H}\cap\overline{K})|_{p}.
\]
 Since this holds for all primes $p$, we conclude that \(k!=3|\pi(\overline{H}\cap\overline{K})|\). However \(\Sy_{k}\) does not have a subgroup of index \(3\) since \(k\geqslant7\) or \(k=5\), which is a contradiction and proving the claim.

With a similar argument to that in Claim~2 we also deduce that \(\overline{M}\leqslant\overline{K}\) and so \(\overline{M}\leqslant\overline{H}\cap\overline{K}\). Hence \(|H\cap M|_{3}=|M|_{3}=3^{k}\).

Let \(X_{0}:=H\cap M\) and \(X_{j}:=\ker\varphi_{j}\cap H\) for \(j\geqslant1\). Then \(X_{j}=(\prod_{i\neq j}M_{i})\cap H\). Thus \(\bigcap_{i=0}^{j}X_{i}=(\prod_{i=j+1}^{k}M_{i})\cap H\). Since \(\prod_{i=j+1}^{k} M_{i}\) is normal in \(M\), we have that \(\bigcap_{i=0}^{j}X_{i}\cap H\) is also normal in \(M\cap H\) for \(j\geqslant0\). Hence, for \(j\geqslant0\), \(\varphi_{j+1}(\bigcap_{i=0}^{j}X_{i})\triangleleft\varphi_{j+1}(H\cap M)\cong\Sp_{2}(3)\) by Claim~1. Note that \(|\bigcap_{i=0}^{j}X_{i}|=|\bigcap_{i=0}^{j+1}X_{i}|\cdot|\varphi_{j+1}(\bigcap_{i=0}^{j}X_{i})|\). We therefore deduce inductively that
\[
    |H\cap M|=|\varphi_{1}(H\cap M)||\varphi_{2}(X_{1})|\ldots|\varphi_{k}(\bigcap_{i=1}^{k-1}X_{i})|.
\]
Since \(|H\cap M|_{3}=3^{k}\) and \(|\varphi_{j+1}(\bigcap_{i=0}^{j}X_{i})|_{3}\leqslant3\) for each \(j\geqslant0\), we find that 
\[
|\varphi_{1}(H\cap M)|_{3}=|\varphi_{2}(X_{1})|=\cdots=|\varphi_{k}(\bigcap_{i=1}^{k-1}X_{i})|_{3}=3.
\]
Recall that \(\varphi_{j+1}(\cap_{i=0}^{j}X_{i})\) is normal in \(\varphi_{j+1}(H\cap M)\cong \Sp_{2}(3)\) for each \(j\geqslant0\). This implies that \(\varphi_{j+1}(\cap_{i=0}^{j}X_{i})\cong\Sp_{2}(3)\) for each \(j\geqslant0\). Thus \(|H\cap M|=|\Sp_{2}(3)|^{k}\). This implies that \(|H|=|H\cap M|(k!)=|\Sp_{2}(3)|^{k}(k!)=|G|\) and \(H=G\), which is a contradiction, completing the proof.
\end{proof}
\subsection{Factorisations of almost simple classical groups}
 The following results are based on \cite[Theorem A]{LPS}, \cite{LX} and \cite{LWX}, which classify the core-free factorisations \(G=AB\) of an almost simple group \(G\) with its socle \(L\) being a classical group. The memoir \cite{LPS} classifies the factorisations of \(G\), for which both \(A\) and \(B\) are maximal core-free subgroups of \(G\). The paper \cite{LX} investigates the situation where both \(A\) and \(B\) are insoluble and \cite{LWX} gives a list of the factorisations such that at least one of \(A\) or \(B\) is soluble. Note that \cite{LX} and \cite{LWX} are based on \cite{LPS}. Combining the results of \cite{LX} and \cite{LWX} gives a full list of factors \(A\) and \(B\), which embodies the possible insoluble composition factor sets of \(A\) and \(B\), namely \(\ICF(A)\) and \(\ICF(B)\), respectively. 

\begin{lemma}\label{maxcore}
Let \(G\) be an almost simple group with socle \(L\). Suppose that \(G=AB\) is a core-free factorisation and let \(G^{*}=AL\cap BL\), so \(L\trianglelefteq G^{*}\leqslant G\). Then there exist core-free maximal subgroups \(X\) and \(Y\) of \(G^{*}\) such that \(G^{*}=XY\) and \(A\cap L\leqslant X\cap L\) and \(B\cap L\leqslant Y\cap L\).
\end{lemma}
\begin{proof}
By \cite[Lemma 2]{1996}, for \(A^{*}:=A\cap G^{*}\) and \(B^{*}=B\cap G^{*}\), we have \(G^{*}=A^{*}B^{*}\), and by \cite[Corollary 3]{1996} there exist core-free maximal subgroups \(X^{*}\) and \(Y^{*}\) of  \(G^{*}\) with \(A^{*}\leqslant X^{*}\) and \(B^{*}\leqslant Y^{*}\), such that \(G^{*}=X^{*}Y^{*}\). By the definition of \(A^{*}\) and \(B^{*}\), \(A\cap L=A^{*}\cap L\leqslant X^{*}\cap L\) and \(B\cap L=B^{*}\cap L\leqslant Y^{*}\cap L\), proving the result.
\end{proof}
\begin{lemma}\label{sp}
Let \(G\) be an almost simple group with socle \(L=\PSp_{2m}(q)\) such that \(q=p^{f}\) is odd and \(m\geqslant2\). Suppose that \(G=AB\) is a core-free factorisation. Then, interchanging \(A\) and \(B\) if necessary, both of the following hold:
\begin{description}
\item[(i)] \(\ppd(p,2mf)\neq\varnothing\) and for each \(r\in\ppd(p,2mf)\), we have \(r\nmid|B|\).
\item[(ii)] \((L, \ICF(A), \ICF(B))\) are as in Table \ref{PSp(2m,q)}.
\end{description}

\begin{table}
    \centering
    \begin{tabular}{c|c|c|c}
    \hline
        L& \(\ICF(A)\)&\(\ICF(B)\)&Conditions  \\
         \hline
         \(\PSp_{2m}(q)\)&\(\{\PSp_{2a}(q^{b})\}\)& \(\{\PSp_{2m-2}(q)
         \}\)&\(m=ab, b>1\)\\
         %\hline
        
         \(\PSp_{6}(3)\)&\(\{\PSL_{2}(13)\}\)&\(\{\PSp_{4}(3)\}\)&\\
         %\hline
         \(\PSp_{6}(3)\)&\(\{\PSL_{2}(27)\}\)&\(\{\A_{5}\}, \varnothing\)& \\
         %\hline
         \(\PSp_{4}(q)\)&\(\{\PSp_{2}(q^{2})\}\)&\(\varnothing\)&\\
         %\hline
         \(\PSp_{4}(3)\)&\(\{\A_{6}\}, \{\A_{5}\}, \varnothing\)&\(\varnothing\)&\\
         \hline
    \end{tabular}
    \caption{Core-free factorisations \(G=AB\) for \(\Soc(G)=\PSp_{2m}(q)\) with \(q\) odd and \(m\geqslant2\)}
    \label{PSp(2m,q)}
\end{table}
\end{lemma}
\begin{proof}
Since \(p\) is odd and \(m\geqslant 2\), we deduce by Corollary \ref{existppd} that \(\ppd(p,2mf)\neq\varnothing\). Let \(r\in\ppd(p,2mf)\). Then \(r>2mf\). On the other hand, since \(G=AB\) is a core-free factorisation, we deduce by Lemma \ref{maxcore} that there exists \(G^{*}\) with \(L\triangleleft G^{*}\leqslant G\) together core-free maximal subgroups \(X\) and \(Y\) of \(G^{*}\) such that \(G^{*}=XY\), \(A\cap L\leqslant X\cap L\) and \(B\cap L\leqslant Y\cap L\). All such pairs \((X\cap L, Y\cap L)\), up to interchanging \(X\cap L\) and \(Y\cap L\), have been given in \cite[Table 1, 2, 3]{LPS}, and so \(L\), \(X\cap L\), \(Y\cap L\) are one of the following:
\begin{description}
\item[(a)] \(L=\PSp_{2m}(q)\), \(X\cap L=\PSp_{2a}(q^{b}).b\) and \(Y\cap L=P_{1}\) where \(m=ab\) and \(b\) is prime;
\item[(b)]\(L=\PSp_{4}(3)\), \(X\cap L=2^{4}.\A_{5}\) and \(Y\cap L=P_{i}\) for \(i=1,2\);
\item[(c)]\(L=\PSp_{6}(3)\), \(X\cap L=\PSL_{2}(13)\) and \(Y\cap L=P_{1}\);
\item[(d)]\(L=\PSp_{4}(3)\cong\PSU_{4}(2)\), \(X\cap L\leqslant 2^4.\A_{5}\), \(Y\cap L\leqslant P_{1}\) or \(3^{3}.\Sy_{4}\);
\item[(e)]\(L=\PSp_{4}(3)\cong\PSU_{4}(2)\), \(X\cap L\leqslant\PSp_{4}(2)\), \(Y\cap L\leqslant P_{1}\).
\end{description}
Note that in cases (d) and (e), \(P_{1}\) refers to a parabolic subgroup of \(\PSp_{4}(3)\). Observe that \(|Y\cap L|_{r}=1\) for all cases (a), (b) and (c). Since \(\Pi(\Out(L))\subseteq\{2\}\cup\Pi(f)\) and \(r>2mf\), we thence deduce that \(|\Out(L)|_{r}=1\). Consequently, \[|B|_{r}\leqslant|Y|_{r}\leqslant|Y\cap L|_{r}|\Out(L)|_{r}=|Y\cap L|_{r}=1,\] proving (i).

Note that \(\ICF(A)=\ICF(A^{(\infty)})=\ICF(A\cap L)\) and \(\ICF(B)=\ICF(B^{(\infty)})=\ICF(B\cap L)\). All of the possible triples \((L, A^{(\infty)}, B^{(\infty)})\) for \(A\) and \(B\) both insoluble and all of the possible triples \((L, A, B)\) for at least one of \(A\) or \(B\) soluble have been classified in \cite[Theorem 8.1]{LWX} and \cite[Theorem 1.1]{LX}, respectively. Thus part (ii) is an immediate result of that.
\end{proof}

\begin{proposition}\label{odd}
Let \(G\) be an almost simple group with socle \(L=\POmega_{m}^{\epsilon}(q)\), where \(q=p^{f}\) is odd, \(m\geqslant 4\), \(\epsilon\in\{+, -, \circ\}\) and \((m,\epsilon)\neq (4, +), (5, \circ), (8,+)\). Suppose that \(G=AB\) is a core-free factorisation. Then the following hold:
\begin{description}
\item[(i)] Let \[
    d:=\begin{cases}
        m-1&\text{if \(m\) is odd}\\ m &\text{if \(m\) is even, \(\epsilon=-\)}\\ m-2& \text{if \(m\) is even, \(\epsilon=+\)}.
    \end{cases}
\] Then \(\ppd(p, df)\neq\varnothing\) and let \(r\in\ppd(p, df)\). Then, interchanging \(A\) and \(B\) if necessary, \(r\) divides \(|A|\), and if also \(r\) divides \(|B|\), then \(A\cap L\) and \(B\cap L\) are contained in \(X\cap L\) and \(Y\cap L\) as listed in Table \ref{tab:YLneq1}. 
\item[(ii)] For \((m,\epsilon)=(6,+)\) let \(s\in\ppd(p,3f)\). Then, interchanging \(A\) and \(B\) if necessary, \(|A|_{s}=1\).
\item[(iii)]  \((L, \ICF(A), \ICF(B))\) are as in Table \ref{O(2m+1,q)}. In particular, \(\ICF(A)\cap\ICF(B)=\varnothing\) unless \((L,\ICF(A),\ICF(B))=(\POmega_{4}^{-}(3), \{\A_{5}\}, \{\A_{5}\})\). Moreover, if \(m\geqslant7\), then there exists \(T\in\ICF(A)\) such that \(|T|_{r}>1\) (\(r\) as defined in (i)).

\end{description}
\begin{table}[]
    \centering
    \begin{tabular}{c|c|c|c}
    \hline
     \(L\)&\(X\cap L\leqslant\) & \(Y\cap L\leqslant\)  & Comment  \\\hline
      \(\Omega_{7}(q)\)&\(\G_{2}(q)\)&\(N_{1}^{-}\) & case (4)\\ %\hline
      \(\Omega_{7}(3)\)&\(\G_{2}(3)\)& \(\Sp_{6}(2)\) or \(\Sy_{9}\)& case (5)\\ %\hline
      \(\POmega_{4}^{-}(3)\cong\PSL_{2}(9)\)&\(\A_{5}\) & \(\PSL_{2}(5)\)& case (10)\\ %\hline
      \(\POmega_{m}^{+}(q)\)&\(N_{1}\)&\(\GU_{m/2}(q).2\), \(m/2\) even& case (15)\\\hline
    \end{tabular}
    \caption{Core-free factorisation \(G=XY\) for \(\Soc(G)=\POmega_{m}^{\epsilon}(q)\) with \(|Y\cap L|_{r}>1\).}
    \label{tab:YLneq1}
\end{table}
\begin{table}
    \centering
    \begin{tabular}{c|c|c|c|c}
    \hline
        &\( L\) & \(\ICF(A)\)&\(\ICF(B)\)& Conditions \\
        \hline
       (1) &\(\Omega_{m}(q)\)&\(\{\POmega_{m-1}^{-}(q)\}\) & \(\varnothing, \{\PSL_{a}(q^{b})\}, \{\PSp_{a}(q^{b})\}\)&\(ab=m-1\) \\
       &&&& \(m\) odd\\
        %\hline
       (2)& \(\Omega_{7}(q)\)&\(\{\G_{2}(q)'\}\)&\(\{\PSL_{4}(q)\}, \{\PSU_{4}(q)\}\),&\\
        &&&\(\{\PSp_{4}(q)\}\), \(\{\PSL_{2}(q^{2})\}\)&\\
        %\hline
       (3)& \(\Omega_{25}(3^{f})\)&\(\{\POmega_{24}^{-}(3^{f})\}\)&\(\{\F_{4}(3^{f})\}\)&\\
        %\hline
       (4)& \(\Omega_{13}(3^{f})\)&\(\{\POmega_{12}^{-}(3^{f})\}\)&\(\{\PSp_{6}(3^{f})\}\)&\\
        %\hline
       (5)& \(\Omega_{13}(3)\)&\(\{\POmega_{12}^{-}(3)\}\)&\(\{\PSL_{2}(13)\}\)&\\
        %\hline
      (6)&  \(\Omega_{9}(3)\)&\(\{\POmega_{8}^{-}(3)\}\)&\(\{\A_{5}\}, \varnothing \)&\\
        %\hline
     (7)&   \(\Omega_{7}(3^{f})\)&\(\{\PSU_{3}(3^{f})\}, \{{}^{2}\G_{2}(3^{f})'\}\)&\(\{\PSL_{4}(3^{f})\}\)&\(f\) odd\\
        %\hline
       (8)& \(\Omega_{7}(3)\)&\(\{\G_{2}(3)'\}\)&\(\{\A_{5}\}, \{\A_{6}\}\)&\\

        %\hline
      (9)&  \(\Omega_{7}(3)\)&\(\{\A_{9}\},\{\Sp_{6}(2)\}\)&\(\{\PSL_{3}(3)\},\{\PSL_{4}(3)\}\),&\\
        &&&\(\{\G_{2}(3)'\}\)&\\
        %\hline
     (10) &  \(\Omega_{7}(3)\)&\(\{\A_{7}\}, \{\A_{8}\}, \{\A_{9}\}\)&\(\{\PSL_{3}(3)\}\)&\\
        &&\(\{\PSL_{3}(4)\}, \{\Sp_{6}(2)\}\)&&\\
        %\hline
       (11)& \(\Omega_{7}(3)\)&\(\{\G_{2}(3)'\}, \{\Sp_{6}(2)\}\)&\(\varnothing\)&\\
        %\hline
        
     (12)&   \(\POmega_{4}^{-}(q)\cong \PSL_{2}(q^{2})\)&\(\varnothing\)&\(\varnothing\)&\\
        %\hline
     (13)&   \(\POmega_{4}^{-}(3)\cong\PSL_{2}(9)\)&\(\{\A_{5}\}\)&\(\{\A_{5}\}\)&\\
        %\hline
     (14)&   \(\POmega_{6}^{-}(q)\cong\PSU_{4}(q)\)&\(\{\PSU_{3}(q)\}\)&\(\{\PSL_{2}(q^{2})\},\)&\\
        &&&\(\{\PSp_{4}(q), \varnothing\)&\\
        %\hline
      (15)&  \(\POmega_{6}^{-}(3)\cong\PSU_{4}(3)\)&\(\{\PSL_{3}(4)\}\)&\(\{\A_{5}\}, \{\PSp_{4}(3)\},\)&\\
        &&&\(\{\PSL_{2}(9)\}, \varnothing\)&\\
        %\hline
    (16)&    \(\POmega_{6}^{-}(3)\cong\PSU_{4}(3)\)&\(\{\PSL_{2}(7)\}\)&\(\{\A_{6}\}\)&\\
        %\hline
     (17)&   \(\POmega_{6}^{-}(5)\cong\PSU_{4}(5)\)&\(\{\A_{7}\}\)&\(\{\PSL_{2}(25)\}\)&\\
        %\hline
      (18)&  \(\POmega_{m}^{-}(q)\)&\(\{\PSU_{m}(q)\}\)&\(\{\Omega_{m-1}(q)\}, \{\POmega_{m-2}^{-}(q)\}\)&\(m/2\) odd\\
        %\hline
      (19)&  \(\POmega_{6}^{+}(q)\cong\PSL_{4}(q)\)&\(\{\PSL_{2}(q^{2})\}, \{\PSp_{4}(q)\}, \varnothing\)&\(\{\PSL_{3}(q)\}\)&\\
        %\hline
      (20)&  \(\POmega_{6}^{+}(q)\cong\PSL_{4}(q)\)&\(\{\PSp_{4}(q)\}\)&\(\varnothing\)&\\
        %\hline
     (21)&   \(\POmega_{6}^{+}(3)\cong\PSL_{4}(3)\)&\(\{\A_{5}\}\)&\(\{\PSL_{3}(3)\}\)&\\
        %\hline
      (22)&  \(\POmega_{6}^{+}(3)\cong\PSL_{4}(3)\)&\(\{\A_{6}\}\)&\(\varnothing\)&\\
        %\hline
      (23)&  \(\POmega_{m}^{+}(q)\)&\(\{\Omega_{m-1}(q)\}\)&\(\{\PSL_{a}(q^{b})\}, \{\PSp_{a}(q^{b})\}\)&\\
        &&&\(\{\PSU_{m/2}(q)\}, \{\PSp_{m/2}(q)\}, \varnothing\)&\(m=ab\)\\
        &&&\(\{\PSL_{2}(q), \PSp_{m/2}(q)\}\)&\\
        %\hline
       (24)& \(\POmega_{m}^{+}(q)\)&\(\{\POmega_{m-2}^{-}(q)\}\)&\(\{\PSL_{m/2}(q)\}\)&\\
        %\hline
        (25)&\(\POmega_{m}^{+}(q)\)&\(\{\PSU_{m/2}(q)\}\)&\(\{\POmega_{m-2}^{+}(q)\}\)&\(m/2\) even\\
        %\hline
       (26)& \(\POmega_{16}^{+}(q)\)&\(\{\Omega_{15}(q)\}\)&\(\{\Omega_{9}(q)\}\)&\\
        %\hline
       (27)& \(\POmega_{12}^{+}(3)\)&\(\{\Omega_{11}(3)\}\)&\(\{\PSL_{2}(13)\}\)&\\
        \hline
    \end{tabular}
    \caption{Core-free factorisations \(G=XY\) for \(\Soc(G)=\POmega_{m}^{\epsilon}(q)\), \(q=p^{f}\) odd}
    \label{O(2m+1,q)}
\end{table}
\end{proposition}
\begin{proof}
Note that \(d\geqslant m-1\geqslant 3\) for \(\epsilon=-\) or \(\circ\) and \(d\geqslant m-2\geqslant 4\) for \(\epsilon=+\). Therefore \(d\geqslant 3\). Since \(p\) is odd we conclude by Corollary \ref{existppd} that \(\ppd(p,df)\neq\varnothing\). Let \(r\in\ppd(p,df)\). Then \(r\in\Pi(\POmega_{m}^{\epsilon}(q))\) and  \(r>df\). Since \(G=AB\), we may assume that \(r\) divides \(|A|\). Further, since \(G=HK\) is a core-free factorisation, by Lemma \ref{maxcore} there exists \(G^{*}\) such that \(L\triangleleft G^{*}\leqslant G\) and core-free maximal subgroups \(X\) and \(Y\) of \(G^{*}\) such that \(G^{*}=XY\) and \(A\cap L\leqslant X\cap L\) and \(B\cap L\leqslant Y\cap L\). Such core-free maximal factorisations  have been classified in \cite[Theorem A, Table 1,2,3]{LPS}, and all pairs \((X\cap L, Y\cap L)\) satisfy one of the following 18 cases:
\begin{description}
\item[(1)] \(L=\Omega_{m}^{\circ}(q)\), \(X\cap L\leqslant N_{1}^{-}\) and \(Y\cap L\leqslant P_{(m-1)/2}\);
\item[(2)] \(L=\Omega_{25}(3^{e})\), \(X\cap L\leqslant N_{1}^{-}\) and \(Y\cap L\leqslant\F_{4}(3^{e})\);
\item[(3)] \(L=\Omega_{13}(3^{e})\), \(X\cap L\leqslant N_{1}^{-}\) and \(Y\cap L\leqslant \PSp_{6}(3^{e}).2\);
\item[(4)] \(L=\Omega_{7}(q)\), \(X\cap L\leqslant \G_{2}(q)\) and \(Y\cap L\leqslant P_{1}, N_{i}^{\pm}\) for \(i=1,2\).
\item[(5)]\(L=\Omega_{7}(3)\), \(X\cap L\leqslant \G_{2}(3)\) and \(Y\cap L\leqslant \Sp_{6}(2)\) or \(\Sy_{9}\);
\item[(6)]\(L=\Omega_{7}(3)\), \(X\cap L\leqslant \Sy_{9}\) and \(Y\cap L\leqslant N^{+}_{1}\) or \(P_{3}\);
\item[(7)]\(L=\Omega_{7}(3)\), \(X\cap L\leqslant\Sp_{6}(2)\) and \(Y\cap L\leqslant N_{1}^{+}\) or \(P_{3}\);
\item[(8)] \(L=\Omega_{7}(3)\), \(X\cap L\leqslant2^{6}.\A_{7}\) and \(Y\cap L\leqslant P_{3}\);
\item[(9)] \(L=\POmega_{4}^{-}(q)\cong \PSL_{2}(q^{2})\), \(X\cap L\leqslant \D_{q^2+1}\) and \(Y\cap L\leqslant \E_{q^{2}}:\C_{\frac{q^2-1}{2}}\), where \(\E_{q^{2}}\) is an elementary abelian group of order \(q^{2}\);
\item[(10)] \(L=\POmega_{4}^{-}(3)\cong\PSL_{2}(9)\cong \A_{6}\), \(X\cap L=\A_{5}\) and \(Y\cap L=\PSL_{2}(5)\cong\A_{5}\);
%\item[(11)] \(L=\Pp\Omega_{6}^{-}(q)\cong\PSU_{4}(q)\), \(X\cap L\leqslant N_{1}\) and \(Y\cap L\leqslant P_{2}\) or \(\PSp_{4}(q)\);
\item[(11)] \(L=\POmega_{6}^{-}(3)\cong\PSU_{4}(3)\), \(X\cap L\leqslant \PSL_{3}(4)\) and \(Y\cap L\leqslant \PSp_{4}(3)\) or \(P_{i}\) for \(i=1,2\). Note that the \(P_i\) subgroup of \(\POmega_{6}^{-}(3)\) is also the \(P_{3-i}\) subgroup of \(\PSU_{4}(3)\) for \(i=1,2\) by \cite[pp.52]{Atlas};
\item[(12)] \(L=\POmega_{m}^{-}(q)\), \(X\cap L\leqslant \GU_{m/2}(q)\) and \(Y\cap L\leqslant P_{1}\) or \(N_{1}\) with \(m\geqslant 6\) and \(m\) is odd. Note that when \(m=6\) and \(L=\POmega_{6}^{-}(q)\), which is isomorphic to \(\PSU_{4}(q)\). Then by observing the factorisations of an almost simple group with socle \(\PSU_{4}(q)\) given in \cite[p.11]{LPS} we have one factorisation such that \(X\cap L\leqslant N_1\), \(Y\cap L\leqslant P_{2}\) or \(\PSp_{4}(q)\). However, by \cite[Table 3.2]{stephen} we have the \(P_2\) subgroup of \(\PSU_4(q)\) is a \(P_1\) subgroup of \(\POmega_{6}^{-}(q)\) and \(\PSp_{4}(q)\) is a \(N_1\) subgroup of \(\POmega_{6}^-(q)\), and by \cite[Table 8.10]{holt} the \(N_1\) subgroup of \(\PSU_{4}(q)\) is \(\GU_3(q)\). Hence that case can be incorporated into case (12);
\item[(13)] \(L=\POmega_{6}^{+}(q)\cong\PSL_{4}(q)\), \(X\cap L\leqslant \GL_{2}(q^{2}).2\) and \(Y\cap L\leqslant P^{\pm}_{3}\). Note that this case arises from the factorisation of the almost simple group with socle \(\PSL_{4}(q)\) given by \(X\cap L\leqslant \GL_{a}(q^b).b, Y\cap L\leqslant P_1\) or \(P_3\) in \cite[p.10, Table 1]{LPS}. However by observing \cite[Table 8.8, Table 8.31]{holt} we deduce that the \(P_1\) and \(P_3\) subgroups of \(\PSL_{4}(q)\) are the  \(P_3^+\) and \(P_3^-\) subgroups of \(\POmega_{6}^{-}(q)\);
\item[(14)] \(L=\POmega_{6}^{+}(q)\cong\PSL_{4}(q)\), \(X\cap L\leqslant\PSp_{4}(q)\) and \(Y\cap L\leqslant P^\pm_{3}\) or \(P_2\). Note that this case arises from the factorisation of an almost simple group with socle \(\PSL_4(q)\) such that \(X\cap L\leqslant\PSp_4(q), Y\cap L\leqslant P_1, P_3\) or \(\Stab(V_1\oplus V_3)\) given by \cite[p.10]{LPS}. As stated in case (13), the \(P_1\) and \(P_3\) subgroups of \(\PSL_{4}(q)\) are \(P_{3}^{\pm}\) subgroups of \(\POmega_{6}^{-}(q)\). Moreover, \(\Stab(V_1\oplus V_3)\) is \(P_{1,3}\) as a subgroup of \(\PSL_4(q)\), and  by observing \cite[Table 8.8, Table 8.31]{holt} we deduce that the \(P_{1,3}\) subgroup is a \(P_2\) subgroup of \(\POmega_6^{+}(q)\); 
\item[(15)] \(L=\POmega_{m}^{+}(q)\) with \(m\geqslant10\), \(X\cap L\leqslant N_{1}\) and \(Y\cap L\leqslant P_{i}\) for \(i=m/2, (m-2)/2\), or \(\GU_{m/2}(q).2\) (\(m/2\) even), or \((\PSp_{2}(q)\otimes \PSp_{m/2}(q)).2\) (\(m\) even), or \(\GL_{m/2}(q).2\) (\(m/2\) odd);
\item[(16)] \(L=\POmega_{m}^{+}(q)\) with \(m\geqslant 10\), \(X\cap L\leqslant N_{2}^{-}\) and \(Y\cap L\leqslant P_{m/2}\) or \(P_{(m-2)/2}\);
\item[(17)] \(L=\POmega_{m}^{+}(q)\) with \(m\geqslant 12\), \(X\cap L\leqslant\GU_{m/2}(q).2\) and \(Y\cap L\leqslant P_{1}\) with \(m\) even.
\item[(18)] \(L=\POmega_{16}^{+}(q)\), \(X\cap L\leqslant N_{1}\) and \(Y\cap L\leqslant\Omega_{9}(q).2\).
\end{description}
Proof of parts (i) and (ii): 

Note that \(\POmega_{4}^{-}(q)\cong \PSL_{2}(q^{2})\) gives examples in cases (9) and (10). For any of the cases (1)-(18), we can check that \(|Y\cap L|_{r}=1\) unless \(Y\cap L\) is listed in Table \ref{tab:YLneq1}.
Note \(r\in\ppd(q,df)\). By \cite[Lemma 2.2]{LX} we have \(|\Out(L)|_{r}=1\). Thus 
\[
    |B|_{r}=|Y|_{r}\leqslant|Y\cap L|_{r}\cdot|\Out(L)|_{r}=1
\]
if \(|Y\cap L|_{r}=1\). This completes the proof of part (i). Now let \(s\in\ppd(p,3f)\) with \((m,\epsilon)=(6,+)\) so \(L\) is in case (13) or (14). By observing cases (13) and (14) we have \(|X\cap L|_s=1\). Again using \cite[Lemma 2.2]{LX} we obtain \(|\Out(L)|_s=1\) and therefore 
\[
    |A|_{r}=|X|_r\leqslant|X\cap L|_r|\Out(L)|_r=1,
\]
proving part (ii).

Now we prove part (iii). Note that \(\ICF(A)=\ICF(A^{(\infty)})\) and \(\ICF(B)=\ICF(B^{(\infty)})\). For \(L=\POmega_{m}^{\epsilon}(q)\), all of the possible pairs \((L, A^{(\infty)}, B^{(\infty)})\) for \(A\) and \(B\) both insoluble have been classified in \cite[Theorem 3.1, 4.1, 5.1, 6.1, 7.1]{LWX}, and for at least one of \(A\) and \(B\) is soluble have been given in \cite[Theorem 1.1]{LX}. Note that some exceptional isomorphisms arise in the small dimensional cases, for instance, \(\POmega_{4}^{-}(q)\cong\PSL_{2}(q^{2})\cong\PSp_{2}(q^{2})\), \(\POmega_{6}^{-}(q)\cong\PSU_{4}(q)\) and \(\POmega_{6}^{+}(q)\cong\PSL_{4}(q)\). A straightforward observation of Table \ref{O(2m+1,q)} gives that \(\ICF(A)\cap\ICF(B)=\varnothing\) unless \((L,\ICF(A),\ICF(B))=(\POmega_{4}^{-}(3), \{\A_{5}\}, \{\A_{5}\})\) and for \(m\geqslant7\) (see rows 1-11 and 23-27), there exists \(T\in\ICF(A)\) such that \(|T|_{r}>1\).
\end{proof}

\begin{lemma}\label{8 +}
Let \(G\) be an almost simple group with socle \(L=\POmega_{8}^{+}(q)\), where \(q=p^{f}\) is odd. Suppose that \(G=AB\) is a core-free factorisation. Then, interchanging \(A\) and \(B\) if necessary, 
\begin{description}
\item[(i)]\((L, \ICF(A), \ICF(B))\) are as in Table \ref{P(8,+)}.
\item[(ii)] \(\ICF(A)\cap \ICF(B)=\varnothing\) unless \(\ICF(A)=\ICF(B)=\{\Omega_{7}(q)\}\). 
\item[(iii)]Moreover, \(\ppd(p,6f)\neq \varnothing\) and for \(r\in\ppd(p,6f)\) and \(T\in\ICF(A)\) we have \(|T|_{r}>1\).
\end{description}
\end{lemma}
\begin{table}
    \centering
    \begin{tabular}{c|c|c|l}
    \hline
    \(L\)& \(\ICF(A)\)    &  \(\ICF(B)\)& Conditions\\
     \hline
       \(\POmega_{8}^{+}(q)\)&\(\{\Omega_{7}(q)\}\)&\(\{\PSU_{4}(q)\}, \{\PSL_{a}(q^{b})\}\)&\\
       &&\( \{\PSp_{2}(q), \PSp_{4}(q)\}\)&\(4=ab, b\geqslant1\)\\
       &&\(\varnothing, \{\PSp_{a}(q^{b})\}\)&\\
       %\hline
       \(\POmega_{8}^{+}(q)\)&\(\{\PSU_{4}(q)\}\)&\(\{\PSL_{4}(q)\}\)&\\
       %\hline
       \(\POmega_{8}^{+}(q)\)&\(\{\Omega_{7}(q)\}\)&\(\{\Omega_{7}(q)\}\)&\(A=B^{\tau}\) for some \\
       &&&triality \(\tau\)\\
       &&\(\{\POmega_{8}^{-}(q^{1/2})\}\)&\\
       %\hline
       \(\POmega_{8}^{+}(3)\)&\(\{\Omega_{7}(3)\}\)&\(\varnothing, \{\A_{9}\}, \{\PSU_{4}(2)\}, \{\Sp_{6}(2))\}, \)&\\
       &&\(\{\A_{5}\}, \{\POmega_{8}^{+}(2)\}\)&\\
       %\hline
       \(\POmega_{8}^{+}(3)\)&\(\{\A_{t}\}, \{\PSU_{4}(3)\}\)& \(\{\PSL_{4}(3)\}\)&\(t=7, 8\) or \(9\)\\
       &\(\{\PSL_{3}(4)\}\), \(\{\Sp_{6}(2)\}\)&&\\
       %\hline
       \(\POmega_{8}^{+}(3)\)&\(\{\POmega_{8}^{+}(2)\}\)&\(\varnothing\), \(\{\PSL_{3}(3)\}\), \(\{\PSL_{4}(3)\}\)&\\
       \hline
    \end{tabular}
    \caption{Core-free factorisations \(G=AB\) for \(\Soc(G)=\POmega_{8}^{+}(q)\) where \(q=p^{f}\) is odd and \(q\geqslant3\).}
    \label{P(8,+)}
\end{table}
\begin{proof}
All of the core-free factorisations of \(G\) an almost simple group with socle \(\POmega_{8}^{+}(q)\) have been classified in \cite[Theorem 8.1]{LX} for both factors of \(G\) insoluble, and in \cite[Theorem 1.1]{LX} for at least one of the factors soluble. We merge those results together and obtain (i) and Table \ref{P(8,+)}. Parts (ii) and (iii) are direct observations from Table \ref{P(8,+)}.
\end{proof}

 \section{\(s\)-arc-transitivity of symplectic groups}\label{sect:mainthm}
 In this section we will determine an upper bound on \(s\) for a \(G\)-vertex-primitive \((G,s)\)-arc-transitive digraph such that \(\Soc(G)=\PSp_{2n}(q)'\) with \(n\geqslant2\). To achieve this goal we have to analyse each class of maximal subgroups of \(G\). First we will give the following hypothesis which will be used throughout the rest of the paper.
 
\begin{hypothesis}\label{general}
\rm{Let \(\Gamma\) be a \(G\)-vertex-primitive \((G,s)\)-arc-transitive digraph of valency at least \(3\), where \(s\geqslant2\) and \(G\) is almost simple with socle \(L=\PSp_{2n}(q)\) such that \(n\geqslant2\) and \(q=p^{f}\) for some prime \(p\). Take an arc \(u\to v\) in \(\Gamma\). Let \(\overline{g}\in L\) such that \(u^{\overline{g}}=v\) and let \(v^{\bar{g}}=w\). Then \(u\to v\to w\) is a \(2\)-arc in \(\Gamma\). Let \(X=\Sp_{2n}(q)\) acting naturally on the symplectic space \((V=\oplus_{i=1}^{n}\langle e_{i}, f_{i}\rangle,B)\), where \(B\) is an alternating form on \(V\) with \(B(e_{i}, e_{j})=B(f_{i}, f_{j})=0\) and \(B(e_{i}, f_{j})=\delta_{ij}\). Let \(\psi\) be the natural projection from \(X\) to \(L\), and \(g\) be a preimage of \(\overline{g}\) under \(\psi\).}
\end{hypothesis}

Under Hypothesis \ref{general}, if \(s\geqslant 3\), then by Lemma \ref{s-1} we find that \(\Gamma\) is \((L,2)\)-arc-transitive so that \(L_{v}=L_{uv}L_{vw}\) with \(L_{uv}^{\overline{g}}=L_{vw}\). Note that if \(\Gamma\) is \((L,2)\)-arc-transitive, then the composition of \(\psi\) and the action of \(L\) on the vertex set of \(\Gamma\) induces an action of \(X\) on the vertex set of \(\Gamma\). In particular, \(X\) acts \(2\)-arc-transitively on \(\Gamma\) and therefore \(X_{v}=X_{uv}X_{vw}\) such that \(X_{uv}^{g}=X_{vw}\).
The next result is from \cite[Lemma 5.4]{small}, which helps us rule out some small cases.

\begin{proposition}\label{citesmall}
Suppose that Hypothesis \ref{general} holds. Then: 
\begin{description}
    \item[(i)] if \(n=2\) then \(q\geqslant 41\);
    \item[(ii)] if \(n=3\) then \(q\geqslant7\);
    \item[(iii)] if \(n=4\) then \(q\geqslant3\).
\end{description}
\end{proposition}

Before analysing each class of maximal subgroups of \(G\), we give some geometric facts about symplectic groups.

\begin{lemma}\label{exercise}
Let  \((V, B)\) be a symplectic space such that \(V=\oplus_{i=1}^{n}\langle e_{i}, f_{i}\rangle\), \(B(e_{i}, e_{j})=B(f_{i}, f_{j})=0\) and \(B(e_{i}, f_{j})=\delta_{ij}\).
\begin{description}
    \item[(i)] If \(V_{1}, V_{2} \leqslant V\) are maximal totally isotropic subspaces such that \(V_{1}\cap V_{2}=0\), then there exists \(h\in\Sp(V)\) such that \(V_{1}^{h}=\oplus_{i=1}^{n}\langle e_{i}\rangle\) and \(V_{2}^{h}=\oplus_{i=1}^{n}\langle f_{i}\rangle\).
    \item[(ii)] If \(V_{1}, V_{2}\leqslant V\) are maximal totally isotropic subspaces such that \(V_{1}\cap V_{2}=0\), then there exists \(h\in\Sp(V)\) such that \(V_{1}^{h}=V_{2}\) and \(V_{2}^{h}=V_{1}\).
    \item[(iii)] If \(V_{1}, V_{2}\leqslant V\) are totally isotropic subspaces such that \(V_{1}+V_{2}\) is non-degenerate, then \(V_{1}\cap V_{2}=0\) and \(\dim V_{1}=\dim V_{2}\). In particular, there exists \(h\in\Sp(V)\) swapping \(V_{1}\) and \(V_{2}\).
\end{description}
\end{lemma}
\begin{proof}
Part (i) is an immediate consequence of \cite[Exercise 8.2 (vii)]{Taylor}.

For part (ii), it follows from part (i) that there exists \(x\in\Sp(V)\) such that \(V_{1}^{x}=\oplus_{i=1}^{n}\langle e_{i}\rangle\) and \(V_{2}^{x}=\oplus_{i=1}^{n}\langle f_{i}\rangle\). Note that there exists \(y\in \Sp(V)\) such that
\(e_{i}^{y}=f_{i}\) and \(f_{i}^{y}=-e_{i}\) for \(1\leqslant i\leqslant n\). Let \(h=xyx^{-1}\). Then \(V_{1}^{h}=V_{2}\) and \(V_{2}^{h}=V_{1}\).

We now prove part (iii). For any \(u\in V_{1}+V_{2}\), \(u=u_{1}+u_{2}\) for some \(u_{1}\in V_{1}\) and \(u_{2}\in V_{2}\). Let \(v\in V_{1}\cap V_{2}\). Then \(B(v,u)=B(v,\sum_{i=1}^{2}u_{i})=0\) as each \(V_{i}\) is totally isotropic. Since \(V_{1}+V_{2}\) is non-degenerate this implies that \(v=0\) and hence \(V_{1}\cap V_{2}=0\). 
On the other hand, since \(V_{1}+V_{2}\) is non-degenerate, the stabiliser of \(V_{1}+V_{2}\) in \(\Sp(V)\) induces \(\Sp(V_{1}+V_{2})\). 

\medskip
\noindent
\textit{Claim:}
    \(\dim V_{1}=\dim V_{2}\).

\medskip    
\noindent
Suppose for a contradiction that \(\dim V_{1}<\dim V_{2}\). Then \(V_{1}\) is properly contained in some maximal totally isotropic subspace of \(V_{1}+V_{2}\) and therefore \(V_{1}\) is properly contained in \(V_{1}^{\perp}\). Thus there exists elements \(w_{1}\in V_{1}\) and \(w_{2}\in V_{2}\) such that \(w_{2}\neq 0\) and \(w_{1}+w_{2}\in V_{1}^{\perp}\). However, this implies that \(w_{2}\in V_{1}^{\perp}\cap V_{2}^{\perp}=(V_{1}+V_{2})^{\perp}\), which is a contradiction since \((V_{1}+V_{2})\cap (V_{1}+V_{2})^{\perp}=0\). This proves the Claim. 

\medskip

Thus \(W:=V_{1}+V_{2}\) is non-degenerate and therefore a symplectic geometry with respect to the restriction \(B|_{W}\). Note that \(V_{1}\) and \(V_{2}\) are maximal totally isotropic subspaces in \(W\) with \(V_{1}\cap V_{2}=0\). Hence by part (ii), there exists \(h\in\Sp(W)\leqslant\Sp(V)\) such that \(V_{i}^{h}=V_{2-i}\) for \(i=1,2\).
\end{proof}

\begin{corollary}\label{uw}
Let \((W, B)\) be a symplectic space of dimension \(2n\) and \(W_{0}\leqslant W\) be a totally isotropic subspace of dimension \(\ell\). Then there exists a totally isotropic \(\ell\)-space \(U_{0}\leqslant W\) such that \(U_{0}+W_{0}=\oplus_{i=1}^{\ell} L_{i}\), where \(L_{i}=\langle u_{i}, w_{i}\rangle\) are hyperbolic lines such that \(u_{i}\in U_{0}\) and \(w_{i}\in W_{0}\).
\end{corollary}
\begin{proof}
Let \(\{e_{1},\ldots, e_{n},f_{1},\ldots,f_{n}\}\) be a symplectic basis for \(W\) and \(W_{1}\) be a maximal totally isotropic subspace of \(W\) such that \(W_{0}\leqslant W_{1}\). Since \(\Sp(W)\) acts transitively on the set of maximal totally isotropic subspaces of \(W\), there exists \(g\in\Sp(W)\) such that \(W_{1}^{g}=E\) where \(E:=\oplus_{i=1}^{n}\langle e_{i}\rangle\). Then, replacing \(W_{1}\) by \(W_{1}^{g}\), we may assume that \(W_{1}=\oplus_{i=1}^{n}\langle e_{i}\rangle\). Let \(F:=\oplus_{i=1}^{n}\langle f_{i}\rangle\). Note that \(\Sp(W)_{E,F}^{E}=\GL(E)\). Since \(\GL(E)\) acts transitively on the set of \(\ell\)-subspaces of \(E\), we may assume that \(W_{0}:=\oplus_{i=1}^{\ell}\langle e_{i}\rangle\). By letting \(U_{0}:=\oplus_{i=1}^{\ell}\langle f_{i}\rangle\), \(u_{i}:=e_{i}\) and \(v_{i}:=f_{i}\) for \(1\leqslant i\leqslant \ell\), we have \(U_{0}+W_{0}=\oplus_{i=1}^{\ell}\langle e_{i},f_{i}\rangle\) and \(L_{i}=\langle e_{i}, f_{i}\rangle\) are hyperbolic lines for all \(i\). Hence the result is proved.
\end{proof}

\subsection{\(\mathcal{C}_{1}\)-subgroups}

Recall from Subsection~\ref{cosgh} that each $(G,1)$-arc-transitive digraph $\Gamma$ is isomorphic to an orbital digraph of the form \(\Cos(G,G_{v},h)\) for some \(h\in G\) such that \(h^{-1}\notin G_{v}hG_{v}\), where $v$ is a vertex of $\Gamma$.
By \cite[Propositions 4.1.3 and 4.1.9]{kleidman}, for \(G\) acting on a symplectic space \(V\), if \(H\leqslant G\) is a maximal \(\mathcal{C}_{1}\)-subgroup, then \(H\) is the stabiliser of some subspace \(W\leqslant V\). In particular, \(W\) is either a totally isotropic subspace or a non-degenerate subspace. First we give an example to demonstrate that there exist \((G,1)\)-arc-transitive digraphs such that \(G_{v}\) is the stabiliser of a non-degenerate subspace. 
Let us consider \(V=\oplus_{i=1}^{n}\langle e_{i}, f_{i}\rangle\), where \(n\geqslant6\) and \(\{e_{1},\dots,f_{n}\}\) is a symplectic basis for \(V\). Let 
\[W_{1}=\langle e_{1}, f_{1}, e_{4}, f_{4}, e_{5}+f_{3}, f_{5}+e_{2}, e_{6}, f_{6},\ldots, e_{n}, f_{n}\rangle
\] 
and 
\[
W_{2}=\langle e_{1}, f_{1}, e_{2}, f_{2}, e_{3}, f_{3}, e_{6}, f_{6},\ldots, e_{n}, f_{n}\rangle.
\] 
Both \(W_{1}\) and \(W_{2}\) are non-degenerate of dimension \(2n-4\). Also note that \(W_{1}\cap W_{2}^{\perp}=\langle e_{4}, f_{4}\rangle\) and \(W_{2}\cap W_{1}^{\perp}=\langle e_{2}, f_{3}\rangle\). The former  is non-degenerate while the latter  is isotropic. Hence there is no element \(h\in G\) such that \((W_{1}, W_{2})^{h}=(W_{2}, W_{1})\), and hence the orbital containing \((W_{1}, W_{2})\) is not self-paired. Thus there exists \(h\in G\) such that \(h^{-1}\notin G_{W_{1}}hG_{W_{1}}\) and so \(\Cos(G,G_{W_{1}},h)\) is a digraph. By \cite[Lemma 2.4]{quasi}, \(\Cos(G,G_{W_{1}},h)\) is \(1\)-arc-transitive.

\begin{proposition}\label{c1nonde}
 Suppose that Hypothesis \ref{general} holds and \(G_{v}\) is a maximal \(\mathcal{C}_{1}\)-subgroup stabilising a proper nontrivial subspace \(W\leqslant V\). Then the following hold:
\begin{description}
    \item[(i)] \(W\) is non-degenerate;
     \item[(ii)]  $L$ is not $2$-arc transitive on $\Gamma$;
     \item[(iii)] \(\Gamma\) is not $(G,3)$-arc-transitive.
 \end{description}
\end{proposition}

% {\color{red}CP part (3) has to be $s=2$, since we assume that $s\geq2$ in Hyp \ref{general}. Or better: change this to $\Gamma$ is not $(G,3)$-arc-transitive.}

\begin{proof}
The graph $\Gamma$ in Hypothesis~\ref{general} is isomorphic to \(\Cos(G, G_{v},g)\) (as defined in the first paragraph of Subsection \ref{cosgh}), where  \(u\to v\) is  an arc of \(\Gamma\) and $g\in G$ satisfies $u^g= v$, and  \(g^{-1}\notin G_{v}gG_{v}\). In this case $G_v$ is the stabiliser in $G$ of the subspace $W$. Note that \(W\) is either totally isotropic or non-degenerate. First assume that \(W\) is a \(k\)-dimensional totally isotropic subspace. Then it follows from \cite[Lemma 2.13]{small} that all orbitals of \(G\) acting on the set of totally isotropic \(k\)-subspaces in \(V\) are self-paired. Therefore \(\Cos(G, G_{v},g)\)  is indeed a graph. Since we are assuming that \(\Gamma\) is a \(G\)-vertex-primitive digraph in Hypothesis \ref{general}, this case cannot occur. Hence \(W\) is non-degenerate, proving part (i). %, and we may assume that \(W=W_{1}\) as defined above. 

\medskip

%We may take   Since \(\Gamma=\Cos(G, G_{v},g)\) is a digraph, it follows that \(g^{-1}\notin G_{W_{1}}gG_{W_{1}}\).

Next, we prove part (ii).  It is convenient to work with a more geometric labeling of $\Gamma$, namely we may take the vertex-set to be \(\Sigma=\{W^{h}\mid h\in G\}\), the arc-set to be \(A=\{W^{h}\to W^{gh}\mid h\in G\}\), and the arc $u\to v$ in Hypothesis~\ref{general} to be  \(W_{1}\to W_{2}\) with $W_1, W_2$ distinct elements of $\Sigma$ and $W_2=W_1^g$. Suppose for a contradiction that \(L\) acts \(2\)-arc-transitively on \(\Gamma\). 
Again it will be convenient to work with the natural action of $X=\Sp_{2n}(q)$ on $\Gamma$, since the permutation group on $\Sigma$ induced by $X$ is equal to $L$. 
 Let \(2k=\dim W_1\) and \(U=W_{1}\cap W_{2}\). We divide the analysis into two cases:

\medskip\noindent
\emph{Case $1$: \(U\) is non-degenerate.}\quad In this case, let \(\ell:=\frac{1}{2}\dim U\), so $1\leq \ell<k$ since $U=W_1\cap W_2$ and $W_1\ne W_2$. There exist hyperbolic lines \(L_{i}, T_{j}\) such that  \(U=\oplus_{i=1}^{\ell}L_{i}\), 
\(W_{1}=U\perp\oplus_{j=\ell+1}^{k}L_{j}\) and \(W_{2}=U\perp\oplus_{j=\ell+1}^{k}T_{j}\). Let \(x\) and \(y\) be elements in \(X\) such that:
\[
\begin{array}{cc}

L_{i}^{x}=\begin{cases}
     L_{i}, & 1\leqslant i\leqslant\ell\\
     T_{i}, &  \ell+1\leqslant j\leqslant k
     \end{cases}

     \text{and}
     &

   L_{i}^{y}=\begin{cases}T_{k}, & i=1\\
    L_{i}, &2\leqslant i\leqslant\ell\\
    T_{i}, &\ell+1\leqslant i\leqslant k-1\\
    L_{1}, &i=k
    \end{cases}
    \\
\end{array}\]
Note that the existence of \(x\) and \(y\) is guaranteed by the transitivity of \(X\) on the symplectic bases of \(V\). Note also \(W_{1}^{x}=W_{1}^{y}=W_{2}\). Let 
\(W_{3}:=W_{2}^{x}\) and \(W_{3}':=W_{2}^{y}\). Then we have \((W_{2}, W_{3})=(W_{1}, W_{2})^{x}\) and \((W_{2}, W_{3}')=(W_{1}, W_{2})^{y}\) are arcs in \(\Gamma\). Thus \(W_{1}\to W_{2}\to W_{3}\) and \(W_{1}\to W_{2}\to W_{3}'\) are \(2\)-arcs in \(\Gamma\). Since \(X\) acts \(2\)-arc-transitively on \(\Gamma\), there exists an element \(z\in X\) such that 
\[
    (W_{1}\to W_{2}\to W_{3})^{z}=W_{1}\to W_{2}\to W_{3}'.
\]
Since \(W_{1}\cap W_{2}=W_{2}\cap W_{3}=U\), it follows that \((W_{1}\cap W_{2})^{z}=(W_{2}\cap W_{3})^{z}\) and so \(W_{1}\cap W_{2}=W_{2}\cap W_{3}'\), which is a contradiction since \(W_{3}'=W_2^y\) contains \(L_{1}^{y}=T_{k}\) which implies that \(W_{2}\cap W_{3}'\) contains \(T_{k}\), while \(T_{k}\) is not contained in \(W_1\cap W_2=U\).

\medskip\noindent 
\emph{Case $2$: \(\Rad(U)\neq 0\).}\quad In this case, let $c=\dim \Rad(U)>0$ and \(2\ell=\dim U/\Rad(U)\geqslant 0\). Then \(\dim U=c+2\ell\), and this is less than \(\dim W_{1}=2k\) since \(U=W_{1}\cap W_{2}\) \(W_{1}\neq W_{2}\). It follows that \(c+\ell<k\). There exist hyperbolic lines \(L_{i}\) for $i\leqslant \ell$, such that \(U=\Rad(U)\perp\oplus_{i\leqslant\ell}L_{i}\) (where we may have $\ell=0$). Also, by Corollary \ref{uw}, there exist totally isotropic $c$-subspaces \(U_{1}\leqslant W_{1}\) and \(U_{2}\leqslant W_{2}\), and  hyperbolic lines \(L_{i}\) in \(W_{1}\), and \(T_{i}\) in \(W_{2}\), for $i=\ell+1, \dots, \ell+c$, such that \(U_{1}\oplus \Rad(U)=\oplus_{i=\ell+c}^{\ell+c}L_{i}\) and \(U_{2}\oplus \Rad(U)=\oplus_{i=\ell+1}^{\ell+c}T_{i}\) are nondegenerate $2c$-subspaces of $W_1$ and $W_2$, respectively. Then we may write \(L_{i}=\langle w_{i}, u_{i}\rangle\) and \(T_{i}=\langle t_{i}, v_{i}\rangle\), for \(\ell+1\leqslant i\leqslant \ell+c\), where each \(w_{i}, t_{i}\in \Rad(U)\) and \(u_{i}\in U_{1}\), \(v_{i}\in U_{2}\). Then the $w_i$ and the $t_i$ each form a basis of $\Rad(U)$, 
and we have 
\begin{equation}\label{e:U}
U=\langle w_{\ell+1},\ldots w_{\ell+c}\rangle\perp\oplus_{i\leqslant \ell}L_{i}=\langle t_{\ell+1},\ldots t_{\ell+c}\rangle\perp\oplus_{i\leqslant \ell}L_{i}.    
\end{equation}
 Finally, there are hyperbolic lines \(L_{i}\) in \(W_{1}\), and \(T_{i}\) in \(W_{2}\), for  \(i= \ell+c+1, \dots, k\), such that 
 
%Then \(L_{j+\ell}=\langle w_{j},u_{j}\rangle\) and \(T_{j+\ell}=\langle t_{j}, v_{j}\rangle\), for \(1\leqslant j\leqslant c\), such that \(w_{j}, t_{j}\in \Rad(U)\), \(u_{j}\in U_{1}\) and \(v_{j}\in U_{2}\). {\color{blue} Note that \(\Rad(U)=\langle w_{1},\ldots w_{c}\rangle=\langle t_{1},\ldots t_{c}\rangle\) and therefore \(U=\langle w_{1},\ldots w_{c}\rangle\perp\oplus_{i\leq \ell}L_{i}=\langle t_{1},\ldots t_{c}\rangle\perp\oplus_{i\leq \ell}L_{i}\). Also $U_i\oplus U\leq W_i$ and $\dim(U_i\oplus U)=2(c+\ell)\leq 2k$. Although $\ell\geq0$ and $k\geq c+\ell$, at least one of these inequalities is strict since $W_1\ne W_2$.} {\color{red} We may also consider in this way: \(U=\Rad(U)\perp (\oplus_{i\leq\ell}L_{i})<W_{1}\) and so \(\dim U=\dim(\oplus_{i\leq\ell}L_{i})+\dim\Rad(U)=2\ell+c<\dim(W_{1})=2k\) and so \(c+\ell<k\). Thus the second inequality should always be strict.}
% For {\color{blue} (Recall that \(c+\ell<k\) and so such \(i\) exists)}, there are hyperbolic lines \(L_{i}\) and \(T_{i}\), in \(W_{1}\) and \(W_{2}\), respectively, such that   
\[
W_{1}=(\oplus_{i\leqslant \ell}L_{i})\perp(\oplus_{i=\ell+1}^{\ell+c}L_{i})\perp(\oplus_{i=\ell+c+1}^{k}L_{i})
\]
and 
\[
W_{2}=(\oplus_{i\leqslant \ell}L_{i})\perp(\oplus_{i=\ell+1}^{\ell+c}T_{i})\perp(\oplus_{i=\ell+c+1}^{k}T_{i}).
\] 
 We note that the first summands of \(W_{1}\) and \(W_{2}\) may be zero since $\ell\geq0$,  but the second and third summands of these decompositions are always non-zero since both $c>0$ and $c+\ell<k$.
% {\color{green} Note that 
% \[
% (\oplus_{i\leqslant \ell}L_{i}) <U < \widehat{U} := (\oplus_{i\leqslant \ell}L_{i})\perp(\oplus_{j=1}^{c}L_{j+\ell})\perp (\oplus_{j=1}^{c}T_{j+l}).
% \]}
We now choose elements \(x, y\in X\) such that their actions on the set of hyperbolic lines $L_i$ in the above decomposition of $W_1$ satisfy the following:
\[\begin{array}{cc}
L_{i}^{x}=\begin{cases}L_{i}, &i\leqslant\ell\\

T_{i},  &\ell+1\leqslant i\leqslant k
\end{cases}
\quad \text{and}\quad  &
     \begin{cases}
     w_{i}^{x}=t_{i}, &\ell+1\leqslant i\leqslant \ell+c\\
     u_{i}^{x}=v_{i}, &\ell+1\leqslant i\leqslant \ell+c.\\
     \end{cases}
\end{array}    
\]
 and 
 %if $\ell\geqslant 1$ then {\color{blue}
%\[\begin{array}{cc}
%L_{i}^{y}=\begin{cases}T_{k}, &i=1\\
%L_{i}, &2\leqslant i\leqslant\ell\\
%T_{i}, &\ell+1\leqslant i\leqslant k-1\\
%L_{1}, &i=k
%\end{cases}
%\quad \text{with}\quad &
   %  \begin{cases}
  %   w_{j}^{y}=t_{j}, &1\leqslant j\leqslant c\\
  %   u_{j}^{y}=v_{j}, &1\leqslant j\leqslant c.
  %   \end{cases}
  %  \end{array}
%\]
%while if \(\ell=0\), then
%\[\begin{array}{cc}
%L_{i}^{y}=\begin{cases}
%T_i    &1\leqslant i\leqslant c-1\\
%T_{c+1} & i=c\\
%T_c,   &i=c+1\\  %\langle t_{c},v_{c}\rangle
%T_{i}, &c+2\leqslant i\leqslant k\\
%\end{cases}
%\quad \text{with} \quad &
  %   \begin{cases}
   %  w_{j}^{y}=t_{j}, &1\leqslant j\leqslant c-1\\
   %  u_{j}^{y}=v_{j}, &1\leqslant j\leqslant c-1\\
    % \langle w_{c},u_{c}\rangle^{y}=T_{c+1}.
   %  \end{cases}
  %  \end{array}
%\]
%}

\[\begin{array}{cc}
L_{i}^{y}=\begin{cases}
L_i    & i\leqslant \ell\\
T_k & i=\ell+1\\
T_{i} & \ell+2\leqslant i\leqslant k-1\\
T_{\ell+1} & i=k

%T_{c+1} & i=c\\
%T_c,   &i=c+1\\  %\langle t_{c},v_{c}\rangle
%T_{i}, &c+2\leqslant i\leqslant k\\
\end{cases}
\quad \text{with} \quad &
    \begin{cases}
  w_{i}^{y}=t_{i}, &\ell+2\leqslant i\leqslant \ell+c\\
  u_{i}^{y}=v_{i}, &\ell+2\leqslant i\leqslant \ell+c.\\
    \end{cases}
   \end{array}
\]

 Let \(W_{3} :=W_{2}^{x}\) and \(W_{3}' :=W_{2}^{y}\), and note that 
\[
W_{1}^{x}=(\oplus_{i\leqslant \ell}L_{i}^{x})\perp(\oplus_{i=\ell+1}^{k}L_{i}^{x})=(\oplus_{i\leqslant \ell}L_{i})\perp(\oplus_{i=\ell+1}^{k}T_{i})=W_{2}
\]
and 
\[
W_{1}^{y}=(\oplus_{i\leqslant \ell}L_{i}^{y})\perp L_{\ell+1}^{y}\perp(\oplus_{i=\ell+2}^{k-1}L_{i}^{y})\perp L_{k}^{y}=(\oplus_{i\leqslant \ell}L_{i})\perp T_{k}\perp(\oplus_{i=\ell+2}^{k-1}T_{i})\perp T_{\ell+1}=W_{2}.
\]

Then \((W_{2}, W_{3})=(W_{1}, W_{2})^{x}\) and \((W_{2}, W_{3}')=(W_{1}, W_{2})^{y}\) and these are arcs of \(\Gamma\). Therefore \(W_{1}\to W_{2}\to W_{3}\) and \(W_{1}\to W_{2}\to W_{3}'\) are \(2\)-arcs in \(\Gamma\), and since \(X\) acts \(2\)-arc-transitively on \(\Gamma\), there exists \(z\in X\) such that \((W_{1}\to W_{2}\to W_{3})^{z}=W_{1}\to W_{2}\to W_{3}'\). This implies that $U^z=(W_1\cap W_2)^z=W_1^z\cap W_2^z=W_1\cap W_2 = U$.

%{\color{red}
%By the definition of $x$ we have $W_2\cap W_3=W_1^x\cap W_2^x = (W_1\cap W_2)^x=U^x$, and you claim $U^y=U$. However if $\ell\geq1$ then you map $L_1\leq U$ to $T_k$ which is not contained in $U$. 
%Something is wrong here - so what was your intuition and what did you prove?}

By the definition of $x$ we have \(U^x=(W_1\cap W_2)^x=W_1^x\cap W_2^x =W_2\cap W_3\), but also it follows from \eqref{e:U} that $U^x=U$.
% \[
% U^x=(W_{0}\perp\oplus_{i\leqslant \ell}L_{i})^{x}=W_{0}^{x}\perp\oplus_{i\leqslant \ell}L_{i}^{x}=\langle w_{1+\ell}^{x},\ldots w_{c+\ell}^{x}\rangle\perp\oplus_{i\leqslant \ell}L_{i}^{x}=\langle t_{1+\ell},\ldots t_{c+\ell}\rangle\perp\oplus_{i\leqslant \ell}L_{i}=W_{0}\perp\oplus_{i\leqslant \ell}L_{i}=U.
% \]
Thus \(W_{1}\cap W_{2}=U=U^x=W_{2}\cap W_{3}\). This implies first that $U^y = (W_1\cap W_2)^y = W_1^y\cap W_2^y = W_2\cap W_3 =  U$, and also that \(U^z=(W_{1}\cap W_{2})^{z}=(W_{2}\cap W_{3})^{z} = W_{2}^z\cap W_{3}^{z}=W_2\cap W_3'\).  Moreover we showed in the previous paragraph that $U^z=U$, and hence we have $W_2\cap W_3' = U = W_2\cap W_3$. Since \(t_{\ell+1}\in \Rad(U)\subseteq U\), it follows that \(t_{\ell+1}^{y^{-1}}\in U^{y^{-1}}=U\). On the other hand $t_{\ell+1}^{y^{-1}}\in T_{\ell+1}^{y^{-1}}=L_k$ and $L_k\cap U=0$, so we have a contradiction.

Thus part (ii) of Proposition~\ref{c1nonde} is proved. For part (iii), if \(\Gamma\) were $(G,3)$-arc-transitive, then by Lemma \ref{s-1}, \(\Gamma\) would be \((L,2)\)-arc-transitive, contradicting part (ii). Thus part (iii) holds, and Proposition~\ref{c1nonde} is proved.
\end{proof}

% \((W_{1}\cap W_{2})^{z}=(W_{2}\cap W_{3})^{z}\) and so \(U=W_{1}\cap W_{2}=W_{2}\cap W_{3}'=W_{1}^{y}\cap W_{2}^{y}=(W_{1}\cap W_{2})^{y}=U^{y}\). Since \(t_{\ell+1}\in U\), it follows that \(t_{\ell+1}^{y^{-1}}\in U^{y^{-1}}=U\). It follows from the definition of \(y\) that \(t_{\ell+1}^{y^{-1}}=w\) for some \(w\in L_{k}\). However, we observe that \(L_{k}\perp W_{2}\) and this implies that \(w\notin W_{2}\), which is a contradiction and}
% proving part (ii) of Proposition~\ref{c1nonde}.

\subsection{\(\mathcal{C}_{2}\)-subgroups}\label{sub:c2}

By \cite[Table 3.5.C]{kleidman}, for an almost simple group \(G\) with socle \(L=\PSp_{2n}(q)'\), if \(H\) is a maximal \(\mathcal{C}_{2}\)-subgroup of \(G\), then either 
\begin{enumerate}
    \item[(I)]  $H$ is a Type (I) subgroup, that is,  \(H\) stabilises a \emph{non-degenerate decomposition} \(V=\oplus_{i=1}^{k}W_{i}\) of the symplectic space \(V\) where \(W_{1},\ldots, W_{k}\) are non-degenerate $(2n/k)$-subspaces, where $k\geqslant 2$; or 
    \item[(II)] \(H\) is a Type (II) subgroup, that is, \(n\geqslant2\), \(q\) is odd, \(H\) stabilises a decomposition $V=W_1\oplus W_2$ with each $W_i$ a totally isotropic $n$-subspace,  and \(H\cap L\cong\GL_{n}(q)/\{\pm 1\}\).
\end{enumerate}

\subsubsection{\(\mathcal{C}_{2}\)-subgroups of Type (I)}

Let \(X_{v}\) be a maximal \(\mathcal{C}_{2}\)-subgroup of \(X=\Sp_{2n}(q)\) of Type (I). Then by \cite[Table 3.5 C]{kleidman}, we have \(X_{v}\cong \Sp_{2m}(q)\wr \Sy_{k}\), where \(n=mk\) and \(m\geqslant2\). Let \(M\) be the base subgroup of \(X_{v}\) and \(M_{i}\) be the \(i\)-th component of \(M\) for \(1\leqslant i\leqslant k\). Then \(M\cong\Sp_{2m}(q)^{k}\) and \(M_{i}\cong\Sp_{2m}(q)\) for \(1\leqslant i\leqslant k\). Indeed, \(M\) is the subgroup of \(X_{v}\) stabilising each of the non-degenerate subspaces \(W_{1},\ldots, W_{k}\) (as defined above). We will use this notation throughout the discussion of \(\mathcal{C}_{2}\)-subgroups of \(\Sp_{2n}(q)\) of Type (I).

\begin{lemma}\label{c2k2}
Suppose that Hypothesis \ref{general} holds and \(k=2\). If \(G_{v}\) is a Type (I) \(\mathcal{C}_{2}\)-subgroup of \(G\) with \((m,q)\neq (1,3), (1,9)\) and \((m,p)\neq (2,2)\), then \(\Gamma\) is not \((L,2)\)-arc-transitive.
\end{lemma}

\begin{proof}
Suppose for a contradiction that \(\Gamma\) is \((L,2)\)-arc-transitive while Hypothesis \ref{general} holds and \(k=2\). Then \(\Gamma\) is therefore \((X,2)\)-arc-transitive (recall that \(X=\Sp_{2m}(q)\) as defined in Hypothesis \ref{general}). Thus \(X_{v}\) admits a homogeneous factorisation \(X_{v}=X_{uv}X_{vw}\) with \(X_{uv}^{g}=X_{vw}\). By Lemma \ref{k=2}, interchanging \(M_{1}\) and \(M_{2}\) if necessary, \(X_{vw}^{(\infty)}=M_{1}\), and there exists \(\alpha\in\Aut(\Sp_{2m}(q))\) such that \(X_{uv}^{(\infty)}=\{(x,x^{\alpha})\mid x\in M_{1}\}\). Therefore \(X_{vw}^{(\infty)}\) stabilises \(W_{1}\) and acts trivially on \(W_{2}\), while the induced action of \(X_{uv}^{(\infty)}\) on both \(W_{1}\) and \(W_{2}\) is \(\Sp_{m}(q)\). Thus \(X_{vw}^{(\infty)}\) and \(X_{uv}^{(\infty)}\) are not conjugate in \(X=\Sp_{2m}(q)\), which is a contradiction to the fact that \((X_{uv}^{(\infty)})^{g}=X_{vw}^{(\infty)}\). Thus the result follows.
\end{proof}
\begin{lemma}\label{sp23c2}
   Suppose that Hypothesis \ref{general} holds and \(G_{v}\) is a Type (I) \(\mathcal{C}_{2}\)-subgroup such that \((m,q)=(1,3)\). Then \(\Gamma\) is not \((G,3)\)-arc-transitive. Moreover, \(\Gamma\) is not \((L,2)\)-arc-transitive for \(n\neq6\).
\end{lemma}
\begin{proof}
First note that \(n=k\) since \(m=1\), and by Proposition \ref{citesmall} \(n\geqslant4\), since \(q\geqslant3\).

\medskip First assume that \(n=4\). Then \(L=\PSp_{8}(3)\). We check by MAGMA \cite{magma} (see Subsection \ref{r:magma}) that \(\Sp_{2}(3)\wr\Sy_{4}\) does not admit a homogeneous factorisation and therefore \(\Gamma\) is not \((X,2)\)-transitive (otherwise \(X_{v}=\Sp_{2}(3)\wr\Sy_{4}=X_{uv}X_{vw}\) is a homogeneous factorisation). This implies that \(\Gamma\) is not \((L,2)\)-arc-transitive and it follows from Lemma \ref{s-1} that \(\Gamma\) is not \((G,3)\)-arc-transitive.

\medskip Now assume that \(n=5\) or \(n\geqslant7\). Suppose for a contradiction that \(\Gamma\) is \((L,2)\)-arc-transitive. Then \(\Gamma\) is therefore \((X,2)\)-arc-transitive and \(X_{v}\) admits a homogeneous factorisation \(X_{v}=X_{uv}X_{vw}\). Since \(k\neq6\), it follows from Lemma \ref{Sp23} that \(\Sp_{2}(3)\wr\Sy_{k}\) does not admit any homogeneous factorisation, contradicting the fact that \(X_{v}=X_{uv}X_{vw}\) is a homogeneous factorisation and therefore \(\Gamma\) is not \((L,2)\)-arc-transitive. Then it follows from Lemma \ref{s-1} that \(\Gamma\) is not \((G,3)\)-arc-transitive.

\medskip It remains to consider \(n=6\), so \(L=\PSp_{12}(3)\). Suppose for a contradiction that \(\Gamma\) is \((G,3)\)-arc-transitive. Then it follows from Lemma \ref{s-1} that \(\Gamma\) is \((L,2)\)-arc-transitive and therefore \(L_{v}\) admits a homogeneous factorisation \(L_{v}=L_{uv}L_{vw}\). By \cite[Theorem 3.8]{Wilson} we have \(L_{v}\cong 2^{5}.(\PSp_{2}(3)\wr\Sy_{6})\) and so \(|L_{v}|=2^{21}\cdot3^{8}\cdot5\). Since \(\Gamma\) is \((G,3)\)-arc-transitive, it follows from Lemma \ref{valency} that \((|X_{v}|/|X_{uv}|)^{3}\) divides \(|G_{v}|\) and hence divides \(|L_{v}||\Out(L)|=2^{22}\cdot3^{8}\cdot5\). Note that \(|L_{v}|^{3}\) divides \(|X_{v}|^{3}\) and \(|X_{v}|^{3}\) divides \(2^{22}\cdot3^{8}\cdot5|X_{uv}|^{3}\). This implies that \(2^{63}\cdot3^{24}\cdot5^{3}=|L_{v}|^{3}\) divides \(2^{22}\cdot3^{8}\cdot5|X_{uv}|_{3}^{3}\) and this implies that \(2^{14}\cdot 3^{6}\cdot5\) divides \(|X_{uv}|\). We compute by MAGMA \cite{magma} (see Subsection \ref{r:magma}) that \(X_{v}=\Sp_{2}(3)\wr\Sy_{6}\) does not admit a factorisation such that the order of each factor is divisible by \(2^{14}\cdot3^{6}\cdot5\) and both factors are isomorphic, which is a contradiction. Thus \(\Gamma\) is not \((G,3)\)-arc-transitive.
\end{proof}

\begin{proposition}\label{p:notc2I}
Suppose that Hypothesis \ref{general} holds. If \(G_{v}\) is a maximal \(\mathcal{C}_{2}\)-subgroup of \(G\) of type I, then \(\Gamma\) is not \((G,3)\)-arc-transitive. Moreover, \(\Gamma\) is not \((L,2)\)-arc-transitive except possibly when \((m,q)= (1,9)\), or  \((m,q)= (1,3)\) with $k=6$, or \((m,p)\neq (2,2)\).
\end{proposition}
\begin{proof}
The assertions follow from Lemma~\ref{sp23c2} when \((m,q)= (1,3)\). So we assume from now on that \((m,q)\neq (1,3)\), and divide our analysis into two cases:

\medskip\noindent
\emph{Case (1): \((m,q)\neq (1,9)\) and \((m,p)\neq (2,2)\).}  Note that \(\Gamma\) is not \((L,2)\)-arc-transitive when \(k=2\), by Lemma \ref{c2k2}. Thus it remains to discuss the situation when \(k\geqslant3\). Suppose that \(\Gamma\) is \((L,2)\)-arc-transitive with \(L_{v}=G_{v}\cap L\) being a \(\mathcal{C}_{2}\) subgroup of \(L\), then \(\Gamma\) is \((X,2)\)-arc-transitive. By Lemma \ref{0} we have \(X_{v}=X_{uv}X_{vw}\) with \(X_{uv}^{g}=X_{vw}\). Thus \(X_{v}=\Sp_{2m}(q)\wr \Sy_{k}\) admits a homogeneous factorisation. However, Lemma \ref{Sp} tells us that such a factorisation does not exist, which is a contradiction.

\medskip\noindent
\emph{Case (2): \((m,q)=(1,9)\) or \((m,p)=(2,2)\).} 
Suppose that \(G_{v}\) is a \(\mathcal{C}_{2}\) subgroup of \(G\) and \(\Gamma\) is \((G,3)\)-arc-transitive. Let us denote by \(\pi\) the projection from \(X_{v}\) to \(X_{v}/M\). Thus \(\pi(X_{v})=\Sy_{k}\), \(\pi(X_{uv})=X_{uv}M/M\) and \(\pi(X_{vw})=X_{vw}M/M\). From the factorisation \(X_{v}=X_{uv}X_{vw}\) we deduce that \(\pi(X_{v})=\pi(X_{uv})\pi(X_{vw})\). Then by Lemma \ref{ab} we find that at least one of \(\pi(X_{uv})\) or \(\pi(X_{vw})\) is a transitive subgroup of \(\Sy_{k}\). Without loss of generality, assume that \(\pi(X_{uv})\) is transitive. Note that \(X_{v}=\Sp_{2m}(q)^{k}: \Sy_{k}\). It follows from Lemma \ref{wreath} that \(\varphi_{1}(X_{uv}\cap M)\cong\cdots\cong \varphi_{k}(X_{uv}\cap M)\) and \(\Pi(\Sp_{2m}(q))=\Pi(\varphi_{1}(X_{uv}\cap M))\). Let \(Z\) be the center of \(\varphi_{1}(M)=\Sp_{2m}(q)\). Then \(Z=\C_{(q-1,2)}\) and \(\varphi_{1}(M)Z/Z=\PSp_{2m}(q)\). Thus \(\Pi(\varphi_{1}((X_{uv}\cap M)Z/Z))\supseteq \Pi(\varphi_{1}(X_{uv}\cap M))\backslash\Pi(Z)\supseteq\Pi(\PSp_{2m}(q))\backslash\{2\}\). Thereby we deduce from \cite[Table 10.3, Table 10.7]{transitive} that one of the following holds:
\begin{description}
\item[(i)] \(m=1, q=9\) and \(\varphi_{1}(X_{uv}\cap M)Z/Z=\A_{5}\);
\item[(ii)] \(m=p=2\), and \(\PSp_{2}(q^{2})'\triangleleft\varphi_{1}(X_{uv}\cap M)Z/Z\leqslant \PSp_{2}(q^{2}).2\);
\item[(iii)] \(\varphi_{1}(X_{uv}\cap M)Z/Z=\PSp_{2m}(q)\);
\end{description}

\medskip
\noindent
\textit{Claim~1:}
\(\varphi_{1}(X_{uv}\cap M)Z/Z=\PSp_{2m}(q)\).

\medskip
\noindent
Suppose that the assertion is not true. Then either case (i) or (ii) occurs. If \(f\neq3\) then \(\ppd(2,2f)\neq\varnothing\) by Corollary \ref{existppd}, choose \(r\in\ppd(2,2f)\) and note that \(r>2f\). We define \(d\) by
\[
d:=\begin{cases}
    3 & \text{if \((m,q)=(1,9)\)}\\
    r & \text{if \((m,p)=(2,2)\) and \(f\neq 3\)}\\
    7 & \text{if \((m,p,f)=(2,2,3)\)}.
\end{cases}
\]Then \(|\varphi_{1}(X_{uv}\cap M)Z/Z|_{d}>1\). Let us denote by \(d^{a}:=|\varphi_{1}(X_{uv}\cap M)|_{d}\) and note in particular that \(a=1\) if \((m,q)=(1,9)\). Then  \[|\Sp_{2m}(q)|_{d}=\begin{cases}3^{2}=d^{2a}& \text{if \((m,q)=(1,9)\)}\\ r^{2a}=d^{2a} &\text{if \(m,p)=(2,2)\) and \(f\neq3\) }\\7^{2}=d^{2a}& \text{if \((m,p,f)=(2,2,3)\)}.\end{cases}\] Hence 
 \(|X_{v}|_{d}=|\Sp_{2m}(q)|_{d}^{k}(k!)_{d}=d^{2ak}(k!)_{d}\). 
 On the other hand, 
\begin{equation}\nonumber
    |X_{uv}|_{d}\leqslant|\varphi_{1}(X_{uv}\cap M)|_{d}^{k}(k!)_{d}\leqslant|\Sp_{2m}(q)|_{d}^{k}(k!)_{d}=d^{ak}(k!)_{d}%(q^{2}-1)^{k}_{r}(r!)_{k}=r^{ak}(k!)_{r}.
\end{equation}Thus the valency of \(\Gamma\) has \(d\)-part 
\[%\nonumber
    \frac{|X_{v}|_{d}}{|X_{uv}|_{d}}\geqslant\frac{d^{2ak}(k!)_{d}}{d^{ak}(k!)_{d}}=d^{ak}.
\] Since \(\Gamma\) is \((G,3)\)-arc-transitive, we deduce from Lemma \ref{valency} that \(|G_{v}|\) is divisible by \(d^{3ak}\). Note that \(\Pi(\Out(L))=\Pi(f)\cup\{2\}\). Thus \(|\Out(L)|_{d}=(2f)_{d}=1\) and so\(|X_{v}|_{d}=|G_{v}|_{d}=r^{3ak}(k!)_{d}\). This leads to
\[%\nonumber
    d^{3ak}\leqslant|G_{v}|_{d}=d^{2ak}(k!)_{d},
\]
which implies that \(d^{ak}\leqslant(k!)_{p}\), contradicting Lemma \ref{sizeppd}. Thus none of cases (i)--(ii) is possible and so \(\varphi_{1}(X_{uv}\cap M)Z/Z\) is \(\PSp_{2m}(q)\), proving Claim~1.

\medskip

Then since \(\pi(X_{uv})\) is transitive, it follows that \(\ICF(X_{uv}\cap M)=\PSp_{2m}(q)'\) and this integer \(\ell:=\m_{X_{uv}\cap M}(\PSp_{2m}(q)')\) we have \(\ell\) divides \(k\).

\medskip
\noindent
\textit{Claim~2:}
\(\ell=k\).

\medskip
\noindent
Suppose for a contradiction that \(\ell<k\). In particular, \(\ell\leqslant k/2\). 
 Now \(\ppd(p,2mf)\neq\varnothing\) by Corollary \ref{existppd}. Let us take \(s\in\ppd(p,2mf)\) and set by \(s^{b}:=|\PSp_{2m}(q)|_{s}\). As \(\Pi(\Out(L))\subseteq\{2\}\cup\Pi(f)\) and \(s>2mf\geqslant 2f\), it follows that \(|\Out(L)|_{s}=1\) and 
\[
|\varphi_{1}(X_{uv}\cap M)|_{s}=|\varphi_{1}(X_{uv}\cap M)Z/Z|_{s}=|\PSp_{2m}(q)'|_{s}=|\PSp_{2m}(q)|_{s}. 
\]
Let us compute the \(s\)-part of \(|X_{uv}|\) and \(|X_{v}|\),
\(
    |X_{uv}|_{s}\leqslant|\PSp_{2m}(q)'|_{s}^{k/2}(k!)_{s}=s^{bk/2}(k!)_{s},
\)
and 
\(
 |X_{v}|_{s}=|\Sp_{2m}(q)|_{s}^{k}(k!)_{s}=s^{bk}(k!)_{s}
\).
Thus the valency of \(\Gamma\) has \(s\)-part
\[
    \frac{|X_{v}|_{s}}{|X_{uv}|_{s}}\geqslant\frac{s^{bk}(k!)_{s}}{s^{bk/2}(k!)_{s}}=s^{bk/2}.
\]
Since \(\Gamma\) is \((G,3)\)-arc-transitive, we deduce from Lemma \ref{valency} that \(|G_{v}|_{s}\) is divisible by \(s^{3bk/2}\). However, 
\[
    |G_{v}|_{s}\leqslant |\Out(L)|_{s}|L_{v}|_{s}=|\PSp_{2m}(q)'|_{s}^{k}(k!)_{s}=s^{bk}(k!)_{s}.
\]
This gives that \(s^{bk/2}\leqslant (k!)_{s}\), but this is impossible since we deduce by Lemma \ref{sizeppd} that \((k!)_{s}<s^{k/(s-1)}\leqslant s^{k/2}\), and the proof of claim is complete.

\medskip

Since \(\ell=k\), we have \(X_{uv}\cap M'\cong\Sp_{2m}(q)'^{k}\) and so \(M'\triangleleft X_{uv}\) and \(Z(M)\triangleleft X_{uv}\). Moreover,
\begin{equation}\nonumber
   M'/Z(M)\cong \PSp_{2k}(q)'^{k}
\end{equation}
is a minimal normal subgroup of \(X_{uv}/Z(M)\) as \(\pi(X_{uv})\) is transitive. Let \(N:=M'^{g}\). Since \(X_{vw}=X_{uv}^{g}\), we deduce that \(N/Z(M)^{g}\)(\(\cong\PSp_{2m}(q)'^{k}\)) is a minimal normal subgroup of \(X_{vw}/Z(M)^{g}\). It follows from  \(X_{vw}\cap M\triangleleft X_{vw}\) that \((X_{vw}\cap M)Z(M)^{g}/Z(M)^{g}\triangleleft X_{vw}/Z(M)^{g}\). Thus either \((N\cap M)Z(M)^{g}/Z(M)^{g}=1\), or \(N/Z(M)^{g}\leqslant(X_{vw}\cap M)Z(M)^{g}/Z(M)^{g} \). Suppose that \((N\cap M)Z(M)^{g}/Z(M)^{g}=1\). Then \(N\cap M\leqslant Z(M)^{g}\), which is soluble. Thus \(\ICF(N)=\ICF(N/(N\cap M))\) and so \(|N/(N\cap M)|\) is divisible by \(|\PSp_{2m}(q)'|^{k}\). Since \(N/(N\cap M)\lesssim \Sy_{k}\), this implies that \(k!\) is divisible by \(|\PSp_{2m}(q)'|^{k}\), which is impossible. Hence \(N/Z(M)^{g}\leqslant (X_{vw}\cap M)Z(M)^{g}/Z(M)^{g}\) and so \((N\cap M)Z(M)^{g}=N\). Then
\begin{equation}\nonumber
    N\cap M\geqslant (N\cap M)'=((N\cap M)Z(M)^{g})'=N'=N,
\end{equation}
and so \(N\leqslant M\). Thus \(N=N'\leqslant M'\) and so \(M'^{g}=M'\) due to \(|N|=|M'|\), contradicting Lemma \ref{normal}. Thus \(\Gamma\) is not \((G,3)\)-arc-transitive and this completes the proof.
\end{proof}

\subsubsection{\(\mathcal{C}_{2}\)-subgroups of Type (II)}
Recall that if \(G_{v}\) is a \(\mathcal{C}_{2}\)-subgroup of Type II of \(G\), then \(n\geqslant2\), \(q\) is odd, and  \(L_{v}=\GL_{n}(q).2/\{\pm 1\}\).

\begin{proposition}\label{C21}
Suppose that Hypothesis \ref{general} holds. Then \(G_{v}\) is not a maximal \(\mathcal{C}_{2}\)-subgroup of Type II. %Then \(\Gamma\) is not \((G,2)\)-arc-transitive.
\end{proposition}

\begin{proof}
Suppose that Hypothesis \ref{general} holds and \(G_{v}\) is a maximal \(\mathcal{C}_{2}\) subgroup such that \(L_{v}=\GL_{n}(q).2/\{\pm1\}\), with \(n\geqslant2\) and $q$ odd. 
By Proposition~\ref{citesmall} it follows that $G_v$ has a unique insoluble composition factor $\PSL_n(q)$, so that \cite[Lemma 3.3]{small} applies, and we note in particular that 
\[
\PSL_n(q)\not\in\{ A_6, M_{11}, M_{12}, \Sp_4(2^f), \POmega_8^+(q), \PSU_3(8),\PSU_4(2)\},
\] 
and that if $n=2$ then $q\geqslant 41$, and if $n=3$ then $q\geqslant7$.   Note also that \(L=\PSp_{2n}(q)\) and \(G_{v}^{(\infty)}\) is  quasisimple such that \(G_{v}^{(\infty)}/Z(G_{v}^{(\infty)})\cong\PSL_{n}(q)\). 
Moreover if both $G_{uv}$ and $G_{vw}$ have $\PSL_n(q)$ as a composition factor then both $G_{uv}$ and $G_{vw}$ contain $\SL_n(q)/(\{\pm 1\}\cap \SL_n(q))$  (a normal subgroup of $G_v$) and it follows that the element $\overline{g}$ in Hypothesis~\ref{general} normalises this normal subgroup, contradicting Lemma~\ref{normal}. Thus 
part (c) of \cite[Lemma 3.3]{small} does not hold, and so \cite[Lemma 3.3]{small} implies that the simple group $\PSL_n(q)$ lies in one of the families $\mathcal{T}_1$ or $\mathcal{T}_2$ defined in \cite[(2.1)]{small}. On checking the definitions of these families against the above display and noting that \(q\) is odd, we see that $\PSL_n(q)\in\mathcal{T}_2$, and so by \cite[Lemmas 2.7 and 3.3]{small}, setting \(\overline{G_{v}}=G_{v}/\R(G_{v})\) (recall that \(\R(G_{v})\) is the largest soluble normal subgroup of \(G_{v}\)), \(\overline{G_{uv}}=G_{uv}\R(G_{v})/\R(G_{v})\) and \(\overline{G_{vw}}=G_{vw}\R(G_{v})/\R(G_{v})\), then interchanging \(\overline{G_{uv}}\)  and \(\overline{G_{vw}}\) if necessary,
the following holds:

\begin{center}
\(n=2, q\geqslant 41\),  \(\overline{G_{uv}}\cap\PSL_{2}(q)\leqslant \D_{q+1}\) and \(\overline{G_{vw}}\cap \PSL_{2}(q)\leqslant \E_{q}:\C_{\frac{q-1}{2}}\),   
\end{center}
where \(\E_{q}\) is an elementary abelian group of order \(q\).

\noindent
Since $q=p^f$ is odd, we have
$|\overline{G_{uv}}\cap \PSL_{2}(q)|_{p}=|\D_{q+1}|_{p}=1$ and 
\begin{equation*}
    |\overline{G_{uv}}|_{p}\leqslant |\overline{G_{uv}}\cap\PSL_{2}(q)|_{p}\cdot|\Out(\PSL_{2}(q))|_{p}=(2f)_{p}=(f)_{p}.
\end{equation*}
On the other hand, for the soluble radical $R(G_v)$, the order \(|\R(G_{v}))|\) divides \(|Z(G_{v})|\cdot|\Out(L)|\) and so 
\begin{equation*}
    |\R(G_{v})|_{p}\leqslant|Z(G_{v})|_{p}\cdot|\Out(L)|_{p}=(q-1)_{p}(2f)_{p}=(f)_{p}.
\end{equation*}
This implies that 
\begin{equation}\label{48}
    |G_{uv}|_{p}^{2}\leqslant(|\overline{G_{uv}}|_{p}\cdot|\R(G_{v})|_{p})^{2}\leqslant (f)_{p}^{4}, 
\end{equation}
which in turn leads to $p^{4f/(p-1)}>(f!)_{p}^{4}\geqslant(f)_{p}^{4}\geqslant|G_{uv}|_{p}^{2}\geqslant|G_{v}|_{p}\geqslant|\GL_{n}(q)|_{p}\geqslant|\GL_{2}(q)|_{p}=p^{f},$
and hence $4>p-1$, so $p\leqslant 3$. Since \(q=p^{f}\) is odd, we have \(p=3\), and since \(q=3^{f}\geqslant 41\), also \(f\geqslant 4\). 
If \(f\equiv\pm1\pmod{3}\), then \((f)_{3}=1\), and so by \eqref{48}, \(|G_{uv}|_{3}^{2}=1\). On the other hand, \(|G_{v}|_{3}\geqslant |\GL_{n}(3^{f})|_{3}\geqslant 3^{f}>1=|G_{uv}|_{3}^{2}\), which is a contradiction. Therefore, \(f\equiv0\pmod{3}\), so $f=3k\geqslant6$. 

We claim that 
\((f)^{2}_{3}\leqslant (f!)_{3}\):  
if $f\not\equiv 0\pmod{9}$, then  \((f)^{2}_{3}=3^2\leqslant (f!)_{3}\). Similarly if $9$ divides $f$ but $f\not\equiv 0\pmod{27}$, then  \((f)^{2}_{3}=3^4\leqslant (f!)_{3}\). Thus we may assume that \((f)_{3}=3^{a}\) with \(a\geqslant 3\), so that $f=3^ak'$ with $k'$ coprime to $3$. Then  $f!$ is divisible by 
\[
3^ak'\cdot 3^{a-1}k' \dots  3k' = 3^{(a+1)a/2}\cdot(k')^a
\]
and it follows, since $(a+1)a/2\geqslant 2a$, that \( (f!)_{3} \geqslant  3^{2a}= (f)^{2}_{3}\), proving the claim.

Using the upper bound in  \eqref{48}, we have \(|G_{uv}|_{3}^{2}\leqslant (f)_{3}^{4}\leqslant (f!)_{3}^{2}<3^{f}\). On the other hand, \(|G_{v}|_{3}\geqslant |\GL_{n}(3^{f})|_{3}\geqslant 3^{f}\), so \(|G_{uv}|_{3}^{2}<|G_{v}|_{3}\), which is a contradiction. Thus Proposition~\ref{C21} is proved.

%Thus case (i) holds so  \((n,q)=(3,8)\).
%\textcolor{red}{CP: Lei can you check if this case has to be considered.}

\end{proof}

\subsection{\(\mathcal{C}_{3}\)-subgroups}
Here \(G_{v}\) is a maximal \(\mathcal{C}_{3}\)-subgroup of the almost simple group $G$ where \(L=\Soc(G)=\PSp_{2n}(q)'\) with $q=p^f$ for a prime $p$, so by \cite[Propositions 4.3.7 and 4.3.10]{kleidman} and \cite[Table 3.5.C]{kleidman},  either 
\begin{enumerate}
    \item[(I)]  $G_v$ is a Type (I) subgroup, that is,  \(L_v=\PSp_{2a}(q^{b}).b\) where \(n=ab\) with \(a\geqslant 1\) and \(b\) a prime; or 
    \item[(II)] \(G_v\) is a Type (II) subgroup, that is,  \(L_v=\GU_{n}(q).2/\{\pm 1\}\) where \(n\geqslant2\) and \(q\) is odd.
\end{enumerate}

\begin{proposition}\label{p:c3I}
Suppose that Hypothesis \ref{general} holds. If \(G_{v}\) is a \(\mathcal{C}_{3}\)-subgroup of Type (I), as above, then \(\Gamma\) is not \((G,3)\)-arc-transitive, and moreover the parameters \((p,a)=(2,2)\). 
\end{proposition}

\begin{proof}
Let \(\overline{H}:=H\R(G_{v})/\R(G_{v})\) for any subgroup \(H\leqslant G_{v}\). Then \(\overline{G_{v}}\) is almost simple with socle \(\PSp_{2a}(q^{b})\) and \(\Pi(\R(G_{v}))\subseteq\Pi(\Out(L))\). Note that \(G_{v}^{(\infty)}\) is quasisimple, so we deduce from \cite[Lemma 3.3 (b)]{small} that \(\overline{G_{v}}=\overline{G_{uv}}\,\overline{G_{vw}}\) is a core-free factorisation and the list of possibilities for \(\{\overline{G_{uv}}, \overline{G_{vw}}\}\) is given in \cite[Lemma 2.6, Lemma 2.7]{small}. Thus, interchanging $\overline{G_{uv}}, \overline{G_{vw}}$ if necessary, one of the following holds:
\begin{description}
\item[(i)]\(a=1, b>1\), \(\overline{G_{uv}}\cap \PSp_{2}(q^{b})\leqslant \D_{2(q^{b}+1)/d}\) and \(\overline{G_{vw}}\cap\PSp_{2}(q^{b})\leqslant \E_{q^{b}}:\C_{\frac{q^{b}-1}{d}}\), where \(\E_{q^{2}}\) is an elementary abelian group of order \(q^{b}\) and \(d:=(2,q-1)\).
\item[(ii)]\((q,a,b)=(3,1,2)\), \(\overline{G_{uv}}\cap\PSp_{2}(9)=\A_{5}\) and \(\overline{G_{vw}}\cap\PSp_{2}(9)=\PSL_{2}(5)\).
\item[(iii)]\((p,a)=(2,2)\), 
\[
\PSp_{2}(2^{2bf})\leqslant \overline{G_{uv}}\cap\PSp_{4}(2^{bf})\leqslant\PSp_{2}(2^{2bf}).2\quad  \text{and}\quad \PSp_{2}(2^{2bf})\leqslant \overline{G_{vw}}\cap\PSp_{4}(2^{bf})\leqslant\PSp_{2}(2^{2bf}).2.\]
\end{description}
It follows from Proposition \ref{citesmall}  that case (ii) does not occur; in case (i), if $b=2$ then $q\geqslant41$, if $b=3$ then $q\geqslant7$, and $(b,q)\ne (4,2)$; and in case (iii), $(b,p^f)\ne (2, 2)$. Consider first case (i). Here, by Corollary \ref{existppd}, \(\ppd(p,2bf)\neq\varnothing\), so let  \(r\in\ppd(p,2bf)\). Then \(r\) does not divide \( p^{bf}-1=q^{b}-1\), and again by Corollary~\ref{existppd},  \(r>2bf\) so \(|\Out(\PSp_2(q^b))|_r=1\). Then 
\begin{equation*}
|\overline{G_{vw}}|_{r}\leqslant |\overline{G_{vw}}\cap \PSp_{2}(q^{b})|_{r}\cdot|\Out(\PSp_{2}(q^{b}))|_{r}\leqslant (p^{bf}-1)_{r}=1.
\end{equation*} 
As \(\Pi(\R(G_{v}))\subseteq\Pi(\Out(L))\) and \(\Pi(\Out(L))\subseteq \Pi(f)\cup\{2\}\), it follows that \(|\R(G_{v})|_{r}=1\) and so
\(|G_{vw}|_{r}\leqslant|\overline{G_{vw}}|_{r}\cdot|\R(G_{v})|_{r}=1\), which contradicts the fact that \(r\in\Pi(G_{v})=\Pi(G_{vw})\).

 Thus case (iii) holds, so \((p, a)=(2, 2)\), which is the first assertion of the lemma.  Suppose for a contradiction that \(\Gamma\) is \((G,3)\)-arc-transitive. By Corollary \ref{existppd}, \(\ppd(2,bf)\neq\varnothing\), and we let  
 \(r\in\ppd(2,bf)\) and \(r^{d}:=(2^{bf}-1)_{r}\). Then, again by  Corollary \ref{existppd}, $r>bf$ so \(|\Out(\PSp_{4}(2^{bf}))|_{r}=1\).   Thus \(|\PSp_{4}(2^{bf})\cap\overline{G_{uv}}|_{r}=r^{d}\) and \(|\PSp_{4}(2^{bf})|_{r}=r^{2d}\). Therefore 
 \[
 |\overline{G_{uv}}|_{r}\leqslant |\PSp_{4}(2^{bf})\cap\overline{G_{uv}}|_{r}\cdot|\Out(\PSp_{4}(2^{bf}))|_{r}=r^{d}.
 \]
 On the other hand, note that \(|\R(G_{v})|\) divides \(|\Out(\PSp_{2n}(q))|=|\Out(\PSp_{2n}(2^{f}))|=f\) (since $n=2b>2$). Since \(r>bf\) it follows that \(|\R(G_{v})|_{r}=1\), which implies that 
 $ |G_{uv}|_{r}=|\overline{G_{uv}}|_{r}\cdot|\R(G_{v})|_{r}=r^{d}$
 and $|G_{v}|_{r}=|\PSp_{4}(2^{bf})|_{r}\cdot|\R(G_{v})|_{r}=r^{2d}$.
 Thus the valency of \(\Gamma\) has \(r\)-part:
 \begin{equation*}
     \frac{|G_{v}|_{r}}{|G_{uv}|_{r}}=\frac{r^{2d}}{r^{d}}=r^{d}.
 \end{equation*}
 Since \(\Gamma\) is \((G,3)\)-arc-transitive, it follows from Lemma \ref{valency} that \(|G_{v}|\) is divisible by \(r^{3d}\), contradicting the fact that \(|G_{v}|_{r}=r^{2d}\) with \(d>0\). Thus the result is proved.
\end{proof}

Recall that if \(G_v\) is a \(\mathcal{C}_{3}\)-subgroup of Type (II), then  \(L_v=\GU_{n}(q).2/\{\pm 1\}\) where \(n\geqslant2\) and \(q\) is odd.

\begin{proposition}\label{p:c3II}
Suppose that Hypothesis \ref{general} holds. Then \(G_{v}\) is not a \(\mathcal{C}_{3}\)-subgroup of Type (II).
\end{proposition}

\begin{proof}
 Let \(\overline{H}:=H\R(G_{v})/\R(G_{v})\) for any subgroup \(H\) of \(G_{v}\). By Proposition~\ref{citesmall}, $(n,q)\ne (2,3)$, and hence \(\overline{G_{v}}\) is almost simple with socle \(\PSU_{n}(q)\) and \(|\R(G_{v})|\) divides \(|Z(L_{v})|\cdot|\Out(L)|\). By \cite[Lemma 3.3]{small} the  factorisation \(\overline{G_{v}}=\overline{G_{uv}}\,\overline{G_{vw}}\) is core-free. Note that \(\PSU_2(q)\cong\PSL_2(q)\) and by \cite[Lemma 5.4]{small} there is no vertex-primitive \((G,2)\)-arc-transitive digraph for an almost simple group $G$ with socle \(\PSp_{4}(q)\) such that \(q=p^{f}\leqslant 37\). Thus either \(n\geqslant3\), or \(n=2\) with \(q> 37\). It  therefore follows from \cite[Lemmas 2.6 and 2.7]{small} that \(n=2\), \(\overline{G_{uv}}\cap \PSU_{2}(q)\leqslant\D_{q+1}\) and \(\overline{G_{vw}}\cap\PSU_{2}(q)\leqslant \E_{q}:\C_{\frac{q-1}{2}}\), where \(\E_{q}\) is an elementary abelian group of order \(q\).

Thus \(|G_{vw}|_{p}\geqslant|\overline{G_{vw}}|_{p}=q=p^{f}\). On the other hand, \(|\R(G_{v})|\) divides \(\frac{q+1}{2}|\Out(L)|\) (as \(|Z(L_{v})|=\frac{q+1}{2}\)) and \(|\Out(L)|\) divides \(4f\), and hence  \(|\R(G_{v})|\) divides \(2f(q+1)\) and
\begin{equation*}
    |G_{uv}|_{p}\leqslant|\overline{G_{uv}}|_{p}\cdot|\R(G_{v})|_{p}=|\R(G_{v})|_{p}=(f)_{p}\leqslant (f!)_{p}< p^{f/(p-1)}\leqslant p^{f/2} <|G_{vw}|_{p} 
\end{equation*}
which is a contradiction and proving the result.
\end{proof}

\subsection{\(\mathcal{C}_{4}\)-subgroups}\label{c4construction}

In this section we assume that $G_v$ is a maximal \(\mathcal{C}_{4}\)-subgroup, as in  \cite[Definition on page 128 and Proposition 4.4.11]{kleidman}, and we describe this situation as follows. 
Let \(V\) be a symplectic space of dimension \(2n\) over \(\mathbb{F}_{q}\) where \(q=p^{f}\) is odd, and let \(B\) be a non-degenerate alternating bilinear form on \(V\). Let \(U\) and \(W\) be  vector spaces over \(\mathbb{F}_{q}\) such that \(\dim U=2k\), \(\dim W=m\) and \(n=km\); and let \(B_{1}\) be an alternating bilinear form on \(U\)  and \(B_{2}\) be a symmetric bilinear form on \(W\). Define \(U\otimes W=\Span\{u\otimes w| u\in U, w\in W\}\), and define \((B_{1}\otimes B_{2})(u_{1}\otimes w_{1}, u_{2}\otimes w_{2}) := B_{1}(u_{1}, u_{2})B_{2}(w_{1}, w_{2})\) for \(u_{1},u_{2}\in U\) and \(w_{1}, w_{2}\in W\), and extend \(B_{1}\otimes B_{2}\) linearly to all of \(U\otimes W\). Then \((V, B)\cong(U\otimes W, B_{1}\otimes B_{2})\) is called a tensor decomposition of $V$. Suppose that \(L=\PSp_{2n}(q)\leqslant G\leqslant\Aut(\PSp_{2n}(q))\). Then the stabiliser in $G$ of the tensor decomposition \((U\otimes W, B_{1}\otimes B_{2})\) is a (usually) maximal \(\mathcal{C}_{4}\)-subgroup of \(G\), and its structure is given by \cite[Proposition 4.4.11]{kleidman}, see also Remark~\ref{rem:C4}). We use the following specialised version of Hypothesis~\ref{general} for the  \(\mathcal{C}_{4}\)  analysis.

\begin{hypothesis}\label{C4}
\rm{ Assume that the notation and assumptions of Hypothesis~\ref{general} hold, and in addition $G_v$ is the stabiliser of a tensor decomposition, as above so that, by \cite[Proposition 4.4.11]{kleidman} and \cite[Lemma 3.7.6]{holt}, $q$ is odd, and \(\Soc(K_{1})\times\Soc(K_{2})\leqslant G_{v}\leqslant K_{1}\times K_{2}\) where \(\Soc(K_{1})\cong\PSp_{2k}(q)\) and \(\Soc(K_{2})\cong \POmega_{m}^{\epsilon}(q)\), with \(km=n\) and \(m\geqslant3\), excluding the cases \((k,m,\epsilon)=(1,4,+)\) and \((m,q)=(3,3)\).
Moreover, let \(\pi_{i}\) be the projection from \(G_{v}\) to \(K_{i}\) for \(i=1,2\).
}
%Let \(\Gamma\) be a \(G\)-vertex-primitive \((G,s)\)-arc-transitive digraph of valency at least \(3\), where $s\geqslant2$ and \(G\) is almost simple with socle \(L = \PSp_{2n}(q)\) with $n\geqslant2$ and $q=p^f$ for some prime $p$. Take an arc \(u\to v\) of \(\Gamma\). Let \(g\) be an element of \(G\) such that \(u^{g}=v\) and let \(w=v^{g}\). Then \(u\to v\to w\) is a \(2\)-arc in \(\Gamma\). Suppose that \(G_{v}\) is a \(\mathcal{C}_{4}\) subgroup of \(G\). Then by \cite[Proposition 4.4.11]{kleidman} \(G_{v}\) is a tensor product stabiliser such that \(\Soc(K_{1})\times\Soc(K_{2})\leqslant G_{v}\leqslant K_{1}\times K_{2}\) where \(\Soc(K_{1})\cong\PSp_{2k}(q)\) and \(\Soc(K_{2})\cong \POmega_{m}^{\epsilon}(q)\), where \(km=n\) and \(m\geqslant3\), excluding the cases \((k,m,\epsilon)=(1,4,+)\) and \((m,q)=(3,3)\). Moreover, let \(\pi_{i}\) be the projection from \(G_{v}\) to \(K_{i}\) for \(i=1,2\).}
\end{hypothesis}
\begin{remark}\label{rem:C4}
{\rm 
We exclude the cases \((k,m,\epsilon)=(1,4,+)\) and \((m,q)=(3,3)\) in Hypothesis \ref{C4} because in these cases the $\mathcal{C}_{4}$-subgroup $G_v$ is not maximal as required for $G$ to be vertex-primitive on $\Gamma$: if \((m,q)=(3,3)\) then, by \cite[p 220 line 17]{kleidman},  \(G_{v}\) is contained in a \(\mathcal{C}_{2}\)-subgroup of \(G\); while if \((k,m,\epsilon)=(1,4,+)\), then \(n=2km=8\) and \(L=\PSp_{8}(q)\), and \((\Sp_{2}(q)\circ \GO^{+}_{4}(q)).2\) is properly contained in a \(\mathcal{C}_{7}\)-subgroup of \(\Sp_{8}(q)\), see \cite[Lemma 3.7.6]{holt} which lists all the maximal subgroups of \(\Sp_{8}(q)\). 
}
\end{remark}
% {\color{red}CP: expository problem: in Hypothesis~\ref{C4}, we exclude $(m,2k,\epsilon)=(4,4,+)$ so this does not appear in the statement of Prop~\ref{p:c4-5}, but we should pick it up in the main theorem or and remarks about trying to prove `not $(L,2)$-arc-trans. Mustn't forget this. CEP I see it is mentioned in Prop \ref{p:c4-5} - is this case considered in the proof?}

Our objective is to prove the following result.

\begin{proposition}\label{p:c4-5}
Suppose that Hypothesis \ref{C4} holds. Then $\Gamma$ is not  $(G,3)$-arc-transitive. Moreover $\Gamma$ is not $(L,2)$-arc-transitive except possibly for $(m,2k,\epsilon) = (5,4,\circ), (3,2,\circ)$, or $(4,4,\pm)$. 
\end{proposition}

First we establish some group theoretic facts. Recall from Hypothesis~\ref{general} that $\overline{g}\in L$ maps $(u,v)^{\overline{g}}=(v,w)$.

\begin{lemma}\label{corefree}\label{divisibility} \label{tensor}
Suppose that Hypothesis \ref{C4} holds and  $\Gamma$ is $(L,2)$-arc-transitive, and let  $i\in\{1,2\}$. 

\begin{description}
        \item[(i)]  If \(H:=\Soc(K_{i})\) is a non-abelian simple group and \(K_{3-i}\) has no section isomorphic to \(H\), then \(\pi_{i}(G_{v})=\pi_{i}(G_{uv})\pi_{i}(G_{vw})\) is a core-free factorisation.

        \item[(ii)] Suppose that $r$ is a prime such that \(r\)  divides \(|\Soc(K_{i})|\) but $r$ does not divide \(|K_{3-i}|\). Then $r$ divides both \(|\pi_{i}(G_{uv})|\) and \(|\pi_{i}(G_{vw})|\). Moreover, if  \(T\in\ICF(\pi_{i}(G_{uv}))\) and \(r\) divides \(|T|\), then \(T\in\ICF(\pi_{i}(G_{vw}))\).

        \item[(iii)] If  \(2k\neq m\), then \(\Soc(K_{1})^{\overline{g}}\neq\Soc(K_{2})\).
\end{description}
\end{lemma}

\begin{proof}
(i)\ Suppose to the contrary that the assumptions of (i) hold for $i$ and \(\pi_{i}(G_{v})=\pi_{i}(G_{uv})\pi_{i}(G_{vw})\) is not core-free. Note that \( H\unlhd \pi_i(G_v)\leqslant \Aut(H)\), so \(H= \Soc(\pi_{i}(G_{v}))\). Then without loss of generality, \(H\unlhd \pi_{i}(G_{uv})\), so \(H\in\ICF(G_{uv})\). Note that \[\ICF(G_{uv})=\ICF(G_{uv}\cap K_{i})\cup\ICF(\pi_{3-i}(G_{uv})) \] and \(\pi_{3-i}(G_{uv})\leqslant K_{3-i}\). Then \(H\notin\ICF(\pi_{3-i}(G_{uv}))\) since \(K_{3-i}\) has no section isomorphic to \(H\), and hence  \(H\in\ICF(K_{i}\cap G_{uv})\). This implies that \(H\leqslant G_{uv}\) and hence that \(H^{\overline{g}}\leqslant G_{vw}\). 
On the other hand, \(H^{\overline{g}}/(H^{\overline{g}}\cap K_{i})\cong\pi_{3-i}(H^{\overline{g}}).\) Since \(H^{\overline{g}}\cong H\) is a non-abelian simple group and \(K_{3-i}\) has no section isomorphic to \(H\), it follows that \(\pi_{3-i}(H^{\overline{g}})=1\). Hence \(H^{\overline{g}}\leqslant K_{i}\). Since $H$ is the unique subgroup of \(K_{i}\) isomorphic to \(H\), we conclude that \(H^{\overline{g}}=H\), contradicting Lemma \ref{normal}. This proves part (i).

\medskip
(ii)\ Since \(r\in\Pi(\Soc(K_{i}))\), it follows that \(r\in\Pi(G_{v})=\Pi(G_{uv})=\Pi(G_{vw})\). Also since  \(|K_{3-i}|_{r}=1\), we have
\[
|G_{uv}|_{r}=|K_{3-i}\cap G_{uv}|_{r}\cdot |\pi_{i}(G_{uv})|_{r}\leqslant|K_{3-i}|_{r}\cdot |\pi_i(G_{uv})|_{r}=|\pi_i(G_{uv})|_{r}\leqslant |G_{uv}|_{r},
\] 
so  \(|\pi_{i}(G_{uv})|_{r}=|G_{uv}|_{r}>1\). The same argument gives \(|\pi_{i}(G_{vw})|_{r}>1\). Now suppose that  \(T\in\ICF(\pi_{i}(G_{uv}))\) and \(|T|_{r}>1\). Since \(G_{uv}\cong G_{vw}\), it follows that \(T\in\ICF(\pi_{i}(G_{vw}))\). On the other, hand \(T\notin\ICF(\pi_{3-i}(G_{vw}))\) as \(|K_{3-i}|_{r}=1\). Thus \(T\in\ICF(\pi_{i}(G_{vw}))\), and part (ii) is proved.

\medskip 
(iii)\ By Hypothesis~\ref{C4}, $G_v$ is the stabiliser of the tensor decomposition \((V, B)\cong (U\otimes W, B_{1}\otimes B_{2})\), with \(\dim U=2k\), \(\dim W=m\), \(n=km\) and \(U, W, B_{1}, B_{2}\) as above. Suppose that $2k\ne m$, and let \(\mathcal{B}_{1}=\{u_{1},\ldots, u_{2k}\}\) and \(\mathcal{B}_{2}=\{w_{1},\ldots, w_{m}\}\) be bases for \(U\) and \(W\), respectively. Then \(K_{1}\) and \(K_{2}\) stabilise the direct sum decompositions \(\mathcal{D}_{1}=(U\otimes w_{1})\oplus\cdots\oplus(U\otimes w_{m})\) and \(\mathcal{D}_{2}=(u_{1}\otimes W)\oplus\cdots\oplus(u_{2k}\otimes W)\), respectively.
Moreover, \(\Soc(K_{i})\) fixes and acts irreducibly on each subspace in  \(\mathcal{D}_{i}\) for \(i=1,2\). Thus all the irreducible \(\Soc(K_{1})\)-submodules of \(V\) have dimension \(2k\) and all the irreducible \(\Soc(K_{2})\)-submoules of \(V\) have dimension \(m\). Since \(2k\neq m\),  the element \(\overline{g}\) does not conjugate \(\Soc(K_{1})\) to \(\Soc(K_{2})\).
\end{proof}

Next we prove some tight restrictions on the parameters $k, m$ and $\epsilon$.

\begin{lemma}\label{l:c4-1}%\label{l:c4-2}
Suppose that Hypothesis \ref{C4} holds and that  $\Gamma$ is $(L,2)$-arc-transitive.  Then  $m, 2k, \epsilon$ satisfy one of the lines of Table~$\ref{t:c4}$.
\end{lemma}

\begin{table}
    \centering
    \begin{tabular}{cccc}
    \hline
        Line& \(m\)&\(2k\)&$\epsilon$  \\
         \hline
         $1$& \(5\)&\(4\)& \(\circ\)\\
         $2$& \(3\)&\(2\)& \(\circ\)\\
         $3$& \(8\)&\(2\) or \(4\)& \(+\)\\
         $4$& \(4\)&\(4\)& \(-\)\\
         \hline
    \end{tabular}
    \caption{Parameters for Lemma~\ref{l:c4-1}}
    \label{t:c4}
\end{table}
\begin{proof}
We prove the lemma via a series of three claims. Recall that $q$ is odd.

\medskip\noindent
\emph{Claim $1$,\quad One of the following holds: $m>2k$, or $(m,\epsilon)=(2k,-)$.}

Suppose that the claim is false. Then either $m<2k$, or $(m,\epsilon)=(2k,+)$. Since $m\geqslant3$ we have $2k\geqslant 4$, and so, by Corollary~\ref{existppd}, \(\ppd(p, 2kf)\neq\varnothing\) and we choose a prime $r\in\ppd(p, 2kf)$. Then $r$ divides $|\Soc(K_1)|$, and $r$ does not divide $|K_2|$.
Thus, by Lemma \ref{divisibility}(ii), $r$ divides both \(|\pi_{1}(G_{uv})|\) and \(|\pi_{1}(G_{vw})|\). Also \(\Soc(K_{1})=\PSp_{2k}(q)\) is a non-abelian simple group, and is not a section of $K_2$, so, by Lemma \ref{corefree}(i),   \(\pi_{1}(G_{v})=\pi_{1}(G_{uv})\pi_{1}(G_{vw})\) is a core-free factorisation,  However, an almost simple group with socle \(\PSp_{2k}(q)\) does not have a core-free factorisation with both subgroups having orders divisible by  \(r\) by Lemma \ref{sp}(i). Thus we have a contradiction, and Claim 1 is proved.

\medskip\noindent
\emph{Claim $2$,\quad  Either $(m,\epsilon)=(2k+2,+), (2k+1,\circ )$ or $(2k,-)$, or $(m,\epsilon)=(8,+)$ with $k\in\{1,2,3\}$ .}

By Claim 1, either $m > 2k$ or $(m,\epsilon)=(2k,-)$. 
Suppose that  $(m,\epsilon)$ is not one of the four pairs listed in Claim 2. Then either  
$m\geqslant 2k+3$ or $(m,\epsilon)=(2k+2,-)$. Note that \(m\geqslant4\), and, in particular, $(m,\epsilon)$ is not $(4,+)$ or $(8,+)$. Let
\[
d=\begin{cases}
m-1&\text{if \(m\) is odd and \(\epsilon=\circ\)} \\
m&\text{if \(m\) is even and \(\epsilon=-\)}\\
m-2&\text{if \(m\) is even and \(\epsilon=+\)}.
\end{cases}
\]
If \(\epsilon\neq-\) then \(m\geqslant 2k+3\geqslant5\) and so \(d\geqslant m-2\geqslant3\), while if \(\epsilon=-\) then \(d=m\geqslant4\). In both cases \(d\geqslant3\) and since \(p\) is odd,  \(\ppd(p,fd)\neq\varnothing\)  by Corollary~\ref{existppd}. Let \(r\in\ppd(p,fd)\). Then \(|\Soc(K_{2})|_{r}>1\) and \(|K_{1}|_{r}=1\). By Lemma \ref{divisibility}(ii), the prime \(r\) divides \(|\pi_{2}(G_{uv})|\) and \(|\pi_{2}(G_{vw})|\).
On the other hand, \(K_{1}\) has no section isomorphic to \(\Soc(K_{2})\) as \(r\nmid|K_{1}|\), so by Lemma \ref{corefree}(i), the  factorisation \(\pi_{2}(G_{v})=\pi_{2}(G_{uv})\pi_{2}(G_{vw})\) is core-free. As noted above, $(m,\epsilon)$ is not $(4,+)$ or $(8,+)$. Hence, since \(|\pi_{2}(G_{uv})|_{r}, |\pi_{2}(G_{vw})|_{r}>1\), the hypotheses of Lemma \ref{odd}(i) hold provided  $ (2k,m,\epsilon)\ne (2,5,\circ)$.  Thus it follows from Lemma \ref{odd}(i)  that one of:
\begin{center}
    (i)\  \(m=7\);\quad (ii)\ \((m,\epsilon,q)=(4,-,3)\); \quad (iii)\ \((m,\epsilon)=(4t,+)\);\quad or\quad 
     (iv)\ $ (2k,m,\epsilon)= (2,5,\circ)$.
\end{center}
% \begin{description}
%     \item[(i)] \(m=7\);
%     \item[(ii)]\((m,\epsilon,q)=(4,-,3)\);
%     \item[(iii)]\((m,\epsilon)=(4t,+)\).
% \end{description}
We first consider cases (i) and (iii). Here \((\ICF(\pi_{2}(G_{uv})),\ICF(\pi_{2}(G_{vw})))\) are as in rows \(1-2\), \(7-11\) or \(23-27\) in Table \ref{O(2m+1,q)}. Then by part (iii) of Lemma~\ref{odd}, \(\ICF(\pi_{2}(G_{uv}))\cap\ICF(\pi_{2}(G_{vw}))=\varnothing\). On checking these rows in Table \ref{O(2m+1,q)}, and interchanging \(\pi_{2}(G_{uv})\) and \(\pi_{2}(G_{vw})\) if necessary, there exists \(T\in\ICF(\pi_{2}(G_{uv}))\) such that \(|T|_{r}>1\). However, this is a contradiction to Lemma \ref{divisibility}(ii).

Next we consider case (ii). Here \((m,\epsilon,q)=(4,-,3)\), and hence \(2k=2\) and so \(L=\PSp_{8}(3)\) and \(|\Aut(L):L|=2\). By \cite[Theorem 3.8]{Wilson} we have \(L_{v}\cong(\PSp_{2}(3)\times\PGO_{4}^{-}(3)).2\) and therefore 
\[
(\PSp_{2}(3)\times\PGO_{4}^{-}(3)).2\leqslant G_{v} \leqslant(\PSp_{2}(3)\times\PGO_{4}^{-}(3)).[4].
\] 
However, we check by MAGMA \cite{magma} (see Subsection~\ref{r:magma}) that there is no homogeneous factorisation \(G_{v}=G_{uv}G_{vw}\) with \(G_{uv}\) and \(G_{vw}\) conjugate in \(G\), which is a contradiction.

Finally we consider  case (iv)\ $ (2k,m,\epsilon)= (2,5,\circ)$.
Note that \(\Soc(K_{2})\cong\POmega_{5}(q)\cong\PSp_{4}(q)\). Since \(K_{1}\cong\PSL_{2}(q)\) has no section isomorphic to \(\Soc(K_{2})\cong\PSp_{4}(q)\), it follows from Lemma \ref{divisibility}(i) that \(\pi_{2}(G_{v})=\pi_{2}(G_{uv})\pi_2(G_{vw})\) is a core-free factorisation. By Corollary~\ref{existppd}, \(\ppd(p,4f)\neq\varnothing\) and we choose \(r\in\ppd(p,4f)\).  Then \(|\Soc(K_{2})|_{r}>1\) and \(|K_{1}|_{r}=1\).  It follows from Lemma~\ref{divisibility} (ii) that \(r\) divides \(|\pi_{2}(G_{uv})|\) and \(|\pi_{2}(G_{vw})|\).    However,  by Lemma \ref{sp} (i),  there is no core-free factorisation \(\pi_{2}(G_{v})=\pi_{2}(G_{uv})\pi_2(G_{vw})\) with \(|\pi_{2}(G_{uv})|_{r}>1\) and \(|\pi_{2}(G_{vw})|_{r}>1\), and we have a contradiction.

\medskip\noindent
\emph{Claim $3$,\quad  \((m,2k, \epsilon)\) are as in one of the lines of Table~\ref{t:c4}.}

 Suppose that Claim 3 is false. Recall that, by Hypothesis~\ref{C4},  $m\geqslant3$, $q$ is odd, \((2k, m,\epsilon)\neq (2,4,+)\), and \((m,q)\ne(3,3)\). It therefore follows from Claim 2 that 
 \[
 m=\begin{cases}
 2k+1\geqslant7&\text{if}\quad \epsilon=\circ\\
 2k\geqslant6&\text{if}\quad\epsilon=-\\
 2k+2\geqslant6&\text{if}\quad\epsilon=+.\end{cases}
 \]
 Also, since \(q\) is odd,  \(K_{i}\) has no section isomorphic to \(\Soc(K_{3-i})\) for \(i=1,2\). Hence  by Lemma \ref{corefree}(i),  for each \(i=1,2\), the factorisation \(\pi_{i}(G_{v})=\pi_{i}(G_{uv})\pi_{i}(G_{vw})\) is core-free. Let us denote by \(H_{i}:=\pi_{i}(G_{uv})\), by \(M_{i}:=\pi_{i}(G_{vw})\) for \(i=1,2\), and by \(X_{i}:=\ICF(H_{i})\cup\ICF(M_{i})\) for \(i=1,2\). 
 First consider the case where \((m, 2k, \epsilon)=(6,4, +)\). By Corollary~\ref{existppd}, \(\ppd(p, 3f)\neq\varnothing\) and we choose a prime \(s\in\ppd(p,3f)\). Then \(|\Soc(K_{2})|_{s}>1\) and  \(|K_{1}|_{s}=1\), and also \(s\in\Pi(G_{uv})=\Pi(G_{vw})\) since \(\Pi(G_{v})=\Pi(G_{uv})=\Pi(G_{vw})\).
 This implies that \(s\in\Pi(H_{2})\cap\Pi(M_{2})\). However by Lemma \ref{odd}(ii) there is no such core-free factorisation of an almost simple group with socle \(\POmega_{6}^{+}(q)\), which is a contradiction.
 
 Thus \(2k\geqslant6\). By rows 1-3 of Table \ref{PSp(2m,q)} in Lemma \ref{sp} we have \(\ICF(H_{1})\cap\ICF(M_{1})=\varnothing\) and one of the following occurs:
\begin{description} 
\item[(i)] \(X_{1}=\{\PSp_{2a}(q^{b}), \PSp_{2k-2}(q)\}\) with \(2k=2ab\) and $b>1$;
\item[(ii)] \((2k,q)=(6,3)\), \(X_{1}=\{\PSL_{2}(13), \PSp_{4}(3)\}\);
\item[(iii)] \((2k,q)=(6,3)\), \(\{\PSL_{2}(27)\}\subseteq X_{1}\subseteq \{\PSL_{2}(27), \A_{5}\}\).
\end{description}
Since \(\ICF(H_{1})\cap\ICF(M_{1})=\varnothing\), this together with the fact that \begin{equation}\label{eqnHM1}
\ICF(H_{1})\subseteq \ICF(G_{uv})=\ICF(G_{vw})\subseteq \ICF(M_{1})\cup\ICF(M_{2})
\end{equation} 
implies that \(\ICF(H_{1})\subseteq\ICF(M_{2})\). Similarly, 
\begin{equation*}%\label{eqnHM1}
\ICF(M_{1})\subseteq \ICF(G_{vw})=\ICF(G_{uv})\subseteq \ICF(H_{1})\cup\ICF(H_{2})
\end{equation*} 
and hence \(\ICF(M_{1})\subseteq\ICF(H_{2})\). Therefore
\begin{equation}\label{eqnHM2}
X_1=\ICF(H_{1})\cup\ICF(M_{1})\subseteq\ICF(M_{2})\cup\ICF(H_{2})=X_2.
\end{equation}

If $(m,2k,\epsilon)=(8,6,+)$ then by Lemma~\ref{8 +}, some row of Table~\ref{P(8,+)} should contain all the simple groups in $X_1$, and this fails for each case (i)--(iii). Thus if $\epsilon=+$ then $m=2k+2\geq 10$, so in all cases $m\geq 2k\geq 6$ and  $(m, \epsilon)\ne (8,+)$. Therefore Proposition \ref{odd} applies with $d=2k\geq 6$.

Suppose first that $d=2k=6$, Then $ \Soc(K_2)=\POmega_7(q)$ or $\POmega_6^-(q)$, and by Proposition \ref{odd}(iii), one 
of the Lines of Table~\ref{O(2m+1,q)} must contain all the groups in $X_1$. The only Line for which this happens is Line 1 with $q=3$, $\Soc(K_2)=\POmega_7(q)$ and $X_1=\{ \PSL_2(q^3)\}$ (as in case (iii) above). However for this Line  $X_2=\{ \POmega_6^-(3), \PSL_2(27)\}$, and it follows from  Proposition~ \ref{odd} (iii) that \(\ICF(H_{2})\cap\ICF(M_{2})=\varnothing\). Then applying the argument in \eqref{eqnHM1} and \eqref{eqnHM2} to \(H_{2}\) and \(M_{2}\) in place of \(H_{1}\) and \(M_{1}\), we conclude that \(X_{2}\subseteq X_{1}\), which is a contradiction. Thus $d=2k\geq 8$, and hence $X_1$ is as in case (i). Then by  Proposition \ref{odd}(iii), one 
of the Lines of Table~\ref{O(2m+1,q)} must contain both $\PSp_{2a}(q^b)$ and $\PSp_{2k-2}(q)$, for some prime $b$ such that $k=ab$. There are no Lines with this property, so we have a contradiction.  This completes the proof of Claim 3, and hence of Lemma~\ref{l:c4-1}.

%{\color{red} Lei: I rewrote the above. Note that for $O^+$ we have $m=2k+2$ so the case  \((m,q,\epsilon)=(6,3,+)\)  you wrote about in your proof is not one of interest as it was dealt with in the first paragraph of the proof of Claim 3.  Could you please check that the above is OK? I liked your proof for $O_7(3)$ and used it.}

% {\color{blue} By checking Table~\ref{O(2m+1,q)}, it follows that case (iii) occurs and \((2k,q)=(6,3)\). In particular, either \((m,q)=(7,3)\) and \(X_{2}=\{\POmega_{6}^{-}(3),\PSL_{2}(27)\}\) (see line 1), or \((m,q,\epsilon)=(6,3,+)\) and \(X_{2}=\{\Omega_{5}(3),\PSL_{2}(27)\}\) (see line 23). By Lemma \ref{odd} (iii), \(\ICF(H_{2})\cap\ICF(M_{2})=\varnothing\). Then applying the argument in \eqref{eqnHM1} and \eqref{eqnHM2} on \(H_{2}\) and \(M_{2}\) in place of \(H_{1}\) and \(M_{1}\), we deduce that \(X_{2}\subseteq X_{1}\), which is a contradiction}.

\end{proof}

We now deal with the lines of Table~$\ref{t:c4}$ separately. First we consider Line $3$, the only line for plus type groups.

\begin{lemma}\label{l:c4-4}
Suppose that Hypothesis \ref{C4} holds and \((m,\epsilon)=(8,+)\) and \(2k\leqslant4\). Then \(\Gamma\) is not \((L,2)\)-arc-transitive.
\end{lemma}
\begin{proof}
Suppose that  \(\Gamma\) is \((L,2)\)-arc-transitive, and note that \(K_{1}\) has no section isomorphic to \(\POmega_{8}^{+}(q)\). Thus by Lemma \ref{corefree}, \(\pi_{2}(G_{v})=\pi_{2}(G_{uv})\pi_{2}(G_{vw})\) is core-free. Let us denote by \(H_{i}:=\pi_{i}(G_{uv})\) and \(M_{i}:=\pi_{i}(G_{vw})\) for \(i=1,2\). Then by Lemma \ref{8 +}, interchanging \(H_{2}\) and \(M_{2}\) if necessary, there exists \(T\in\ICF(H_{2})\) such that \(|T|_{r}>1\) for any \(r\in\ppd(p,6f)\).

On the other hand, since \(T\in\ICF(H_{2})\) we have \(T\in\ICF(G_{uv})=\ICF(G_{vw})\). Let \(r\in\ppd(p,6f)\).  Note that \(|K_{1}|_{r}=1\), and so \(T\notin\ICF(M_{1})\). This together with the fact that \(\ICF(G_{vw})\subseteq\bigcup_{i=1}^{2}\ICF(M_{i})\) implies that  \(T\in\ICF(M_{2})\). Thus \(\ICF(M_{2})\cap\ICF(H_{2})\neq \varnothing\). By Lemma \ref{8 +} we have \(\ICF(H_{2})=\ICF(M_{2})=\{\Omega_{7}(q)\}\). By checking \cite[Table 4]{LPS} we conclude that \(H_{2}\cap\Soc(K_{2})=M_{2}\cap\Soc(K_{2})=\Omega_{7}(q)\). For \(r\in\ppd(q,4f)\) let \(r^{c}:=(q^{4}-1)_{c}\). Then \(r>4f\) and \(c\geqslant1\) and \(|\POmega_{8}^{+}(q)|_{r}=r^{2c}\). Moreover, \(|H_{2}\cap\Soc(K_{2})|_{r}=|\Omega_{7}(q)|_{r}=r^{c}\) and \(|\Out(\Soc(K_{2}))|_{r}=(24f)_{r}=1\). This implies that \(|H_{2}|_{r}=r^{c}\) and similarly \(|M_{2}|_{r}=r^{c}\). Assume that \(\Gamma\) is \((G,3)\)-arc-transitive. Then let \(\gamma\) be the valency of \(\Gamma\) we deduce from Lemma \ref{valency} that \(|G_{v}|_{r}\geqslant \gamma_{r}^{3}\). Note that \(\gamma=|G_{v}|/|G_{uv}|\) and \(|G_{uv}|_{r}=|H_{1}|_{r}r^{c}\). Thus the valency \(\gamma\) of \(\Gamma\) has \(r\)-part
\[
    \frac{|G_{v}|_{r}}{|G_{uv}|_{r}}=\frac{|K_{1}|r^{c}}{|H_{1}}.
\]
Since \(|G_{v}|_{r}=|K_{1}|_{r}r^{2c}>\gamma_{r}^{3}\) we have 
\(
   |K_{1}|_{r}r^{2c}>(|K_{1}|_{r}r^{c}/|H_{1}|)^{3}
\)
and it gives 
\(|H_{1}|_{r}> r^{c/3}|K_{1}|_{r}^{2/3}>1\).
With the similar argument on \(G_{vw}\), we obtain that \(|M_{1}|_{r}>1\). Note that \(|H_{1}|_{r}=|M_{1}|_{r}=1\) when \(2k=2\). Thus it suffices to consider the case when \(2k=4\). By Lemma \ref{sp} there is no core-free factorisation for an almost simple group with socle \(\PSp_{4}(q)\) with \(q\) odd and both factors having orders divisible by \(r\). Hence at least one of \(H_{1}\) or \(M_{1}\) contains \(\Soc(K_{1})\). Without loss of generality, assume that \(\Soc(K_{1})\leqslant H_{1}\). This implies that \(\Soc(K_{1})\cong \PSp_{2k}(q)\in \ICF(G_{uv})\). Since \(G_{uv}/(G_{uv}\cap K_{1})=H_{2}\) and \(\ICF(H_{2})=\{\Omega_{7}(q)\}\) we have \(\PSp_{2k}(q)\in\ICF(G_{uv}\cap K_{1})\). Thus \(\Soc(K_{1})\trianglelefteq G_{uv}\). Moreover, \(\PSp_{4}(q)\cong\Soc(K_{1})^{g}\trianglelefteq G_{vw}\). Since \(G_{vw}/(G_{vw}\cap K_{1})\cong M_{2}\) and \(\ICF(M_{2})=\{\Omega_{7}(q)\}\) we have \(\Soc(K_{1})^{g}\leqslant K_{1}\). In particular, \(\Soc(K_{1})^{g}=\Soc(K_{1})\), contradicting Lemma \ref{normal} and proving the result.
\end{proof}

Next we deal with Line 4 of Table~$\ref{t:c4}$, the only line for minus type groups, but this time we only disprove 
\((G,3)\)-arc-transitivity.

\begin{lemma}\label{l:c4-5}
Suppose that Hypothesis \ref{C4} holds and \((m, 2k,\epsilon)=(4, 4, -)\). Then \(\Gamma\) is not \((G,3)\)-arc-transitive.
\end{lemma}

\begin{proof}
Suppose that  \(\Gamma\) is \((G,3)\)-arc-transitive. Here \(\Soc(K_{2})=\POmega_4^-(q) \cong\PSL_{2}(q^{2})\) and \(K_{2}\) has no section isomorphic to \(\Soc(K_{1})\cong\PSp_{4}(q)\). If follows from Lemma \ref{corefree} that \(\pi_{1}(L_{v})=\pi_{1}(L_{uv})\pi_{1}(L_{vw})\) is a core-free factorisation. On the other hand \(\ppd(p,4f)\neq\varnothing\) since \(p\) is odd. Let \(r\in\ppd(p,4f)\) and let \(r^{a}:=(q^{2}+1)_{r}\). Then \(|\PSp_{4}(q)|_{r}=|\PSL_{2}(q^{2})|_{r}=r^{a}\) and \(|L_{v}|_{r}=r^{2a}\). On the other hand, since \(|\Out(L)|_{r}=(2f)_{r}=1\) we have \(|G_{v}|_{r}\leqslant|L_{v}|_{r}\cdot |\Out(L)|_{r}=r^{2a}\).

\medskip
\noindent
\emph{Claim.}\label{44}\quad 
\(|\pi_{1}(L_{uv})|_{r}>1\) and  \(|\pi_{1}(L_{vw})|_{r}>1\).

\medskip
Suppose for a contradiction that \(|\pi_{1}( L_{uv})|_{r}=1\). Then 
\(
    |L_{uv}|_{r}\leqslant|\pi_{1}( L_{uv})|_{r}|\cdot \pi_{2}(L_{uv})|_{r}=|\pi_{2}(L_{uv})|_{r}\leqslant(q^{2}+1)_{r}=r^{a}
\).
Thus the valency of \(\Gamma\) has \(r\)-part
\[
    \frac{|L_{v}|_{r}}{|L_{uv}|_{r}}\geqslant\frac{r^{2a}}{r^{a}}=r^{a}.
\]
Since \(\Gamma\) is \((G,3)\)-arc-transitive it follows from Lemma \ref{valency} that \(|G_{v}|\) is divisible by \(r^{3a}\), contradicting the fact that \(|G_{v}|_{r}\leqslant r^{2a}\). A similar argument shows that \(|\pi_{1}(L_{vw})|_{r}>1\). Thus the claim is proved.

\medskip
 However, by Lemma \ref{sp} there is no core-free factorisation of an almost simple group with socle \(\PSp_{4}(q)\) with \(q\) odd such that both factors have orders divisible by \(r\). This contradicts the Claim, and thereby proves the lemma.
\end{proof}

It remains to consider lines 1-2 of Table~\ref{t:c4}, the lines for orthogonal groups of odd dimension.

\begin{lemma}\label{l:c4-6}
Suppose that Hypothesis \ref{C4} holds and \((m, 2k,\epsilon)=(5, 4, \circ)\) or \( (3,2,\circ)\). Then \(\Gamma\) is not \((G,3)\)-arc-transitive.
\end{lemma}

\begin{proof}
Suppose that  \(\Gamma\) is \((G,3)\)-arc-transitive, so  by Lemma \ref{s-1}, $\Gamma$ is \((L,2)\)-arc-transitive. In particular, \(L_{v}=L_{uv}L_{vw}\) with \(L_{uv}^{g}=L_{vw}\). Let us denote by \(H_{i}:=\pi_{i}(G_{uv})\) and \(M_{i}:=\pi_{i}(G_{vw})\) for \(i=1,2\). 

\medskip\noindent
\emph{Case: \((m, 2k,\epsilon)=(5, 4, \circ)\).} \quad Here the argument is similar in spirit to that of the previous lemma. We have  \(\PGO_{5}(q)\cong\PSp_{4}(q).2\). Moreover, \(\ppd(p,4f)\neq\varnothing\) since \(p\) is odd. Let \(r\in\ppd(p, 4f)\), let \(r^{a}:=(q^{4f}-1)_{r}\), and note that $r> 4f > |\Out(L)|$ so $|\Out(L)|_r=1$. Then \(|\PGO_{5}(q)|_{r}=|\PSp_{4}(q)|_{r}=r^{a}\) and \(|L_{v}|_{r}=r^{2a}\).

\medskip
\noindent
\emph{Claim 1.}\quad 
For each $i\in\{1,2\}$, \(|K_{i}\cap L_{uv}|_{r}\) and \(|K_{i}\cap L_{vw}|_{r}>1\).

\medskip
 Suppose for a contradiction that, for some $i$, we have \(|K_{i}\cap L_{uv}|_{r}=1\). Then 
\(|L_{uv}|_{r}=|K_{i}\cap L_{uv}|_{r}\cdot |\pi_{3-i}(L_{uv})|_{r}=|\pi_{3-i}(L_{uv})|_{r}\leqslant r^{a}
\).
This implies that the valency of \(\Gamma\) has \(r\)-part
\[
    \frac{|L_{v}|_{r}}{|L_{uv}|_{r}}\geqslant\frac{r^{2a}}{r^{a}}=r^{a}.
\]
Since \(\Gamma\) is \((G,3)\)-arc-transitive, it follows from Lemma \ref{valency} that \(|G_{v}|\) is divisible by \(r^{3a}\). On the other hand, \(|G_{v}|_{r}\leqslant|L_{v}|_{r}\cdot |\Out(L)|_{r}=|L_{v}|_{r}=r^{2a}\) and
we have a contradiction. Thus \(|K_{i}\cap L_{uv}|_{r}>1\), and the same argument shows that  \(|K_{i}\cap L_{vw}|_{r}>1\). Since the same argument works for $i=1$ and $i=2$, Claim 1 is proved. 

\medskip
 Since \(K_{i}\cap L_{uv}\trianglelefteq\pi_{i}(L_{uv})\) and \(K_{i}\cap L_{vw}\trianglelefteq\pi_{i}(L_{vw})\) for \(i=1\) and \(2\), it follows from Claim 1 that \(|\pi_{i}(L_{uv})|_{r}>1\) and \(|\pi_{i}(L_{vw})|_{r}>1\).  Then  by Lemma \ref{sp}(i). the factorisation \(\pi_{i}(L_{v})=\pi_{i}(L_{uv})\pi_{i}(L_{vw})\) is not core-free for $i=1, 2$. 
  Thus, at least one of \(\pi_{2}(L_{uv})\) or \(\pi_{2}(L_{vw})\) contains \(\Soc(K_{2})\) we, and without loss of generality we assume that \(\Soc(K_{2})\triangleleft\pi_{2}(L_{uv})\), so \(\Soc(\pi_{2}(L_{uv}))=\Soc(K_{2})\). Now $r$ does not divide $|\Out(K_2)|$, and hence, by Claim 1,  \(L_{uv}\cap \Soc(K_{2})\) is nontrivial. Thus  \(L_{uv}\cap \Soc(K_{2})\) is a nontrivial normal subgroup of $\pi_2(L_{uv})\cap \Soc(K_2)=\Soc(K_2)$, and since  \(\Soc(K_{2})\) is simple it follows that $L_{uv}$ contains $\Soc(K_2)$ as a normal subgroup.
  
  Consider $L_{vw}=L_{uv}^g$ with $g$ as in Hypothesis~\ref{C4}. The nonabelian simple subgroup $Y:=\Soc(K_2)^g$ is contained in $L_{vw}<G_v$, and as $K_2\unlhd G_v$, the intersection $Y\cap K_2\unlhd Y$ and hence $Y\cap K_2$ is either $Y$ or $1$. Suppose the former holds. Then $Y\leq K_2$ and the only subgroup of $K_2$ isomorphic to $Y$ is $\Soc(K_2)$. It follows that $\Soc(K_2)=Y=\Soc(K_2)^g$, that is, $g$ normalises the normal subgroup $\Soc(K_2)$ of  $G_v$, contradicting Lemma~ \ref{normal}. Thus $Y\cap K_2=1$ and hence $Y\cong \pi_1(Y)\leqslant \pi_1(G_v)=K_1$. As the only subgroup of $K_1$ isomorphic to $Y$ is $\Soc(K_1)=K_1$, we have $\pi_1(Y)=K_1$.   Further, since $r$ does not divide $|\Out(K_1)|$, it follows from Claim 1 that $L_{vw}\cap K_1\ne 1$, so $L_{vw}\cap K_1$ is a nontrivial normal subgroup of the simple group $\pi_1(L_{vw})=K_1$, and hence $K_1\leqslant L_{vw}=L_{uv}^g$. If $\PSp_4(q)$ has multiplicity $1$ as a composition factor of $L_{vw}$, then \(K_{1}\) is the unique normal subgroup of \(L_{vw}\) isomorphic to \(\PSp_{4}(q)\) and therefore \(\Soc(K_{2})^{g}=K_{1}\). However, this is a contradiction to Lemma \ref{tensor}(iii). Therefore $L_{vw}$ has at least two composition factors isomorphic to \(\PSp_{4}(q)\), and this implies that $L_{vw}$ contains $Y_1:= K_1\times \Soc(K_2)$ (a subgroup of index $2$ in $L_v$, and equal to $\Soc(G_v)$). As $L_{uv}\cong L_{vw}$ it follows that $Y_1\leq L_{uv}$ also, and hence that $g$ normalises $Y_1$, again contradicting Lemma~\ref{normal}.  Therefore we conclude that \((m, 2k,\epsilon)\ne (5, 4, \circ)\),

\medskip\noindent
\emph{Case:  \((m, 2k,\epsilon)=(3, 2, \circ)\).}\quad Here $n=mk=3$ and so by Proposition~\ref{citesmall}(ii), $q\geq7$. Recall that $p$ is odd. If \(q\) is not a Mersenne prime then \(\ppd(p,2f)\neq \varnothing\); we choose \(d\in\ppd(p,2f)\) and set \(d^{c}:=(p^{2f}-1)_{d}\). 

\medskip\noindent
\emph{Claim~2. \quad For some $i$, one of \(L_{uv}\cap K_{i}\) or \(L_{vw}\cap K_{i}\) does not contain \(\Soc(K_{i})\).}

\medskip
Suppose to the contrary that \(\Soc(K_{i})\) is contained in \(L_{uv}\cap K_{i}\) and \(L_{vw}\cap K_{i}\) for \(i=1\) and \(2\). In particular, \(\Soc(L_{v})=\Soc(K_{1})\times\Soc(K_{2})\trianglelefteq L_{uv}\) and \(\Soc(L_{v})^{g}\trianglelefteq L_{vw}\). Since \(\Soc(L_{v})\) is the unique subgroup of \(L_{v}\) isomorphic to \(\PSL_{2}(q)^{2}\), it follows that \(\Soc(L_{v})^{g}=\Soc(L_{v})\), contradicting Lemma \ref{normal}. Thus Claim 2 is proved

\medskip\noindent
\emph{Claim~3. \quad Let \(i\in\{1,2\}\). Then \(|K_{i}\cap L_{uv}|_{p}>1\) and \(|K_{i}\cap L_{vw}|_{p}>1\), and moreover if \(q\) is not a Mersenne prime, then also \(|K_{i}\cap L_{uv}|_{d}>1\) and \( |K_{i}\cap L_{vw}|_{d}>1\).}
  
\medskip
Suppose to the contrary that \(|K_{i}\cap L_{uv}|_{p}=1\) for some $i$. Then 
\[
    |L_{uv}|_{p}=|K_{i}\cap L_{uv}|_{p}|\cdot\pi_{3-i}(L_{uv})|_{p}=|\pi_{3-i}(L_{uv})|_{p}\leqslant|\PSL_{2}(q)|_{p}=p^{f}.
\]
On the other hand, \(|L_{v}|_{p}=|\PSL_{2}(q)|_{p}^{2}=p^{2f}\) and this implies that the valency of \(\Gamma\) has \(p\)-part
\[
    \frac{|L_{v}|_{p}}{|L_{uv}|_{p}}\geqslant\frac{p^{2f}}{p^{f}}=p^{f}.
\]
Since \(\Gamma\) is \((G,3)\)-arc-transitive, it follows from Lemma \ref{valency} that \(|G_{v}|_{p}\geqslant p^{3f}\). On the other hand \(|G_{v}|_{p}\leqslant|L_{v}|_{p}|\Out(L)|_{p}\), so
\[
    (2f)_{p}=|\Out(L)|_{p}\geqslant\frac{|G_{v}|_{p}}{|L_{v}|_{p}}\geqslant\frac{p^{3f}}{p^{2f}}=p^{f},
\]
which is impossible since $(2f)_{p}=(f)_{p}\leqslant (f!)_{p}<p^{f/(p-1)}\leqslant p^{f/2}$. Hence \(|K_{i}\cap L_{uv}|_{p}>1\) for each $i$, and the same argument shows that \(|L_{vw}\cap K_{i}|_{p}>1\) for each $i$, proving the first part of Claim~3. Now assume that \(q\) is not a Mersenne prime and $d$ is as above.  Suppose that \(|L_{uv}\cap K_{i}|_{d}=1\) for some $i$.  Then 
\[
    |L_{uv}|_{d}=|L_{uv}\cap K_{i}|_{d}|\cdot\pi_{3-i}(L_{uv})|_{d}=|\pi_{3-i}(L_{uv})|_{d}\leqslant|\PSL_{2}(q)|_{d}=d^{c}.
\]
Note that \(|L_{v}|_{d}=|\PSL_{2}(q)|^{2}_{d}=d^{2c}\) and this implies that the valency of \(\Gamma\) has \(d\)-part
\[
    \frac{|L_{v}|_{d}}{|L_{uv}|_{d}}\geqslant\frac{d^{2c}}{d^{c}}=d^{c}.
\]
Since \(\Gamma\) is \((G,3)\)-arc-transitive, it follows from Lemma \ref{valency} that \(|G_{v}|_{d}\geqslant d^{3c}\) and hence
\[
d^{3c}\leqslant|G_{v}|_{d}\leqslant|L_{v}|_{d}|\Out(L)|_{d}=d^{2c}(2f)_{d}=d^{2c},
\]
which is a contradiction.  Thus  \(|L_{uv}\cap K_{i}|_{d}>1\) for each $i$, and the same argument shows that \(|L_{vw}\cap K_{i}|_{d}>1\) for each $i$, proving Claim~3.

\medskip
Suppose first that $q$ is not a Mersenne prime and $q\ne 9$. Then by Claim 3 and \cite[Table 10.3]{transitive}, for each $i$, \(\Soc(K_{i})\) is contained in \(K_i\cap L_{uv}\) and \(K_i\cap L_{vw}\), contradicting Claim 2. 

Suppose next that $q=p=2^t-1$ is a Mersenne prime (where $t\geqslant3$ as $q\geqslant7$). Then  by Claim 3 and \cite[Table 10.3]{transitive} we deduce that, for each $i$, either \(K_i\cap L_{uv}\leqslant \C_{p}:\C_{\frac{p-1}{2}}\) or $K_i\cap L_{uv}$ contains $\Soc(K_i)$, and  similarl; either \(K_i\cap L_{vw}\leqslant \C_{p}:\C_{\frac{p-1}{2}}\) or $K_i\cap L_{vw}$ contains $\Soc(K_i)$.  By Claim 2, at least one of these groups is soluble, without loss of generality  \(K_i\cap L_{uv}\leqslant \C_{p}:\C_{\frac{p-1}{2}}\), and we note that this subgroup has odd order. Thus  
\[
    |L_{uv}|_{2}=|K_i\cap L_{uv}|_{2}\cdot |\pi_{3-i}(L_{uv})|_{2}=|\pi_{3-i}(L_{uv})|_{2}\leqslant|\PGO_{3}(q)|_{2}=2^{t+1}.
\]
On the other hand, \(|L_{v}|_{2}=2|\PSL_{2}(q)|^{2}=2^{2t+1}\) and this implies that the valency of \(\Gamma\) has \(2\)-part:
\[
    \frac{|L_{v}|_{2}}{|L_{uv}|_{2}}\geqslant\frac{2^{2t+1}}{2^{t+1}}=2^{t}.
\]
Since \(\Gamma\) is \((G,3)\)-arc-transitive, it follows from Lemma \ref{valency} that \(|G_{v}|_{2}\geqslant 2^{3t}\). However, this would imply that
\[
2^{3t}\leqslant|G_{v}|_{2}\leqslant|L_{v}|\cdot |\Out(L)|_{2}=2^{2t+1}(2)_{2}=2^{2t+2},
\]
which is a contradiction  since \(t\geqslant3\). 

This leaves the exceptional case where $q=9$ and here by Claim 3,  for each $i$, \(\Soc(K_i)\cap L_{uv}\) is $\A_6$ or $\A_5$, and  \(\Soc(K_i)\cap L_{vw}\)  is $\A_6$ or $\A_5$. Also, by Claim 2, at least one of these groups is $\A_5$ and we note that $L_{uv}\cong L_{vw}$. Without loss of generality we assume that \(\Soc(K_i)\cap L_{uv}\)  is $\A_5$, and we claim that \(\pi_{j}(L_{uv})\geqslant \A_{6}\), where $\{i,j\}=\{1,2\}$. Suppose that \(\pi_{j}(L_{uv})\not\geqslant \A_{6}\). Then \(|\pi_{j}(L_{uv})|_{3}\leqslant3\) and so
\[
|L_{uv}|_{3}=|L_{uv}\cap K_{i}|_{3}\cdot |\pi_{j}(L_{uv})|_{3}\leqslant3^{2}.
\]
This implies that the valency of \(\Gamma\) has \(3\)-part
\[
\frac{|L_{v}|_{3}}{|L_{uv}|_{3}}=\frac{(6!)_{3}^{2}}{3^{2}}=3^{2}.
\]
Since \(G\) acts \(3\)-arc-transitively on \(\Gamma\), it follows that \(|G_{v}|_{3}\geqslant (3^{2})^{3}=3^{6}\). 
On the other hand, \(|G_{v}|_{3}\leqslant |L_{v}|_{3}\cdot |\Out(L)|_{3}=3^{4}\cdot |\Out(L)|_{3}=3^{4}\) as \(|\Out(\PSp_{6}(9))|_{3}=1\), which is a contradiction. 
Hence \(\pi_{j}(L_{uv})\) contains \(\A_{6}\). Since \(\Soc(K_{j})\cap L_{uv}\) is a non-trivial normal subgroup of \(\pi_{j}(L_{uv})\) it follows that \(\Soc(K_{j})\cap L_{uv}=\A_{6}=\Soc(K_j)\). 
Thus $\Soc(L_{uv})\cong \A_6\times \A_5$  and contains $\Soc(K_{j})$. Since $L_{uv}\cong L_{vw}$, also $\Soc(L_{vw})\cong \A_6\times \A_5$. If   $\Soc(K_{j})$ is contained in $L_{vw}$, then it is the unique normal subgroup of both $L_{uv}$ and $L_{vw}$ isomorphic to $\A_6$ and hence \(\Soc(K_{j})^{\overline{g}}=\Soc(K_{j})\), contradicting Lemma \ref{normal}. Thus $\Soc(K_{j})\cap L_{vw} = \A_5$ and $\Soc(K_{i})$ is the unique normal subgroup of $L_{vw}$ isomorphic to $\A_6$,  and so \(\Soc(K_{j})^{\overline{g}}=\Soc(K_{i})\). However, this is a contradiction to Lemma \ref{divisibility} (iii). Thus the result is proved.
\end{proof}

We now draw together these lemmas to prove Proposition~\ref{p:c4-5}.

\medskip\noindent
\emph{Proof of Proposition~\ref{p:c4-5}.}\quad Suppose that Hypothesis~\ref{C4} holds and that $\G$ is $(L,2)$-arc-transitive. Then by Lemma~\ref{l:c4-1}, $m,k,\epsilon$ satisfy one of the lines of Table~\ref{t:c4}. By Lemma~\ref{l:c4-4}, Line 3 of Table~\ref{t:c4} does not hold. Also by Lemma~\ref{l:c4-5} and~\ref{l:c4-6}, for the other three lines of Table~\ref{t:c4}, $\G$ is not $(G,3)$-arc-transitive, but it is not proved that it cannot be $(L,2)$-arc-transitive. This completes the proof.

\subsection{\(\mathcal{C}_{5}\)-subgroups}

Here \(G_{v}\) is a maximal \(\mathcal{C}_{5}\)-subgroup of the almost simple group $G$ where \(\Soc(G)=\PSp_{2n}(q)\) with $q=p^f$ for a prime $p$, so by \cite[Proposition 4.5.4]{kleidman},  \(L_{v}=G_{v}\cap L\cong \PSp_{2n}(q^{1/r}).c\), for some prime $r$ dividing $f$, where \(c:=(2,q-1,r)\). 

\begin{proposition}\label{p:c5}
Suppose that Hypothesis \ref{general} holds and that \(G_{v}\) is a maximal \(\mathcal{C}_{5}\)-subgroup of \(G\) as above. Then $\Gamma$ is not  $(G,3)$-arc-transitive. Moreover $\Gamma$ is not $(L,2)$-arc-transitive except possibly for  \((n,p)=(2,2)\) for some $f\geqslant 6$.

\end{proposition}

\begin{proof}
Assume that $\G$ is $(L,2)$-arc-transitive, and note that this property holds if $\G$ is $(G,3)$-arc-transitive by Lemma~\ref{s-1}. Then \(L_{v}=L_{uv}L_{vw}\) with \(L_{uv}\cong L_{vw}\).

Let $\R(G_v)$ be the soluble radical of $G_v$. Then \(G_{v}/\R(G_{v})\) is an almost simple group with socle \(\PSp_{2n}(q^{1/r})'\). For a subgroup $H\leqslant G_v$ we write \(\overline{H}:=H\R(G_{v})/\R(G_{v})\). Then by Lemma~\ref{normal}, \(G_{v}/\R(G_{v})\)  cannot be contained in both $\overline{G_{uv}}$ and  $\overline{G_{vw}}$. 

\medskip\noindent
\emph{Case: $p=2$.} \quad Here $R(L_v)=1$ and $L_v$ is an almost simple group with a homogeneous factorisation $L_v=L_{uv}L_{vw}$. It follows from \cite[Proposition 3.3 and Table 1]{linear} that $n=2$ and $\Soc(L_{uv})\cong \Soc(L_{vw})\cong \Sp_2(4^{f/r})'$, with $f/r\geqslant 1$. By Lemma~\ref{citesmall}, \(q=2^{f}\geqslant 41\)  and so  $f\geqslant 6$. 
We have \(|L_{v}|_{2}=2^{4f/r}\) and \(|L_{uv}|_{2}\leqslant2|\PSp_{2}(4^{f/r})|_{2}=2^{1+2f/r}\), and so the valency of \(\Gamma\) has \(2\)-part:
\[
    \frac{|L_{v}|_{2}}{|L_{uv}|_{2}}\geqslant\frac{2^{4f/r}}{2^{1+2f/r}}=2^{\frac{2f}{r}-1}.
\]

 Suppose for a contradiction that \(\Gamma\) is \((G,3)\)-arc-transitive.  Then it follows from Lemma \ref{valency} that \(|G_{v}|_{2}\geqslant 2^{\frac{6f}{r}-3}\). On the other hand, 
\[
    |G_{v}|_{2}\leqslant|L_v|_2|\Out(L)|_2=2^{4f/r}(2f)_{2}\leqslant 2^{\frac{4f}{r}+1}(f)_{2}.
\]
Now \(2^{\frac{4f}{r}+1}(f)_{2}\geqslant 2^{\frac{6f}{r}-3}\) and so  \((f)_{2}\geqslant2^{\frac{2f}{r}-4}\). If \(r=2\), then \((f)_{2}\geqslant 2^{f-4}\), but this is a contradiction to Corollary \ref{n-3} since \(f\geqslant6\). Therefore \(r\) is odd and \((\frac{f}{r})_{2}=(f)_{2}\geqslant 2^{\frac{2f}{r}-4}\); and this implies that $f/r=1$ or $2$. In these cases $|G_v|_3=3^2$, while $|G_{uv}|_3=3$, and the same argument leads to a contradiction as follows: the $3$-part of the valency is $|G_v|_3/|G_{uv}|_3 = 3$ so that $|G_v|_3\geq 3^3$ by Lemma~\ref{valency}.   Thus Proposition~\ref{p:c5} is proved for $p=2$.

\medskip\noindent
\emph{Case: $p$ odd.} \quad Here we note that the hypotheses of \cite[Lemma 3.3]{small} hold for $G_v$ with $M=\PSp_{2n}(q^{1/r})$ and $M$ is not a composition factor of both $G_{uv}$ and $G_{vw}$. We deduce from \cite[Lemmas 2.6, 2.7 and 3.3]{small} that \(\overline{G_{v}}=\overline{G_{uv}}\,\overline{G_{vw}}\) is a core-free factorisation and that \(M=\PSp_4(3)\cong \PSU_4(2)\) with $p^{f/r}=3$, so $f=r$, and $L_v\cong \PSp_4(3).c$ with $c=(2,r)$.  Thus \(L_{v}\) is  almost simple with socle \(\PSp_{4}(3)\), and it follows from \cite[Proposition 3.3]{linear} that \(L_{v}\) does not admit a homogeneous factorisation, contradicting the fact that \(L_{v}=L_{uv}L_{vw}\) with \(L_{uv}\cong L_{vw}\). This completes the proof of Proposition \ref{p:c5}. %Since $|M|=2^6,3^4.5$ and \(\overline{G_{v}}=\overline{G_{uv}}\,\overline{G_{vw}}\) is a core-free factorisation it follows that $|\overline{G_{uv}}|$ is divisible by $2^3.3^2.5$ and so $|\overline{L_{uv}}|$ has index in $|\overline{L_{v}}|$ dividing $2^4.3^2$. By \cite[p. 27]{Atlas} we see that  $\Soc(\overline{L_{uv}})=A_6$.  
%Suppose now that \(\Gamma\) is \((G,3)\)-arc-transitive, so that by Lemma \ref{s-1}, \(\Gamma\) is \((L,2)\)-arc-transitive. The usual argument leads to a contradiction:  the $3$-part of the valency is $|G_v|_3/|G_{uv}|_3 = 3^2$ so that $|G_v|_3\geq 3^6$ by Lemma~\ref{valency}.
\end{proof}

\subsection{\(\mathcal{C}_{6}\)-subgroups}
Here \(G_{v}\) is a maximal \(\mathcal{C}_{6}\)-subgroup  of the almost simple group $G$ where \(\Soc(G)=\PSp_{2n}(q)\) with $q=p^f$ for a prime $p$, so by \cite[Proposition 4.6.9]{kleidman},  \(n=2^{m-1}\) and \(q=p\) is odd with \(L_{v}=\C_{2}^{2m}.K\), where \(K=\Omega_{2m}^{-}(2).c\) and $c=1$ if \(q\equiv \pm 3\pmod{8}\) and  $c=2$ if  \(q\equiv\pm1\pmod{8}\). %\(K=\SO_{2m}^{-}(2)\) if \(q\equiv\pm1\pmod{8}\).

\begin{proposition}\label{p:c6}
Suppose that Hypothesis \ref{general} holds. Then \(G_{v}\) is not a maximal \(\mathcal{C}_{6}\)-subgroup of $G$.
\end{proposition}

\begin{proof}
Suppose that Hypothesis \ref{general} holds and that \(G_{v}\) is a maximal \(\mathcal{C}_{6}\) subgroup with $G_v$ and $K$ as above.
If \(n\leqslant 4\), then all the possibilities for \(G_{v}\) are given in \cite[Table 8.1, 8.12]{holt} and in all these cases, a computation in MAGMA \cite{magma} (see Subsection \ref{r:magma}) shows that \(G_{v}\) has no homogeneous factorisation \(G_{v}=G_{uv}G_{vw}\) such that \(|G_{v}|/|G_{uv}|\geqslant3\). Thus we may assume that \(n\geqslant 8\), so $m\geq 4$. 

Note that \(G/L\leqslant\Out(L)\cong \C_{2}\), and that $G=LG_v$ (since $L$ is vertex-transitive) so $G_v/L_v\cong G/L$. Let \(\pi\) be the natural projection from \(G_{v}\) to \(G_{v}/L_{v}\). Then \(\pi(G_{uv})\pi(G_{vw})=\pi(G_{v})\leq\C_{2}\). In particular, at least one of \(\pi(G_{uv})\) or \(\pi(G_{vw})\) equals \(\pi(G_{v})\). Without loss of generality, assume that \(\pi(G_{uv})=\pi(G_{v})\). For any subgroup \(H\leqslant G_{v}\), we denote \(H\R(G_{v})/\R(G_{v})\) by \(\overline{H}\), where $R(G_v)$ is the soluble radical of $G_v$.

Note that \(\overline{G_{v}}\) is an almost simple group with socle \(\overline{K'} \cong \Omega_{2m}^{-}(2)\). Observe that \(\Pi(K)=\Pi(G_{v})\). This together with the facts that \(\Pi(G_{v})=\Pi(G_{uv})=\Pi(G_{vw})\) and 
\[
    \Pi(G_{uv})=\Pi(\overline{G_{uv}})\cup\{2\}=\Pi(\overline{G_{vw}})\cup\{2\}
    \]
    implies that \(\Pi(\overline{G_{uv}})\cup \{2\}=\Pi(\overline{G_{vw}})\cup\{2\}=\Pi(K)\). By \cite[Corollary 5]{transitive} it follows that \(\Omega_{2m}^{-}(2)\leqslant \overline {G_{uv}}\cap \overline{G_{vw}}\). 
    Since \(G_{uv}\cong G_{vw}\), and since \(G_{uv}/L_{uv}\) and \(G_{vw}/L_{vw}\) are both soluble, we have  \(\Omega_{2m}^{-}(2)\leqslant\overline{L_{uv}}\cap \overline{L_{vw}}\), in particular, 
\[
\Omega_{2m}^{-}(2)\leqslant L_{uv}\C_{2}^{2m}/\C_{2}^{2m}.
\]
Note that \(\Soc(K)\cong \Omega_{2m}^{-}(2)\) is irreducible on \(\C_{2}^{2m}\) and therefore \(G_{uv}\cap \C_{2}^{2m}=1\) or \(\C_{2}^{2m}\). 
Suppose that \(G_{uv}\cap\C_{2}^{2m}=\C_{2}^{2m}\). Then \(G_{uv}\) contains the subgroup \(X= \C_{2}^{2m}.\Omega_{2m}^{-}(2)\) which is a characteristic subgroup of \(G_{v}\). Since \(G_{uv}^{\overline{g}}=G_{vw}\), we have \(X^{\overline{g}}\leqslant G_{vw}\). However, since \(X\) is the unique subgroup of \(G_{v}\) isomorphic to \(\C_{2}^{2m}.\Omega_{2m}^{-}(2)\), we conclude that \(X^{\overline{g}}=X\), contradicting Lemma \ref{normal}.

Therefore \(G_{uv}\cap\C_{2}^{2m}=1\). In particular, \(\Soc(L_{uv})=\Omega_{2m}^{-}(q)\). Since \(L_{uv}L_{vw}\subseteq L_{v}\), it follows that, for any odd prime \(r\), 
\[
    \frac{|L_{uv}|_{r}\cdot |L_{vw}|_{r}}{|L_{uvw}|_{r}}\leqslant |L_{v}|_{r}=|\Omega_{2m}^{-}(2)|_{r}.
\]
Then since \(|L_{uv}|_{r}=|L_{vw}|_{r}=|\Omega_{2m}^{-}(2)|_{r}\), we deduce that \(|L_{uvw}|_{r}=|\Omega_{2m}^{-}(2)|_{r}\). Thus by \cite[Corollary 1.5]{transitive} applied to $L_{uv}$, we have \(\Omega_{2m}^{-}(2)\leqslant L_{uvw}\). This implies that \(\Soc(L_{uv})\leqslant L_{uvw}\) and so \(\Soc(L_{uvw})=\Soc(L_{uv})\). Thus
\[
\Soc(L_{uv})= \Soc(L_{uvw}) \leqslant L_{uvw}  \leqslant L_{vw} 
\]
and since $L_{vw}\cong L_{uv}$ is almost simple with socle $\Omega_{2m}^{-}(q)$, it follows that $\Soc(L_{vw})= \Soc(L_{uv})$, and hence 
 \(\Soc(L_{uv})^{\overline{g}}=\Soc(L_{uv})\), contradicting Lemma \ref{normal} and proving the result.
\end{proof}

\subsection{\(\mathcal{C}_{7}\)-subgroups}

Here \(G_{v}\) is a maximal \(\mathcal{C}_{7}\)-subgroup  of the almost simple group $G$ where \(\Soc(G)=\PSp_{2n}(q)\) with $q=p^f$ for a prime $p$, so by \cite[Proposition 4.7.4]{kleidman},  \(2n=(2m)^k\), $k>1$,  \(qk\) is odd,  \((m,q)\neq(1,3)\), and \(L_{v}=L\cap G_{v}\cong \PSp_{2m}(q)^k.2^{k-1}.\Sy_{k}\) of index $2$ in  \(K =\PGSp_{2m}(q)\wr S_{k}\).  

\begin{proposition}\label{p:c7}
Suppose that Hypothesis \ref{general} holds and that  \(G_{v}\) is a maximal \(\mathcal{C}_{7}\)-subgroup. Then $\G$ is not $(G,3)$-arc-transitive.
\end{proposition}

\begin{proof}
Suppose that Hypothesis \ref{general} holds, that \(G_{v}\) is a maximal \(\mathcal{C}_{7}\) subgroup with $G_v$ and $K$ as above, and that $\G$ is $(G,3)$-arc-transitive. Then by Lemma \ref{s-1}, \(\Gamma\) is \((L,2)\)-arc-transitive, and by Lemma \ref{0},  \(L_{v}=L_{uv}L_{vw}\) where \(L_{uv}^{\overline{g}}=L_{vw}\). 

Let \(M:=\prod_{i=1}^{k}M_{i}\) denote the base group of \(K\), and let \(\pi\) and \(\varphi_{i}\) denote the projections from \(K\) to \(K/M\), and from \(M\) to \(M_{i}\) for \(1\leqslant i\leqslant k\), respectively. Since \(|K:L_{v}|=2\) we have \(|\pi(K):\pi(L_{v})|\leq 2\)  and so \(\A_{k}\leqslant\pi(L_{v})=\pi(L_{uv})\pi(L_{vw})\) (recall that $k$ is odd). By \cite[Lemma 2.5]{linear} at least one of \(\pi(L_{uv})\) or \(\pi(L_{vw})\) is a transitive subgroup of \(\Sy_{k}\), and without loss of generality, we assume that \(\pi(L_{uv})\) is transitive. Then by \cite[Lemma 3.5]{linear},  \(\varphi_{1}(L_{uv}\cap M)\cong\cdots\cong\varphi_{k}(L_{uv}\cap M)\) and \(\Pi(M_{1})\subseteq\Pi(\varphi_{1}(L_{uv}\cap M))\). Let \(T_{i}:=\varphi_{i}(L_{uv}\cap M)\cap \Soc(M_{i})\), for \(i=1,\ldots, k\). Then since 
\[
\Pi(\varphi_{1}(L_{uv}\cap M))\subseteq\Pi(T_{1})\cup\{2\}
\] 
and \(\Pi(M_{1})=\Pi(\PSp_{2m}(q))\) we have  
\[
\Pi(T_{1})\cup\{2\}=\Pi(\PSp_{2m}(q)).
\] 
Therefore we deduce from \cite[Corollary 5, Table 10.3, Table 10.7]{transitive} that one of the following holds (recall that $q$ is odd).
\begin{description}
\item[(i)] \(m=1, q=2^{a}-1\) is a Mersenne prime such that \(a\geqslant3\), \(T_{1}\leqslant q:((q-1)/2)\);
\item[(ii)] \(m=1,q=9\) and \(T_{1}=\A_{5}\);
%\item[(ii)] \(m=1, q=2^{a}-1\) is a Mersenne prime such that \(a\geqslant3\), \(T_{1}\leqslant q:((q-1)/2)\);
\item[(iii)] \(m=2, q=7\) and \(T_{1}= \A_{7}\);
\item[(iv)] \(m=2, q=3\) and \(T_{1}= \Sy_5\) or \(2^4.\A_5\);
\item[(v)] \(m=2\) and \(\PSp_{2}(q^2)\triangleleft T_{1}\leqslant\PSp_{2}(q^{2}).2\).
\item[(vi)] \(\PSp_{2m}(q)\leqslant \varphi_{1}(L_{uv}\cap M)\)
\end{description}
Note that case (iv) arises from the isomorphism \(\PSp_{4}(3)\cong\PSU_{4}(2)\). 

\medskip

%Assume that (i) holds. Then $|\varphi_{1}(L_{uv}\cap M)|_{3}=|T_{1}|_{3}\cdot(2)_{3}=|\A_{5}|_{3}=3$,  and therefore \(|L_{uv}|_{3}=|\varphi_{1}(L_{uv}\cap M)|_{3}^{k}(k!)_{3}\leqslant3^{k}(k!)_{3}\). On the other hand, \(|L_{v}|_{3}=|\A_{6}|^{k}_{3}(k!)_{3}=3^{2k}(k!)_{3}\). Thus the valency of \(\Gamma\) has \(3\)-part 
%\[
 %   \frac{|L_{v}|_{3}}{|L_{uv}|_{3}}\geqslant\frac{3^{2k}(k!)_{3}}{3^{k}(k!)_{3}}=3^{k}.
%\]
%Since \(\Gamma\) is \((G,3)\)-arc-transitive, by Lemma \ref{valency}, \(3^{3k}\) divides \(|G_{v}|\). However, since \(q=9\), we have \[|G_{v}|_{3}=|L_{v}|_{3}|\Out(L)|_{3}=3^{2k}(k!)_{3}< 3^{2k+\frac{k}{2}}<3^{3k},\] which is a contradiction as \(k\geq3\).

%\medskip

Assume that (i) holds. Then \(|\varphi_{1}(L_{uv}\cap M)|_{2}\leqslant|T_{1}|_{2}\cdot2=|q(\frac{q-1}{2})|_{2}\cdot2=2\) and therefore \(|L_{uv}|_{2}^{2}\leqslant|\varphi_{1}(L_{uv}\cap M)|^{2k}_{2}(k!)_{2}^{2}\leqslant 2^{2k}(k!)_{2}^{2}\). On the other hand, since \(a\geqslant3\) we have \(|\PSp_{2}(q)|_{2}=2^{a}\geqslant 2^{3}\). Hence
\[
|L_{v}|_{2}=2^{-1}|K|_{2}=2^{-1}|\PSp_{2m}(q)|^{k}_{2}2^{k}(k!)_{2} = 2^{(a+1)k-1}(k!)_{2}\geq 2^{4k-1}(k!)_{2}.
\]
Note that \(|L_{v}|_{2}\leqslant|L_{uv}|_{2}^2\).  This implies that \((k!)_{2}\geqslant 2^{2k-1}>2^{k}\), which is a contradiction.

\medskip

Next assume that one of the cases (ii), (iii), (iv) or (v) holds. Now \(|L_{v}|_{p}=|\PSp_{2m}(q)|_{p}^{k}(k!)_{p}=(q^{m^{2}k})_{p}(k!)_{p}=p^{m^{2}fk}(k!)_{p}\), and a careful check shows that in each case \(|\varphi_{1}(L_{uv}\cap M)|_{p}\leqslant p^{m^{2}f/2}\). Therefore
\[
|L_{uv}|_{p}\leq  |\varphi_{1}(L_{uv}\cap M)|_{p}^{k}\cdot (k!)_{p}\leq p^{m^{2}fk/2}\cdot (k!)_{p}.
\]
Hence, the valency of \(\Gamma\) has \(p\)-part 
\[
\frac{|L_{v}|_{p}}{|L_{uv}|_{p}}\geqslant\frac{p^{m^{2}fk}(k!)_{p}}{p^{m^{2}fk/2}(k!)_{p}}=p^{m^{2}fk/2}.
\]
Since \(\Gamma\) is \((G,3)\)-arc-transitive, it follows from Lemma~\ref{valency} that \(|G_{v}|\) is divisible by \(p^{3m^{2}fk/2}\). On the other hand, \(|\Out(L)|_{p}=(f)_{p}\), and we note that   \(f\leqslant m^{2}fk/3\) since  \(k\geqslant3\). Thus
\[
p^{3m^{2}fk/2}\leqslant|G_{v}|_{p}=|L_{v}|_{p}\cdot |\Out(L)|_{p}\leqslant p^{m^{2}fk}(k!)_{p}(f)_{p}<p^{m^{2}fk}(k!)_{p}p^{f}<p^{4m^{2}fk/3}(k!)_{p},
\]
which implies that \((k!)_{p}>p^{2m^{2}fk/3}\geqslant p^{2k/3}\), and this is a contradiction to Lemma \ref{sizeppd}.
%Assume that (iii) holds. Then \(|\varphi_{1}(L_{uv}\cap M)|_{7}=|T_{1}|_{7}\cdot(2)_{7}=|\A_{7}|_{7}=7\) and therefore 
% \[
% |L_{uv}|_{7}=|\varphi_{1}(L_{uv}\cap M)|^{k}_{7}(k!)_{7}\leqslant7^{k}(k!)_{7}<7^{k+\frac{k}{6}}=7^{7k/6}.
% \] 
%\[
%|L_{uv}|_{7}=|\varphi_{1}(L_{uv}\cap M)|^{k}_{7}(k!)_{7}=7^{k}(k!)_{7}.
%\] 
%Hence 
%\[
%7^{2k}(k!)_{7}^2\geqslant |L_{uv}|_{7}^{2}\geq |L_{v}|_{7}=|\PSp_{4}(7)|_{7}^{k}(k!)_{7}=7^{4k}(k!)_{7}
%\] 
%which implies that $(k!)_{7}\geq 7^{2k}$,  a contradiction.

%\medskip

%Assume that (iv) holds. Then \(|\varphi_1(L_{uv} \cap M)|_3=|T_{1}|_{3}(2)_{3}=3\) and 
%\[
%|L_{uv}|_3=|\varphi_1(L_{uv} \cap M)|^{k}_3(k!)_3= 3^k(k!)_3<3^{k+\frac{k}{2}}.
%\] 
%This implies that 
%$  |L_{uv}|_3^{2}<3^{4k}<|\PSp_{4}(3)|_{3}^{k}(k!)_{3}=|L_{v}|_{3}$, which is a contradiction.

%\medskip
%{\color{red}Lei: I rewrote - putting in references, correcting a few mistakes. If you're happy with the above please remove the blue..}

\medskip
%Assume that (v) holds. Let us take \(r\in\ppd(p,2f)\) and denote by \(r^{a}:=(q^{2}-1)_{r}\). Then \(|\varphi_{1}(L_{uv}\cap M)|_{r}=|T_{1}|_{r}\cdot(2)_{r}=r^{a}\) and therefore \(|L_{uv}|_{r}=|\varphi_{1}(L_{uv}\cap M)|^{k}_{r}(k!)_{r}=r^{ak}(k!)_{r}\). On the other hand, \(|L_{v}|_{r}=|\PSp_{4}(q)|_{r}^{k}(k!)_{r}=r^{2ak}(k!)_{r}\). Thus the valency of \(\Gamma\) has \(r\)-part
%\begin{align}\nonumber
   % \frac{|L_{v}|_{r}}{|L_{uv}|_{r}}=\frac{r^{2ak}(k!)_{r}}{r^{ak}(k!)_{r}}=r^{ak}.
%\end{align}
%Since \(\Gamma\) is \((G,3)\)-arc-transitive, it follows from Lemma \ref{valency} that \(|G_{v}|\) is divisible by \(r^{3ak}\). However, \[|G_{v}|_{r}\leqslant|L_{v}|_{r}|\Out(L)|_{r}=r^{2ak}(k!)_{p}<r^{2ak+\frac{k}{r-1}}<r^{3ak},\] which is a contradiction.

%{\color{blue} Assume that (v) holds. Then \(|L_{v}|_{p}=|\PSp_{4}(q)|_{p}^{k}(k!)_{p}=p^{4fk}(k!)_{p}\). On the other hand, \(|\varphi_{1}(L_{uv}\cap M)|_{p}=p^{2f}\). This implies that 
%\[
%|L_{uv}|_{p}= |T_{1}|_{p}^{k}(k!)_{p}=p^{2fk}(k!)_{p}.
%\]
%Hence, the valency of \(\Gamma\) has \(p\)-part 
%\[
%\frac{|L_{v}|_{p}}{|L_{uv}|_{p}}\geqslant\frac{p^{4fk}(k!)_{p}}{p^{2fk}(k!)_{p}}=p^{2fk}.
%\]
%Since \(\Gamma\) is \((G,3)\)-arc-transitive, it follows that \(|G_{v}|\) is divisible by \(p^{6fk}\). On the other hand, \(|\Out(L)|_{p}=(f)_{p}\) and so 
%\[
%p^{6fk}\leqslant|G_{v}|_{p}=|L_{v}|_{p}|\Out(L)|_{p}\leqslant p^{4fk}(k!)_{p}(f)_{p}<p^{4fk}(k!)_{p}p^{f}<p^{5fk}(k!)_{p}
%\]
%as \(k\geqslant3\).
%This implies that \((k!)_{p}>p^{fk}\geqslant p^{k}\), which is a contradiction to Lemma \ref{sizeppd}.}
%\medskip

Thus case (vi) holds and \(\PSp_{2m}(q)\leqslant\varphi_{1}(L_{uv}\cap M)\). Let  \(\ell:=\m_{L_{uv}\cap M}(\PSp_{2m}(q))\). Then since \(\pi(L_{uv})\) is transitive, it follows from Scott's Lemma \cite[Theorem 4.16]{csaba} that \(\ell\) divides \(k\). Since \(k\) is odd either \(\ell\leqslant\frac{k}{3}\) or \(\ell=k\). 
Suppose that \(\ell\leqslant\frac{k}{3}\). Since the set of  composition factors of \(L_{uv}\cap M\) is equal to that of \(\varphi_{1}(L_{uv}\cap M)\) , and since the set of composition factors of \(\varphi_{1}(L_{uv}\cap M)\) is either \(\{\PSp_{2m}(q)\}\) or \(\{\PSp_{2m}(q), \C_{2}\}\), we deduce that 
\[
|L_{uv}|_{p}^{2}\leqslant |\PSp_{2m}(q)|_{p}^{2\ell}(k!)_{p}^{2}\leqslant p^{2m^{2}fk/3}(k!)_{p}^{2}.
\] 
Since \(|L_{uv}|_{p}^{2}\geqslant|L_{v}|_{p}\) and \(|L_{v}|_{p}=p^{m^{2}fk}(k!)_{p}\), it follows that 
\((k!)_{p} \geqslant p^{m^{2}fk/3}\). This together with the fact that \((k!)_{p}<p^{k/(p-1)}\) implies that \((p-1)m^{2}f<3\). However, this is a contradiction since \((m,p,f)\neq(1,3,1)\). Thus \(\ell=k\) and \(\Soc(M)=\PSp_{2m}(q)^{k}\leqslant L_{uv}\). Moreover, since \(\pi(L_{uv})\) is transitive, \(\Soc(M)\) is a minimal normal subgroup of \(L_{uv}\). As \(L_{uv}^{\overline{g}}=L_{vw}\), it follows that \(\Soc(M)^{\overline{g}}\) is a minimal normal subgroup of \(L_{vw}\) isomorphic \(\PSp_{2m}(q)^{k}\). Then since \(L_{vw}\cap M\) is normal in \(L_{vw}\), we conclude that either \((L_{vw}\cap M)\cap\Soc(M)^{\overline{g}}=M\cap \Soc(M)^{\overline{g}}\) is trivial, or \(\Soc(M)^{\overline{g}}\leqslant \Soc(M)^{\overline{g}}\cap M\). If \(M\cap\Soc(M)^{\overline{g}}=1\), then \(\Soc(M)^{\overline{g}}\cong\PSp_{2m}(q)^{k}\lesssim S_{k}\), which is impossible. Thus \(\PSp_{2m}(q)^{k}\cong\Soc(M)^{\overline{g}}\leqslant M\), but \(\Soc(M)\) is the unique subgroup in \(M\) isomorphic to \(\PSp_{2m}(q)^{k}\), so we have \(\Soc(M)^{\overline{g}}=\Soc(M)\), which contradicts Lemma \ref{normal} and proves the result.
\end{proof}

\subsection{\(\mathcal{C}_{8}\)-subgroups}

Here \(G_{v}\) is a maximal \(\mathcal{C}_{8}\)-subgroup  of the almost simple group $G$ where \(\Soc(G)=\PSp_{2n}(q)\) with $q=p^f$ for a prime $p$, so by \cite[Proposition 4.8.6 and Table 4.8.A]{kleidman}, $p=2$, and $\Soc(G_v)=\Omega_{2n}^\pm(q)$. We show that this case does not arise, using a result of Inglis.

% \begin{lemma}\label{Hxy}
% Suppose that \(q\) is even and \(\Sp_{2n}(q)\leqslant G\leqslant \Aut(\Sp_{2n}(q))\) and \(H\) is a \(\mathcal{C}_{8}\) subgroup of \(G\). Then every orbital for the action of \(G\) on the set of right cosets of \(H\) is self-paired.%for any \(x, y\in G\), there exists \(g\in G\) such that \(H_{x}^{g}=H_{y}\) and \(H_{y}^{g}=H_{x}\).
% \end{lemma}
% This leads to the following result.

\begin{proposition}\label{p:c8}
Suppose that Hypothesis \ref{general} holds. Then \(G_{v}\) is not a maximal \(\mathcal{C}_{8}\)-subgroup.
\end{proposition}

\begin{proof}
If \(G_{v}\) is a maximal \(\mathcal{C}_{8}\) subgroup, then $q$ is even and by \cite[Theorem 1]{inglis},  every orbital for the action of \(G\) on the set of right cosets of \(G_v\), and hence for the $G$-action on $\Omega$, is self-paired.
\end{proof}

\subsection{\(\mathcal{C}_{9}\)-subgroups}

Here \(G_{v}\) is a maximal \(\mathcal{C}_{9}\)-subgroup  of the almost simple group $G$ where \(\Soc(G)=\PSp_{2n}(q)\) with $q=p^f$ for a prime $p$, If $n\ne 2$ then \(G\leqslant\PGammaSp_{4}(q)\) and this family of subgroups is called $\mathcal{S}$ in \cite{kleidman}, and is defined in \cite[Definition, p. 3]{kleidman}. In particular $G_v$ is almost simple. Moreover if $n=2$, then it follows from \cite[Theorem 3.7, 3.8]{Wilson}  that either \((n,p)=(2,2)\) and \(G\nleqslant\PGammaSp_{4}(q)\), or again \(G_{v}\) is almost simple.

\begin{proposition}\label{p:c9}
    Suppose that Hypothesis \ref{general} holds and that \(G_{v}\) is a maximal \(\mathcal{C}_{9}\)-subgroup. Then $\Gamma$ is not $(G,3)$-arc-transitive. 
\end{proposition}

\begin{proof}
Suppose that Hypothesis \ref{general} holds and \(G_{v}\) is a \(\mathcal{C}_{9}\)-subgroup. If \((n,p)=(2,2)\) and \(G\nleqslant\PGammaSp_{4}(q)\), then by \cite[Table 8.14]{holt} we have \(L_{v}=G_{v}^{(\infty)}=\Sz(q)\) and \(q\geqslant 8\). Let \(R\) be the soluble radical of \(G_{v}\) and \(\overline{S}:=SR/R\) for any \(S\leqslant R\). Then by \cite[Lemma 3.3]{small} applied to $G$, the subgroups \(\overline{G_{uv}}\) and \(\overline{G_{vw}}\) are core-free in \(\overline{G_{v}}\), and \(\overline{G_{v}}=\overline{G_{uv}}\,\overline{G_{vw}}\) satisfies either \cite[Lemma 2.6]{small} or \cite[Lemma 2.7]{small}. However, there is no such core-free factorisation given in \cite[Lemma 2.6 and Lemma 2.7]{small}, which is a contradiction. 
Thus \(G_{v}\) is almost simple, and in this case it follows from \cite[Corollary 3.4]{linear} that $\Gamma$ is not $(G,3)$-arc-transitive.
\end{proof}

\subsection{\(\mathcal{A}_{1}\): \(\{U, W\}\) Decomposition Stabilisers}

By \cite[Section 14]{asch}, if \(G_{v}\) is a maximal subgroup of  the almost simple group $G$  with \(\Soc(G)=\PSp_{2n}(q)\) (where $q=p^f$ for a prime $p$), and if $G_v$ does not lie in  \(\mathcal{C}_{i}\) for any $i\leq 9$, then \(n=2\), \(q\) is even, $G$ contains a graph automorphism, and \(G_{v}\) lies in one of  three families \(\mathcal{A}_{1}, \mathcal{A}_{2}\),  \(\mathcal{A}_{3}\) of maximal novelty subgroups of \(G\), see also \cite[Section 7.2.2]{holt}. We treat these cases in this and the following two subsections. In this subsection we deal with $G_v$ in the family $\mathcal{A}_{1}$ which we describe in the next paragraph.

\medskip

Here \(V=( \oplus_{i=1}^{2}\langle e_{i}, f_{i}\rangle, B)\) is a symplectic space over \(\mathbb{F}_{q}\) with \(q\)  a power of \(2\), and \(B(e_{i}, e_{j})=B(f_{i}, f_{j})=0\) and \(B(e_{i}, f_{j})=\delta_{ij}\). A totally isotropic subspace \(V'\) is called a \emph{totally isotropic point} if \(\dim(V')=1\), and a \emph{totally isotropic line} if \(\dim (V')=2\). If \(U\), \(W\) are a totally isotropic point and line, respectively, such that \(U< W\), then the stabiliser \(H\) in the group \(\G=\Aut(\Sp(V))\) of the unordered pair \(\{U, W\}\) is maximal in $G$ and is called a maximal  $\mathcal{A}_{1}$-subgroup. 

\begin{proposition}\label{p:a1}
    Suppose that Hypothesis \ref{general} holds. and that \(\G=\Aut(\Sp(V))\), with $(n,p)=(2,2)$. Then \(G_{v}\) is not a maximal \(\mathcal{A}_{1}\)-subgroup of $G$.
\end{proposition}

Proposition\ref{p:a1} follows immediately from the following Lemma~\ref{a1}, which shows that every orbital of the $G$-action on the set $\Omega$ of such unordered pairs  \(\{U, W\}\) is self-paired, and hence this $G$-action does not even yield a $1$-arc-transitive digraph. % Hence \(\Cos(G,H,g)\) is a graph for all \(g\in G\).

\begin{lemma}\label{a1}
Let \(V, G, \Omega\) be as above, and let  \(\{U_{1}, W_{1}\}, \{U_{2}, W_{2}\}\in\Omega\).
% a symplectic space over a field of characteristic \(2\) of dimension \(4\), \( G=\Aut(\Sp(V))\), and \(\{U_{1}, W_{1}\}\) and \(\{U_{2}, W_{2}\}\) be two unordered pairs as defined above such that \(U_{i}\leqslant W_{i}\) are totally isotropic points and lines in \(V\). 
Then there exists an element \(g\in G\) such that \(\{U_{1}, W_{1}\}^{g}=\{U_{2}, W_{2}\}\) and \(\{U_{2}, W_{2}\}^{g}=\{U_{1}, W_{1}\}\).
\end{lemma}

In the proof we use the fact, see \cite[p. 360]{holt}, that $\Omega$ is the edge set of a bipartite distance transitive graph $\Gamma$ with vertex set $\Delta_1\cup\Delta_2$, where $\Delta_i$ is the set of totally isotropic $i$-subspaces of $V$, for $i=1,2$, The automorphism group $\Aut(\Gamma)$ is $G$, the diameter  of $\Gamma$ is $4$, and by \cite[Proposition 7.2.3]{holt}, each pair of vertices at distance $3$ in $\Gamma$ is joined by a unique path of length $3$.

\begin{proof}
We divide the analysis into three cases:

\medskip\noindent
\emph{Case $1$: \(W_{1}\cap W_{2}=0.\)}\quad It is easiest to treat this case algebraically. Since \(W_{1}\) and \(W_{2}\) are maximal totally isotropic subspaces of \(V\)  with trivial intersection, by Lemma \ref{exercise}(i) there exists \(h\in \Sp(V)\) such that \(W_{1}^{h}=\langle e_{1}, e_{2}\rangle\) and \(W_{2}^{h}=\langle f_{1}, f_{2}\rangle\). Thus we assume that \(W_{1}=\langle e_{1}, e_{2}\rangle\) and \(W_{2}=\langle f_{1}, f_{2}\rangle\). Moreover, since the stabiliser  $\Sp(V)_{W_1,W_2}$ induces $\GL_2(q)$ on $W_1$, we may further choose \(U_{1}=\langle e_{1}\rangle\). Then, since $\Sp(V)_{W_1,W_2, U_1}$ fixes $U_1^\perp\cap W_2=\langle f_2\rangle$, and is transitive on the $q$  other $1$-subspaces of $W_2$, we may assume that $U_2= \langle f_2\rangle$ or $U_2=\langle f_1\rangle$. In the former case,  the  element  \(g\in \GL(V)\) such that \(e_{i}^{g}=f_{3-i}, f_{i}^{g}=e_{3-i}\), for \(i=1,2\), preserves the form $B$ (since $q$ is even) and hence lies in $\Sp(V)$ and interchanges \(\{U_{1}, W_{1}\}\) and \(\{U_{2}, W_{2}\}\). Similarly in the latter case the  element  \(g\in \GL(V)\) such that \(e_{i}^{g}=f_{i}, f_{i}^{g}=e_{i}\), for \(i=1,2\), lies in $\Sp(V)$ and interchanges \(\{U_{1}, W_{1}\}\) and \(\{U_{2}, W_{2}\}\).

\medskip\noindent
\emph{Case $2$: \(\dim (W_{1}\cap W_{2})=1\).}\quad 
In this case we make use of properties of the distance transitive graph $\Gamma$ discussed above. Let $X:=W_1\cap W_2$. If $X\not\in\{U_1, U_2\}$, then in $\Gamma$ the distances $d(U_1, W_2)=3$ and $d(U_2,W_1)=3$ and so there exists $g\in G$ such that $(U_1,W_2)^g=(U_2,W_1)$. Further (see \cite[Proposition 7.2.3]{holt}), $g$ must map the unique $3$-path $(U_1, W_1, X, W_2)$ to the $3$-path $(U_2, W_2, X, W_1)$, and hence $g$ interchanges  \(\{U_{1}, W_{1}\}\) and \(\{U_{2}, W_{2}\}\). Next suppose that $X=U_i\ne U_{3-i}$ for some $i\in\{1,2\}$. Then in $\Gamma$, $d(W_i,U_{3-i})=3$
 and $P:=(W_i, U_i, W_{3-i}, U_{3-i})$ is the unique corresponding $3$-path. Thus there exists $g\in G$ such that $(W_i,U_{3-i})^g=(U_{3-i}, W_i)$, and $g$ must reverse the $3$-path $P$, and hence must interchange  \(\{U_{1}, W_{1}\}\) and \(\{U_{2}, W_{2}\}\). Finally suppose that $X=U_1=U_2$, so that $d(W_1,W_2)=2$. Here we argue algebraically: $W_1/X$ and $W_2/X$ are two $1$-subspaces of the nondegenerate $2$-space $X^\perp/X$, and as the group induced by $\Sp(V)_{X}$ on $X^\perp/X$ is $\Sp_2(q)$ which is $2$ transitive on $1$-spaces, it follows that $\Sp(V)_X$ contains an element $g$ which interchanges $W_1$ and $W_2$ and hence also  interchanges  \(\{U_{1}, W_{1}\}\) and \(\{U_{2}, W_{2}\}\).

\medskip\noindent 
\emph{Case $3$: \(W_{1}=W_{2}\)}\quad In this case, if \(U_{1}=U_{2}\), then the result holds with \(g=1\). So assume that \(U_{1}\neq U_{2}\). Then as \(\Sp(V)_{W_1}\) induces \(\GL(W_{1})\) on \(W_{1}\), which is $2$-transitive on the $1$-subspaces of $W_1$,  it follows that $\Sp(V)_{W_1}$ contains an element $g$ which interchanges $U_1$ and $U_2$ and hence also  interchanges  \(\{U_{1}, W_{1}\}\) and \(\{U_{2}, W_{2}\}\). 
\end{proof}

%{\color{red} If you're happy with the above please delete the text below up to the end of the section.}

\subsection{\(\mathcal{A}_{2}\): Normalisers of \(\C_{q\pm1}^{2}\)}

Here we treat the second family \(\mathcal{A}_{2}\) of maximal novelty subgroups of \(G\leqslant\Aut(\Sp(V))\).  These subgroups are the normalisers in $G$ of subgroups \(M=\C_{q-1}^{2}\) (with $q>4$) or  \(M=\C_{q+1}^{2}\) of $\Sp(V)$. As in \cite[Proposition 7.2.7]{holt}, we view $M$ as a subgroup of a $\mathcal{C}_2$-subgroup $\Sp_2(q)\wr S_2$ of $\Sp(V)$,  $N_G(M)$ is maximal  in $G$, and  \(N_{\Sp(V)}(M) = M: \D_8 \cong \C_{q\pm1}^{2}: \D_{8}\).

%{\color{red}
%Lei: I notice that you don't want $G=\Aut(\Sp(V))$ which is OK. but I don't think you know that $G$ contains $\Sp(V)\rtimes\langle \gamma\rangle $, you only have that $G$ contains `a graph automorphism', so an element $h\gamma$ for some $h\in\Gamma\Sp(V)$ with $\gamma$ your favourite graph aut of order $2$. - and $h\gamma$ might not have order $2$ so you might not get a semidirect product.

%can you please check what I wrote above. Is it OK? I was following your reference in \cite{holt}. I think the $M.D_8$ is just the normaliser in $\Sp(V)$ as all the field automorphisms will normalise $M$. Will these extra autos affect your argument? You will have $G_v=M.D_8.f$. You could perhaps just prove `not $(G,3)$-arc transitive' as assuming $(G,3)$-arc transitive you would get $(L,2)$-arc-transitive, and then you could use the group $L_v=M.D_8$. There is a similar problem in the next section.

% By \cite[Proposition 7.27]{holt} we find that the normalisers in \(G=\Aut(\Sp_{4}(2^{f}))\) of subgroups of type \(\C^{2}_{q\pm1}\) are maximal in \(G\). Let us denote by \(M=\C_{q\pm1}^{2}\). Then \(N_{G}(M)\cong \C_{q\pm1}^{2}: \D_{8}\). In particular, \(\D_{8}\) is generated by \(s\) and \(t\) such that \((x,y)^{s}=(y^{-1},x)\) and \((x,y)^{t}=(y,x)\) for any \((x,y)\in M\).
%}

\begin{proposition}\label{p:a2}
    Suppose that Hypothesis \ref{general} holds.  If $G_v$ is a maximal \(\mathcal{A}_{2}\)-subgroup, then \(\Gamma\) is not \((G,3)\)-arc-transitive.
\end{proposition}

 % and that \(G\) contains  an element \(h\gamma\) where \(h\in \Sp(V)\) and \(\gamma\) is a graph automorphism of \(\Sp(V)\cong \Sp_{4}(q)\) with \(q=2^{f}\).

 % Here $q$ is even {\color{blue} and \(G\) contains an element \(h\gamma\) where \(h\in \Sp(V)\)} and \(\gamma\) is a graph automorphism in \(\Aut(\Sp(V))\).

% \begin{proposition}\label{a2}
% Suppose that \(\Gamma\) is a \(G\)-vertex-primitive \((G,2)\)-arc-transitive digraph where \(G=\Aut(\Sp_{4}(q))\) and \(q\) is even. Then \(G_{v}\) is not an \(\mathcal{A}_{2}\) subgroup.
% \end{proposition}
\begin{proof}
Suppose for a contradiction that \(G_{v}\) is a maximal \(\mathcal{A}_{2}\)-subgroup and that \(s\geqslant3\). Then by Lemma \ref{s-1}, \(\Gamma\) is \((L,2)\)-arc-transitive and so \(L_{v}=L_{uv}L_{vw}\) with \(L_{uv}^{g}=L_{vw}\). Also $L_v=G_v\cap L =M:\D_8 \cong \C_{q\pm1}^{2}: \D_{8}$, and the complement $\D_8=\langle b,c\rangle$ such that \((x,y)^{b}=(y^{-1},x)\) and \((x,y)^{c}=(y,x)\) for all \((x,y)\in M\).  Let \(M = M_1\times M_2\), and let \(\varphi_{i}\) (for $i=1,2$) and \(\pi\) be the natural projections  \(\varphi_i:(M\to M_{i}\) and \(\pi:L_{v} \to L_{v}/M\). 

Since \(|M|\) is odd,  \(|\pi(L_{uv})|_{2}=|L_{uv}|_{2}=|L_{vw}|_{2}=|\pi(L_{vw})|_{2}\).
Moreover, \(\pi(L_{uv})\pi(L_{vw})=\pi(L_{v})=\D_{8}\), and it follows that \(|\pi(L_{uv})|, |\pi(L_{vw})|\geqslant 4\). This implies that \(b^{2}\in\pi(L_{uv}) \cap \pi(L_{vw})\), and we note that \((x,y)^{b^{2}}=(x^{-1},y^{-1})\) for all \((x,y)\in M\). Moreover, since \(L_{uv}\cap M\) and \(L_{vw}\cap M\) are generated by the odd-ordered elements of \(L_{uv}\) and \(L_{vw}\), respectively, it follows since  \(L_{uv}^{g}=L_{vw}\) that  \((L_{uv}\cap M)^{g}=L_{vw}\cap M\). Hence  
\[
\pi(L_{uv})\cong L_{uv}/(L_{uv}\cap M)\cong L_{uv}^{g}/(L_{uv}\cap M)^{g}=L_{vw}/(L_{vw}\cap M)\cong\pi(L_{vw}).
\]

Since \(\pi(L_{uv})\pi(L_{vw})=\pi(L_{v})=\D_{8}\), we deduce that either (i) \(\pi(L_{uv})=\pi(L_{vw})=\D_{8}\), or (ii) \(\pi(L_{uv})=\langle b^{2},d\rangle\) and \(\pi(L_{vw})=\langle b^{2}, d'\rangle\) where \(\{d,d'\}=\{c,bc\}\). In case (ii) exactly one of $\pi(L_{uv})$ and $\pi(L_{vw})$ contains $bc$, and hence exactly one of these groups interchanges $M_1$ and $M_2$. We may assume (in either case (i) or case (ii)) that $\pi(L_{uv})$ interchanges $M_1$ and $M_2$ (since the argument in case (ii) where  $\pi(L_{vw})$ interchanges $M_1$ and $M_2$ is similar). Thus we have \(\varphi_{1}(L_{uv}\cap M)\cong \varphi_{2}(L_{uv}\cap M)\). 

Let \(r\in\Pi(M)=\Pi(q\pm1)\) and denote by \(r^{a}:=(q\pm1)_{r}\). Then \(|M|_{r}=r^{2a}\) and \(|\Syl_{r}(M)|=1\) (as \(M\) is abelian). Let \(P=P_1\times P_2\cong C_{r^a}^2\) be the unique Sylow \(r\)-subgroup of \(M\), then since \(|L_{v}|_{r}=|M|_{r}\), it follows that \(P\in\Syl_{r}(L_{v})\).  Let \(R\in\Syl_{r}(L_{uv}\cap M)\), then \(\Syl_{r}(L_{uv}\cap M)=\{R\}\) (as \(L_{uv}\cap M\) is abelian). Consequently, \(\Syl_{r}(L_{uv}^{g})=\Syl_{r}(L_{vw})=\{R^{g}\}\).
Note that \(R\cap R^{g}\leqslant L_{uv}\cap L_{vw}\). Thus \(|R\cap R^{g}|_{r}\leqslant|L_{uv}\cap L_{vw}|_{r} \) , which implies that 
\[
   |P|_{r}=|L_{v}|_{r}=\frac{|L_{uv}|_{r}\cdot|L_{vw}|_{r}}{|L_{uv}\cap L_{vw}|_{r}}\leqslant\frac{|R|_{r}\cdot|R^{g}|_{r}}{|R\cap R^{g}|_{r}}=|RR^{g}|.
\]
Since \(RR^{g}\leqslant P\), it follows that \(P=RR^{g}\).

\medskip
\noindent
\textit{Claim:}
\(|R\cap M_{i}|_{r}=1\) for \(i=1,2\), and \(R=\{(x, x^{\alpha})\mid x\in P_1\}\) for some automorphism $\alpha$ of $P_1=\C_{r^{a}}$.

\medskip
\noindent
Note that  \(R\cap M_{1}\cong R\cap M_{2}\) as \(\pi(L_{uv})\) acts transitively on \(\{M_{1}, M_{2}\}\) and \(R\) is normal in \(L_{uv}\). Suppose for a contradiction that \(|R\cap M_{i}|_{r}>1\) for \(i=1,2\). Consequently there exists \(R_{0}\leqslant R\) such that 
\[
R_{0}=(R_{0}\cap M_{1})\times (R_{0}\cap M_{2})
\] 
and \(R_{0}\cap M_{i}\cong \C_{r}\) for \(i=1,2\). In particular, \(R_{0}\cong \C_{r}^{2}\). However, since there is a unique subgroup in \(L_{v}\) isomorphic to \(\C_{r}^{2}\), we have \(R_{0}^{g}\leqslant L_{uv}^{g}=L_{vw}\), and this is a contradiction to Lemma \ref{normal}. Thus \(|R\cap M_{i}|_{r}=1\) for each $i$. This implies that $R\cong \varphi_i(R)\leqslant P_i=C_{r^a}$ for each $i$. 
On the other hand
\[
|R|_{r}=|L_{uv}|_{r}\geqslant|L_{v}|_{r}^{1/2}=r^{a}
\] 
and hence \(|R|=r^{a}\), and it follows that $R$ is a diagonal subgroup of $P=P_1\times P_2$. Thus $R$ is as claimed, for some automorphism $\alpha$ of $C_{r^a}$, and
the claim is proved.

\medskip
% Since \(|R|=|R\cap M_{1}|\cdot|\varphi_{2}(R)|\), we have by the Claim that \(|R|_{r}=|\varphi_{2}(R)|_{r}\leqslant r^{a}\). This together with the fact that \[|R|_{r}=|L_{uv}|_{r}\geqslant|L_{v}|_{r}^{1/2}=r^{a}\] implies that \(|R|=r^{a}\). Therefore \(|\varphi_{i}(R)|_{r}=|\varphi_{i}(L_{uv}\cap M)|_{r}=r^a\), for \(i=1,2\), and there exists \(\alpha\in\Aut(\C_{q+1})\) such that \(R=\{(x, x^{\alpha})\mid x\in \C_{r^{a}}\}\).

Since \(R\) is the unique Sylow \(r\)-subgroup of \(L_{uv}\), it follows that  \(R\triangleleft L_{uv}\), and since $\pi(L_{uv})$ contains $bc$, the group  \(L_{uv}\) contains an element \(mbc\) for some \(m\in M\). In particular, \(R^{mbc}=R\) and since \(M\) is abelian, 
\[
    (x, x^{\alpha})^{mbc}=    (x, x^{\alpha})^{bc}= ((x^{\alpha})^{-1}, x)^{c}=(x,(x^{\alpha})^{-1})\in R
\] 
for all \((x, x^{\alpha})\in R\). Thus,  for all \(x\in R\), \(x^{\alpha}=(x^{\alpha})^{-1}=(x^{-1})^{\alpha}\), and hence $x=x^{-1}$, which is a contradiction as \(r\) is odd.
We conclude that \(\Gamma\) is not \((G,3)\)-arc-transitive.
\end{proof}

\subsection{\(\mathcal{A}_{3}\): Normalisers of \(\C_{q^{2}+1}\)}

Finally we treat the third family \(\mathcal{A}_{3}\) of maximal novelty subgroups of \(G\leqslant \Aut(\Sp(V))\). These subgroups are the normalisers in \(G\) of subgroups \(M=\C_{q^{2}+1}\) of \(\Sp(V)\). It is shown in \cite[Proposition 7.2.7]{holt} that \(N_{G}(M)\) is a maximal subgroup of \(G\) and \(N_{\Sp(V)}(M)=M:4=\C_{q^{2}+1}:4\).

\begin{proposition}\label{a3}
Suppose that Hypothesis \ref{general} holds. If $G_v$ is a maximal \(\mathcal{A}_{3}\)-subgroup, then \(\Gamma\) is not \((G,3)\)-arc-transitive. 
 
\end{proposition}
\begin{proof}
Suppose for a contradiction that \(G_{v}\) is a maximal \(\mathcal{A}_{3}\) subgroup, and that \(s\geqslant3\). Then by Lemma \ref{s-1}, \(\Gamma\) is \((L,2)\)-arc-transitive.
On the other hand \(L_{v}=\C_{q^{2}+1}:4\) with $\gcd(q^2+1, 4)=1$ (since \(q\) is even), and we have a contradiction by Lemma \ref{coprime ab}. 
\end{proof}

\section{Proof of Theorem \ref{maintheorem}}

We now prove Theorem \ref{maintheorem}. Suppose that $G, \Gamma, s$ satisfy the hypotheses of Theorem~\ref{maintheorem}, and that $s\geqslant2$. 
Since \(G\) acts primitively on the vertex set of \(\Gamma\), it follows that \(G_{v}\) is a maximal subgroup of \(G\). Then by  \cite[Section 14]{asch}, \(G_{v}\) is either a maximal \(\mathcal{C}_{i}\)-subgroup  of \(G\) for some $i\leqslant 9$,   or \(n=2\), \(q\) is even, and \(G_{v}\) is an \(\mathcal{A}_{1}, \mathcal{A}_{2}\), or \(\mathcal{A}_{3}\) maximal novelty subgroup of \(G\). These cases are treated separately in Propositions \ref{c1nonde}, \ref{p:notc2I}--\ref{p:c3II}, \ref{p:c4-5}, \ref{p:c5}--\ref{p:a1} and \ref{p:a2}--\ref{a3}, and together show that \(s=2\). Therefore Theorem \ref{maintheorem} is proved.

%\subsection{Future Prospects}
%The methods employed in this paper were a significant development of those used to analyse the analogous problem for linear groups with socle \(\PSL_{n}(q)\) in \cite{linear}. The new results or ones similar to them, especially those in Section~\ref{s:homfact} on homogeneous factorisations of wreath products, should find application for other classical groups. In particular, the analysis for unitary and orthogonal groups should follow a similar pattern as that for the symplectic groups \(\PSp_{2n}(q)\) in this paper, although additional new ideas will be needed.  Indeed, the first author \cite{LChenThesis} has thoroughly analysed \(\mathcal{C}_{1}\), \(\mathcal{C}_{3}\), and \(\mathcal{C}_{4}\) maximal subgroups for unitary groups \(\PSU_{m}(q)\) in his PhD thesis  \cite{LChenThesis}, and the authors and Xia are in the progress of completing this analysis for the other kinds of maximal subgroups in forthcoming work. 

\end{document}